\documentclass{amsart}

\usepackage{ae,aecompl}
\usepackage{chngcntr}
\usepackage{graphicx}
\usepackage[all,cmtip]{xy}
\usepackage{amsmath, amscd}
\usepackage{amsthm}
\usepackage{amssymb}
\usepackage{amsfonts}
\usepackage{qsymbols}
\usepackage{latexsym}
\usepackage{mathrsfs}
\usepackage{cite}
\usepackage{color}
\usepackage{url}
\usepackage{enumerate}
\usepackage[utf8]{inputenc}
\usepackage[T1]{fontenc}

\usepackage{verbatim}

\usepackage[colorlinks, citecolor = blue]{hyperref}

\usepackage[colorinlistoftodos,prependcaption,textsize=small]{todonotes}

\allowdisplaybreaks

\newcommand {\abs}[1]{\lvert#1\rvert}

\newcommand {\C}{{\mathbb C}}

\newcommand {\ud}{\mathrm{d}}

\newcommand {\Ellr}{L^{r}}

\newcommand {\F}{{\mathcal{F}}}

\newcommand {\ind}{{\mathbf{1}}}

\newcommand{\lb}{\langle}
\newcommand{\rb}{\rangle}

\newcommand {\calL}{{\mathcal{L}}}

\newcommand {\N}{{{\mathbb N}}}

\newcommand{\one}{{{\bf 1}}}

\newcommand {\R}{{\mathbb R}}

\newcommand {\Schw}{\mathcal{S}}
\newcommand {\Sch}{\mathcal{S}}
\newcommand {\T}{{\mathbb T}}

\newcommand {\Z}{{{\mathbb Z}}}

\newcommand {\vanish}[1]{\relax}

\newcommand{\wh}{\widehat}

\newtheorem{theorem}{Theorem}[section]
\newtheorem{lemma}[theorem]{Lemma}
\newtheorem{proposition}[theorem]{Proposition}
\newtheorem{corollary}[theorem]{Corollary}

\theoremstyle{definition}
\newtheorem{definition}[theorem]{Definition}
\newtheorem{remark}[theorem]{Remark}
\newtheorem{example}[theorem]{Example}

\renewcommand{\asymp}{\eqsim}

\numberwithin{equation}{section}
\protected\def\ignorethis#1\endignorethis{}
\let\endignorethis\relax

\title[Extensions of the vector-valued Hausdorff-Young inequalities]
{Extensions of the vector-valued Hausdorff-Young inequalities}

\author{Oscar Dominguez}
\address{Departamento de An\'alisis Matem\'atico, Facultad de Matem\'aticas\\ Universidad Complutense de Madrid\\ Plaza de Ciencias 3\\ 28040 Madrid\\ Spain}
\email{oscar.dominguez@ucm.es}

\author{Mark Veraar}
\address{Delft Institute of Applied
Mathematics\\
Delft University of Technology\\
P.O.~Box 5031\\
2628 CD Delft\\
The Netherlands}
\email{M.C.Veraar@tudelft.nl}



\thanks{The first author is supported in part by MTM2017-84058-P (AEI/FEDER, UE).
The second author is supported by the VIDI subsidy 639.032.427 of the Netherlands Organisation for Scientific Research (NWO). Part of the work was done during the visit of the first author to the Isaac Newton Institute for Mathematical Sciences, Cambridge, EPSCR Grant no EP/K032208/1.}

\begin{document}
\begin{abstract}
In this paper we study the vector-valued analogues of several inequalities for the Fourier transform. In particular, we consider the inequalities of Hausdorff--Young, Hardy--Littlewood, Paley, Pitt, Bochkarev and Zygmund. The Pitt inequalities include the Hausdorff--Young and Hardy--Littlewood inequalities and state that the Fourier transform is bounded from $L^p(\mathbb{R}^d,|\cdot|^{\beta p})$ into $L^q(\mathbb{R}^d,|\cdot|^{-\gamma q})$ under certain condition on $p,q,\beta$ and $\gamma$. Vector-valued analogues are derived under geometric conditions on the underlying Banach space such as Fourier type and related geometric properties. Similar results are derived for $\mathbb{T}^d$ and $\mathbb{Z}^d$ by a transference argument. We prove sharpness of our results by providing elementary examples on $\ell^p$-spaces. Moreover, connections with Rademacher (co)type are discussed as well.
\end{abstract}

\maketitle

\setcounter{tocdepth}{1}
\tableofcontents

\section{Introduction}\label{sec:introduction}

Weighted norm inequalities for Fourier transforms are of central interest in harmonic analysis. In particular, working with power weights, the classical Pitt's inequality \cite{Pitt37} (see also \cite{Stein56}),
\begin{equation}\label{PittScalarintroseries}
\|(\widehat{f}(n))\|_{\ell^q(\Z^d, (|n| + 1)^{-\gamma q})} \leq C \|f\|_{L^p(\T^d, |\cdot|^{\beta p})}, \quad 1 < p \leq q < \infty,
\end{equation}
holds if and only if
\begin{equation}\label{PittAssumpScalarintro}
\max\left\{0, d \left(\tfrac{1}{p} + \tfrac{1}{q} - 1\right) \right\} \leq \gamma < \tfrac{d}{q} \quad \text{and} \quad \beta - \gamma = d \Big(1- \tfrac{1}{p} - \tfrac{1}{q}\Big).
\end{equation}
Here $C$ denotes a positive constant which depends on $p,q,\beta,\gamma,d$ and $f(x) = \sum_{n\in \Z^d}\widehat{f}(n) e^{ 2\pi i n\cdot x}$.  One of the beauties of Pitt's inequalities is that they contain Plancherel's theorem ($p=q=2, \,\beta = \gamma = 0$), the Hausdorff--Young inequalities ($q=p' \geq 2, \, \beta=\gamma=0$) and the Hardy--Littlewood inequalities ($p=q \leq 2, \,\beta=0$, or $p=q \geq 2, \, \gamma=0$).
Similarly, the Pitt's inequality for the Fourier transform
\begin{equation*}
	\widehat{f}(\xi) = \F(\xi)= \int_{\R^d} f(t) e^{-2 \pi i t \cdot \xi} \, \ud t, \quad \xi \in \R^d,
\end{equation*}
that is,
\begin{equation}\label{PittScalarintro}
\|\widehat{f}\|_{L^q(\R^d, |\cdot|^{-\gamma q})} \leq C \|f\|_{L^p(\R^d, |\cdot|^{\beta p})}, \quad 1 < p \leq q < \infty,
\end{equation}
 is valid if and only if the conditions given in (\ref{PittAssumpScalarintro}) hold.

Extensions of Pitt's inequalities have been studied in many papers. Below we mention only some of them. A version for more general orthonormal sequences has been obtained in \cite{Stein56}.
Versions of Pitt's inequality with more general weights and generalized Fourier transforms have for instance been studied in \cite{Benedetto-Heinig03, BHJ, DeCarliGorbachevTikhonov, Heinig84, LifTik,StroWhe, GorbachevIvanovTikhonov}. There is special interest in the sharp constants $C$ in \eqref{PittScalarintroseries} and \eqref{PittScalarintro} in mathematical physics (see e.g.\ \cite{Beckner75, Beck95,Beckner08,Beck12,Eil,Herbst,Yaf} and references therein). In particular, there is a connection with the uncertainty principle and with Hardy-Rellich estimates which are used in the theory of elliptic PDEs.

In this paper we will study vector-valued analogues of Pitt's inequality. The special case $\beta=\gamma=0$ and $q=p'\geq 2$ is extensively studied in the literature. It turns out that even in this special case \eqref{PittScalarintro} does not hold if $f$ takes values in an arbitrary Banach space $X$. In \cite{Pee69} Peetre states that a Banach space $X$ has Fourier type $p \in [1,2]$ if \eqref{PittScalarintro} holds for $\beta=\gamma=0$ and $q=p'$. It turns out that this notion becomes more restrictive if $p$ is larger, and if $p=2$ this property characterizes Hilbert spaces (see \cite{Kwa}). Fourier type has been used in many parts of analysis. For instance in interpolation theory \cite{Pee69, CwikelSagher, SuWe}, Fourier multiplier theory \cite{BuKim,Girardi-Weis03,Hyt04,Rozendaal-Veraar17a}, evolution equations \cite{ArBaHiNe, ChTom, Rozendaal-Veraar17c, WeWro}, and geometry of Banach spaces \cite{HyNeVeWe16, HyNeVeWe2, PietschHis, Pietsch-Wenzel98}.

Dealing with the case $p=q \in (1, 2]$, Parcet, Soria and Xu obtained in \cite{ParcetSoriaXu} the following vector-valued Fourier inequality. Assume that $X$ is of Fourier type $p$. Then
 \begin{equation}\label{PittPSX}
\|(\widehat{f}(n))\|_{\ell^p(\Z, (|n| + 1)^{-(\alpha/p - \alpha + 1) p}; X)} \leq C \|f\|_{L^p(\T;X)}, \quad 1 < \alpha < 2.
\end{equation}
In particular, this inequality becomes useful in \cite{ParcetSoriaXu} to show the vector-valued analogue of the `little Carleson theorem', which provides a weak version of a question raised by Rubio de Francia \cite{RubioDeFrancia} on pointwise convergence of Fourier series taking values in UMD Banach spaces. It was shown by Hyt\"onen and Lacey \cite[Proposition 2.1]{HytonenLacey} that UMD is a necessary condition for the convergence problem. There, they also give a positive answer to the conjecture of Rubio de Francia for the class of \emph{intermediate UMD spaces} (see \cite[Theorem 1.1]{HytonenLacey}), but the general UMD case is still unsolved at the time of this writing. Another example of Pitt's inequality with $1 < p \leq q=2$ was derived by Rozendaal and the second author in \cite{Rozendaal-Veraar17a} in order to study sharpness in vector-valued multiplier theorems.

However, apart from the above mentioned special cases, the complete characterization of vector-valued Pitt's inequalities still remains open. Therefore, the first goal of the present paper is to study these inequalities in further detail to fill the gap in the existing literature. Namely, we establish the following
\begin{theorem}\label{thm:Pitt2intro}
Let $X$ be of Fourier type $p_0 \in (1, 2]$. Let $1 < p \leq q < \infty$ and $\beta, \gamma \geq 0$. Assume that
\begin{equation}\label{thm:Pitt2intro1}
\max\left\{0, d \left(\tfrac{1}{\min\{p, p_0\}} + \tfrac{1}{q} - 1\right) \right\} < \gamma < \tfrac{d}{q} \quad \text{and} \quad \beta - \gamma = d \Big(1- \tfrac{1}{p} - \tfrac{1}{q}\Big).
\end{equation}
Then, we have
\begin{equation}\label{thm:Pitt2intro2}
\|\widehat{f}\|_{L^q(\R^d, |\cdot|^{-\gamma q};X)} \leq C \|f\|_{L^p(\R^d, |\cdot|^{\beta p};X)}.
\end{equation}
Furthermore, if $\gamma = \max\left\{0, d \left(\tfrac{1}{\min\{p, p_0\}} + \tfrac{1}{q} - 1\right) \right\}$ then \eqref{thm:Pitt2intro2} holds true if one of the following conditions is satisfied:
\begin{enumerate}[\upshape(i)]
\item \label{thm:Pitt2intro3} $p = q \in (1, p_0) \cup (p_0',\infty), \quad p_0 \neq 2$;
\item \label{thm:Pitt2intro3*} $p=q \in (1, \infty), \quad p_0=2$;
\item \label{thm:Pitt2intro4} $p < q, \quad p \in (1, p_0] \cup [p_0',\infty)$;
\item \label{thm:Pitt2intro5}$p < q, \quad p \in (p_0, p_0'), \quad p_0' \leq q$.
\end{enumerate}
\end{theorem}
If none of the conditions \eqref{thm:Pitt2intro3}-\eqref{thm:Pitt2intro5} holds, then in general the endpoint case $\gamma = \max\left\{0, d \left(\tfrac{1}{\min\{p, p_0\}} + \tfrac{1}{q} - 1\right) \right\}$ in \eqref{thm:Pitt2intro2} fails to be true. This will be shown by giving elementary counterexamples in $\ell^r$-spaces.

Some comments are in order. The conditions \eqref{thm:Pitt2intro1} show the prominent role played by the geometry of $X$ (i.e., $p_0$) in Pitt's inequalities. This is illustrated by the fact that $\max\left\{0, d \left(\tfrac{1}{\min\{p, p_0\}} + \tfrac{1}{q} - 1\right) \right\} \to \frac{d}{q} -$ as $p_0 \to 1+$, and so the range of $\gamma$ for which \eqref{thm:Pitt2intro2} holds becomes more restrictive. The best possible scenario occurs if $X= \C$ (or more generally, if $X$ is a Hilbert space) where $\max\left\{0, d \left(\tfrac{1}{\min\{p, 2\}} + \tfrac{1}{q} - 1\right) \right\} = \max\left\{0, d \left(\tfrac{1}{p} + \tfrac{1}{q} - 1\right) \right\}$ and hence, Theorem \ref{thm:Pitt2intro} covers classical Pitt's inequalities \eqref{PittScalarintro} in the full range of parameters \eqref{PittAssumpScalarintro}. However, in a general Banach space $X$, the endpoint Pitt's inequality \eqref{thm:Pitt2intro2} with $\gamma = \max\left\{0, d \left(\tfrac{1}{\min\{p, p_0\}} + \tfrac{1}{q} - 1\right) \right\}$ may be false. In fact, its validity depends on the relations between the integrability parameters $p$ and $q$ and the geometric conditions on $X$ (i.e., $p_0$). More precisely, if $p \in (1, p_0) \cup (p_0',\infty)$ then the endpoint Pitt's inequality holds true for all $q \geq p$. However, this is not the case in the range $p \in (p_0, p_0')$ where the interrelations between $q$ and $p_0$ are now important. Namely, if $q \geq p_0'$ then the limiting cases of Pitt's inequalities are fulfilled, but this is not true in the remainder cases (i.e., $q \in (p, p_0')$). Such an interesting phenomenon is new and can not observed in the classical setting.

Given the Hilbert space characterization mentioned before \cite{Kwa}, it is rather surprising that in the above result, outside of the Hilbert space setting ($p_0 < 2$), the case $p=q=2$ is allowed as long as one considers $d (\tfrac{1}{p_0} -\tfrac{1}{2}) < \beta = \gamma < \tfrac{d}{2}$.

Theorem \ref{thm:Pitt2intro} is stated for Banach spaces with nontrivial Fourier type. It is well known that many Banach spaces satisfy this assumption. For example the following conditions are equivalent
\begin{enumerate}[(a)]
\item\label{eq:Ftype} $X$ has nontrivial Fourier type;
\item $X$ is $B$-convex;
\item\label{eq:Kconvex} $X$ is $K$-convex;
\item\label{eq:Rtype} $X$ has nontrivial (Rademacher) type;
\item\label{eq:Wtype} $X$ has nontrivial Walsh type.
\end{enumerate}
Moreover, $X$ has nontrivial Fourier type if one of following conditions holds:
\begin{enumerate}[(a)]
\setcounter{enumi}{6}
\item $X$ is UMD;
\item $X$ is uniformly convex;
\item $X$ is superreflexive.
\end{enumerate}
The deep results \eqref{eq:Rtype} $\Rightarrow$ \eqref{eq:Ftype} and \eqref{eq:Kconvex} are due to Bourgain \cite{Bourg88} and Pisier \cite{Pis82}, respectively. For details on these results and on the remaining implications we refer to the monographs \cite{HyNeVeWe16,HyNeVeWe2,Pietsch-Wenzel98,Pis16}.

The counterparts of Theorem \ref{thm:Pitt2intro} for the Fourier transform on the groups $\Z^d$ and $\T^d$ will be deduced by transference (see Corollary \ref{cor:Pitt2Periodic}). In particular, if $p = q = p_0$ and $1 < \alpha < 2$, then we obtain the following improvement of \eqref{PittPSX},
\begin{equation}\label{PittPSX2}
\|(\widehat{f}(n))\|_{L^p(\Z, (|n| + 1)^{-( \alpha/p - \alpha + 1) p};X)} \leq C \|f\|_{L^p(\T, |\cdot|^{(2 - \alpha) p/p'};X)}.
\end{equation}
Furthermore, \eqref{PittPSX2} fails to be true in the limiting cases $\alpha =1, 2$.

In order to prove Theorem \ref{thm:Pitt2intro} one cannot directly extend the interpolation method of \cite{Stein56} unless $p_0=2$. Instead we will use the following techniques. Firstly the extension of the Hardy--Littlewood--Sobolev inequalities due to Stein--Weiss plays an important role to obtain the sharp cases of Pitt's inequality.
Secondly the connection between Fourier type and other forms of (co)type such as Paley and Hardy--Littlewood (co)type is heavily used. Paley (co)type for arbitrary orthonormal sequences was introduced in \cite{GaKaKo01}. The classical inequalities due to Hardy and Littlewood, namely,
\begin{equation*}
\|\widehat{f}\|_{L^{p}(\R^d, |\cdot|^{-d(2-p)})} \leq C \|f\|_{L^p(\R^d)}, \quad 1 < p \leq 2,
\end{equation*}
lead us to introduce the Hardy-Littlewood type of Banach spaces (see Definition \ref{def:type}) and investigate its properties. This concept can be regarded as an extension to Lebesgue spaces of the notion of Hardy type due to Blasco and Pelczynski \cite{BlascoPelczynski}, which characterizes the validity of the vector-valued extension of the classical Hardy inequality
 \begin{equation*}
\|(\widehat{f}(n))\|_{\ell^{1}(\Z, (|n|+1)^{-1})} \leq C \|f\|_{H^1(\T)}.
\end{equation*}
Here, $H^1(\T)$ is the atomic Hardy space of periodic functions.

 We hope that Theorem \ref{thm:Pitt2intro}, besides its intrinsic interest, will find further applications in any of the above topics related to analysis in Banach spaces.

Our second goal is to study Hausdorff-Young type inequalities for Banach spaces. The classical Zygmund inequality \cite[p. 158]{Zygmund} claims that
\begin{equation}\label{ZYGIntro}
 \sum_{n=1}^\infty (1 + \log n)^b \wh{f}^\ast(n)\frac{1}{n} \leq C \int_{\T} |f(t)| (\log^+|f(t)|)^{b+1} \, dt, \quad b + 1>0.
\end{equation}
Here, $\log^+(t) = \log(2+t)$ and $(\wh{f}^\ast(n))$ denotes the non-increasing rearrangement of the sequence $(\wh{f}(n))$. In Theorem \ref{ThmZygmund} we show that the validity of the $X$-valued version of the Zygmund inequality \eqref{ZYGIntro} characterizes the nontrivial Fourier type of $X$. To be more precise, we provide a more general result which works with the vector-valued extensions of \eqref{ZYGIntro} to integrable functions in Lorentz-Zygmund spaces, as well as their discrete counterparts. To illustrate our approach, we would like to single out the particular case of Orlicz norms, which is of special interest. Namely, the following statements are equivalent:
	\begin{enumerate}[\upshape(i)]
	\item\label{Kol} $X$ has nontrivial Fourier type,
	\item\label{Kol1} the Zygmund inequality holds
	\begin{equation*}
 \sum_{n=1}^\infty (1 + \log n)^b \wh{f}^\ast(n)\frac{1}{n} \leq C \int_{\T^d} \|f(t)\|_X (\log^+\|f(t)\|_X)^{b+1} \, \ud t, \quad b + 1>0,
\end{equation*}
	\item\label{Kol2} the end-point Zygmund inequality holds
	\begin{equation*}
		\sum_{n=1}^\infty \wh{f}^\ast(n) \frac{1}{n (1+\log n)} \leq C \int_{\T^d} \|f(t)\|_X (\log^+ \log^+ \|f(t)\|_X) \, \ud t,
	\end{equation*}
	\item\label{Kol3} the Zygmund inequality for sequences holds
	\begin{equation*}
		\int_0^1 (1 + |\log t|)^b f^\ast(t) \frac{dt}{t} \leq C  \sum_{n \in \Z^d} \|c_n\|_X \left(\log^{+} \left(\frac{1}{\|c_n\|_X}\right) \right)^{b+1}, \quad b+1 < 0,
	\end{equation*}
	where $f(t) = \sum_{n \in \Z^d} c_n e^{2 \pi i n \cdot t}, \, t \in \T^d$.
	\end{enumerate}

Note that the inequality given in \eqref{Kol2} involves the vector-valued counterpart of $L(\log \log L)(\T^d)$. This space plays a distinguished role in the study of  convergence problems of Fourier series in part due to the Kolmogorov's fundamental example \cite{Kolmogorov23} of a function in $L(\log \log L)(\T^d)$ with almost everywhere divergent Fourier series.

As mentioned, the equivalence between \eqref{Kol}--\eqref{Kol3} is a particular case of Theorem \ref{ThmZygmund}. Its proof is given in Section \ref{ProofThmZygmund} and relies on limiting interpolation techniques (see Section \ref{SectionLimInt}) together with the Pisier's result that the type of a Banach space can be characterized as the fact that it contains almost isometric copies of finite dimensional spaces (see Lemma \ref{lem:Pisiertype}).

As an application of our previous results, we are able to obtain quantitative versions of vector-valued Zygmund-type inequalities. In particular, Parcet, Soria and Xu \cite[Theorem 2]{ParcetSoriaXu} showed that if $X$ has nontrivial Fourier type and
\begin{equation*}
 	\int_\T \|f(t)\|_X (\log^+ \|f(t)\|_X)^b \, dt \leq \rho
 \end{equation*}
for some $b > 0$, then there exist positive constants $a(\rho, b)$ and $A(\rho, b)$ such that
\begin{equation}\label{ParcetSoriaXu+}
	\sum_{n=-\infty}^\infty \text{exp} \Big(-a(\rho,b) \|\wh{f}(n)\|_X^{-\frac{1}{b}} \Big) \leq A(\rho,b).
\end{equation}
This result is not only interesting itself but also plays a crucial role in the proof of the `little Carleson theorem' for UMD Banach spaces (see \cite[Theorem 1]{ParcetSoriaXu}). Here, we extend \eqref{ParcetSoriaXu+} in several directions. Namely, we obtain its counterparts for the broader scale of periodic Lorentz-Zygmund spaces as well as its discrete counterparts (see Theorem \ref{ThmZygmundExp}) and Lorentz-Zygmund spaces on $\R^d$ (cf. Section \ref{SectionAdditional2}). Furthermore, our method enables us to show \eqref{ParcetSoriaXu+} for arbitrary uniformly bounded orthonormal systems (see Corollary \ref{Cor7.19}). Note that in general it is unknown wether nontrivial (Rademacher) type implies nontrivial type for an arbitrary uniformly bounded orthonormal system (see \cite[Section 8.3]{GaKaKoTo98}).

Limiting interpolation also allows us to show the vector-valued version of the Bochkarev's inequality. In  \cite{Herz}, Herz showed that
\begin{equation*}
\F : L^{2,q}(\T^d) \to \ell^{2,r}(\Z^d) \quad \text{if and only if} \quad q \leq 2 \leq r.
\end{equation*}
In other words, apart from trivial cases, the Fourier transform does not map Lorentz spaces $L^{2,q}(\T^d)$ to $\ell^{2,q}(\Z^d)$. However, Bochkarev \cite{Bochkarev97, Bochkarev} proved that this obstruction can be overcome with the help of additional logarithmic decays of the sequence of Fourier of coefficients. More precisely, if $q > 2$ then
\begin{equation}\label{BochkarevIntro}
	\wh{f}^\ast (n) \leq C n^{-1/2} (1 + \log n)^{1/2 - 1/q} \|f\|_{L^{2,q}(\T)}, \quad n \in \N.
\end{equation}
Further related results can be found in \cite[Corollary 5.2]{Sinnamon}.

Assume that $X$ has Fourier type $p_0$ and let $q > p_0$. Then, we show that (see Corollary \ref{cor:Bochkarev})
	\begin{equation}\label{BochkarevIntro2}
	\wh{f}^\ast (n) \leq C n^{-1/p_0'} (1 + \log n)^{1/p_0 - 1/\max\{p_0',q\}} \|f\|_{L^{p_0,q}(\T^d;X)}.
\end{equation}
In particular, if $X = \C$ then we recover \eqref{BochkarevIntro}. In fact, we will obtain a stronger estimate than \eqref{BochkarevIntro2} in terms of Lorentz-Zygmund norms (see Theorem \ref{thm:PaleyExtreme}).

This paper is organized as follows:
\begin{itemize}
\item Section \ref{sec:prel} contains several preliminaries on interpolation and duality.
\item In Section \ref{sec:HYPHL} we define the notions of Fourier, Paley and Hardy--Littlewood (co)type and discuss their relations.
\item In Section \ref{sec:trans} a general result is obtained which allows to transfer estimates for the Fourier transform between the groups $\T^d$, $\Z^d$ and $\R^d$. We also present several counterexamples to prove sharpness of results of Section \ref{sec:HYPHL}.
\item In Section \ref{Pitt's inequality} we give the proof of Theorem \ref{thm:Pitt2intro}. Moreover, sharpness of the results is obtained on the class of $\ell^p$-spaces.
\item In Section \ref{sec:Radtype} some relations between Rademacher, Fourier, Paley, Hardy--Littlewood (co)type and Pitt's inequalities are obtained.
\item In Section \ref{sec:Paley} we characterize the validity of the Zygmund inequality for Banach spaces and we show the vector-valued analogue of the Bochkarev's inequality.
\end{itemize}

\subsection*{Notation}
As usual, $\R^d$ denotes the $d$-dimensional real Euclidean space, $\T^d =  [-\frac{1}{2}, \frac{1}{2}]^d$ is the $d$-dimensional torus, $\Z^d$ is the lattice of all points in $\R^d$ with integer-valued components, $\N$ is the collection of all natural numbers and $\C$ is the complex plane. We use $B_r$ to denote the Euclidian ball in $\R^d$ centered at the origin with radius $r > 0$.

We denote nonzero Banach spaces over $\C$ by $X$. Given two Banach spaces $X_0, X_1$, we write $X_0 \hookrightarrow X_1$ if $X_0 \subseteq X_1$ and the natural embedding from $X_0$ into $X_1$ is continuous.

For $1 \leq p \leq \infty$, the number $p'$ is given by $\frac{1}{p} + \frac{1}{p'}=1$.

 We use the notation $A \lesssim_q B$ whenever $A \leq C B$ and $C$ is a positive constant which only depends on the parameter $q$. Similarly, we write $A\eqsim_q B$ if both $A\lesssim_q B$ and $B\lesssim_q A$.

\section{Preliminaries}\label{sec:prel}

\subsection{Function spaces}\label{Sec: FunctionSpaces}
Let $\Schw(\R^d;X)$ denote the space of $X$-valued Schwartz functions on $\R^d$. We write $\Schw(\R^d) = \Schw(\R^d;\C)$. The space of vector-valued tempered distributions is defined by $\Schw'(\R^d;X) = \calL(\Schw(\R^d),X)$, the space of all bounded linear operators from $\Schw(\R^d)$ to $X$. The Fourier transform is denoted by $\F$ and we use the normalization
\[\wh{f}(\xi) = \F (f)(\xi) = \int_{\R^d} f(t) e^{-2\pi i t\cdot \xi} \, \ud t, \ \ \  f\in L^1(\R^d;X).\]
Recall that $\F$ acts continuously on $\Schw(\R^d;X)$ and $\Schw'(\R^d;X)$.

Let $(S,\Sigma,\mu)$  be a measure space with $\sigma$-finite positive measure $\mu$. For a weight function $w:S\to (0,\infty)$ which is integrable on sets in $\Sigma$ of finite measure and $p\in [1, \infty)$, we consider the Bochner spaces $L^{p}(S,w;X)$ formed by all (equivalence classes of) strongly $\mu$-measurable functions $f: S \to X$ having a finite norm
\begin{equation}\label{Bochner}
\|f\|_{L^{p}(S,w;X)}= \Big(\int \limits _{S} \|f(s)\|_{X} ^{p} w(s)\, \ud \mu(s)\Big)^{1/p}.
\end{equation}
For $p=\infty$ we let $L^{\infty}(S,w;X) = L^{\infty}(S;X)$ denote the functions which are essentially bounded and set
\[\|f\|_{L^{\infty}(S,w;X)} = \text{ess.}\sup_{s\in S} \|f(s)\|_{X}.\]
In the special case $\omega \equiv 1$ we simply write $L^p(S;X)$ and $L^p(S, w) = L^p(S, w; \C)$. If $S=\Z^d$ then we shall use the notation $\ell^p(\Z^d, w; X)$.

If $f : S \to X$ then by $f^\ast(t),\, 0 < t < \mu(S)$, we mean the \emph{non-increasing rearrangement} of the scalar-valued function $\|f(s)\|_X, \, s \in S$. For $1 \leq p \leq \infty, 1 \leq q \leq \infty$ and $-\infty < b < \infty$, the \emph{Lorentz-Zygmund space} $L^{p,q} (\log L)^b(S;X)$ is the class of all $\mu$-strongly measurable functions $f : S \to X$ such that
\begin{equation}\label{DefLZ}
	\|f\|_{L^{p,q}(\log L)^b(S;X)} = \Big(\int_0^{\mu(S)} (t^{1/p} (1 + |\log t|)^b f^\ast(t))^q \,\frac{dt}{t} \Big)^{1/q} < \infty
\end{equation}
(with the usual modification if $p=\infty$ and/or $q=\infty$). See \cite{BennettRudnick, EdmundsEvans}. If $p=\infty$ then the case of interest is $b + 1/q < 0 \, (b \leq 0 \text{ if } q=\infty)$. Otherwise, it is easy to check that the Lorentz-Zygmund space $L^{\infty,q}(\log L)^b(S;X)$ becomes trivial, in the sense that it contains the zero function only.  Note that if $b=0$ in $L^{p,q}(\log L)^b(S;X)$ then we obtain the \emph{Lorentz spaces} $L^{p,q}(S;X)$ and if, in addition, $p=q$ then we get $L^p(S;X)$. If $p=q$ and $b \neq 0$ then we recover the \emph{Zygmund spaces} $L^p(\log L)^b(S;X)$. If $X = \C$ then we simply write $L^{p,q}(\log L)^b(S)$. In the special case $S=\Z^d$, we shall use the standard notation $\ell^{p,q}(\log \ell)^b(\Z^d;X)$, and we will use the following equivalent quasi-norm on this space:
\[\|(x_n)_{n\in \Z^d}\|_{\ell^{p,q}(\log \ell)^b(\Z^d;X)} =  \Big(\sum_{k=1}^\infty (k^{1/p} (1 + \log k)^b x_k^*)^q \frac{1}{k} \Big)^{1/q},\]
where $(x_k^*)_{k\geq 1}$ denotes the nonincreasing rearrangement of $(x_n)_{n\in \Z^d}$.

\subsection{Interpolation of Bochner spaces}\label{SectionIB}
Let $(X_0, X_1)$ be a Banach couple. If $0 < \theta < 1$ and $1 \leq q \leq \infty$, we shall denote by $(X_0, X_1)_{\theta,q}$ the \emph{real interpolation space}. The \emph{complex interpolation space} is denoted by $[X_0, X_1]_\theta$. See \cite{Bennett-Sharpley88, Triebel95}.

For the convenience of the reader, we collect some well-known interpolation formulas for Bochner spaces.

\begin{lemma}[{\cite[Section 1.18.5]{Triebel95}}]\label{lem:interpolationWeighted}
Let $1 \leq p_0, p_1 < \infty$ and $0 < \theta < 1$. Assume that $1/p = (1-\theta)/p_0 + \theta/p_1$ and $w = w_0^{(1-\theta) p/p_0} w_1^{\theta p/p_1}$. Then,
	\begin{equation*}
		[L^{p_0}(S, w_0;X), L^{p_1}(S, w_1;X)]_\theta = L^p(S, w;X).
	\end{equation*}
\end{lemma}

\begin{lemma}[{\cite[Section 1.18.6]{Triebel95}}]\label{lem:interpolationLebesgue}
Let $1 \leq p_0, p_1 \leq \infty, p_0 \neq p_1, 1 \leq q \leq \infty, 0 < \theta < 1$, and $1/p = (1-\theta)/p_0 + \theta/p_1$. Then,
	\begin{equation*}
		(L^{p_0}(S;X), L^{p_1}(S;X))_{\theta, q} = L^{p,q}(S;X).
	\end{equation*}
\end{lemma}

\begin{lemma}[{\cite[Section 1.18.4]{Triebel95}}]\label{lem:interpolationLebesgue2}
Let $1 \leq p_0, p_1 < \infty, 0 < \theta < 1$, and $1/p = (1-\theta)/p_0 + \theta/p_1$. Let $(X_0, X_1)$ be a Banach couple. Then,
\begin{equation*}
	(L^{p_0}(S;X_0), L^{p_1}(S;X_1))_{\theta,p} = L^p(S; (X_0, X_1)_{\theta,p})
\end{equation*}
and
\begin{equation*}
	[L^{p_0}(S;X_0), L^{p_1}(S;X_1)]_{\theta} = L^p(S; [X_0, X_1]_{\theta}).
\end{equation*}
\end{lemma}

As in \cite[Proposition 2.4.23]{HyNeVeWe16} one sees that $\Sch(\R^d;X)$ is dense in $L^p(\R^d;X)\cap L^q(\R^d;X)$ for all $1\leq p\leq q<\infty$. By Lemma \ref{lem:interpolationLebesgue}, the same holds for the Lorentz spaces $L^{p,r}(\R^d;X)$ for $p,r\in [1, \infty)$ (see \cite[Theorem 1.6.2]{Triebel95}).
The same density result holds for the weighted $L^p$-spaces with $A_{\infty}$-weights (see \cite[Lemma 3.5]{LinVerLp}). Moreover, we will also use the more elementary fact that the step functions are dense in each of these spaces.

\subsection{Norming subspaces and duality}
As usual, we denote by $X^*$ the dual space of $X$ with $\lb x, x^* \rb = x^*(x), \, x \in X, \, x^* \in X^*$. We say that a set $Y\subseteq X^*$ is {\em norming} for $X$ if $\sup_{x^*\in Y\setminus \{0\}}\frac{\lb x, x^*\rb}{\|x^*\|_{X^*}} = \|x\|_X$.

For later use, we shall need some duality assertions for vector-valued function spaces.
\begin{lemma}\label{lem:duality0}
Let $(S, \Sigma, \mu)$ be a $\sigma$-finite measure space and $w:S\to (0,\infty)$ be measurable. Let $p\in (1, \infty)$ and let $Y\subseteq X^*$ be a normed closed subspace of $X^*$ which is norming for $X$. Then $L^{p'}(S,w^{-\frac{1}{p-1}};Y)$ is norming for $L^p(S,w;X)$ with respect to the duality pairing
\[\lb f, g\rb = \int_{S} \lb f(s), g(s) \rb \,\ud\mu(s).\]
Moreover, the subspace of $Y$-valued simple functions in $L^{p'}(S,w^{-\frac{1}{p-1}};Y)$ is also norming for $L^p(S,w;X)$.
\end{lemma}
\begin{proof}
The unweighted case is contained in \cite[Proposition 1.3.1]{HyNeVeWe16}. The weighted case can be derived as a consequence by applying the unweighted case to $w^{1/p} f$ (see (\ref{Bochner})). The final statement follows by density.
\end{proof}

\begin{lemma}\label{lem:dualitynew}
	Let $(S,\Sigma, \mu)$ be a $\sigma$-finite nonatomic measure space. Let $1 < p, q < \infty$ and let $Y\subseteq X^*$ be a normed closed subspace which is norming for $X$.
Then,
\begin{equation}\label{dualLorentznew}
		\|f\|_{L^{p,q}(S; X)}  \asymp  \sup \left\{\Big|\int_{S} \lb  f(s), g(s)\rb  \,\ud \mu (s)\Big| : \|g\|_{L^{p',q'}(S;Y)} \leq 1\right\}.
	\end{equation}
\end{lemma}

Some comments are in order. The restriction to nonatomic measure spaces already appears in the scalar case (see \cite[Theorem 4.7]{Bennett-Sharpley88}). On the other hand, unlike in Lemma \ref{lem:duality0},  \eqref{dualLorentznew} does not hold with equality in general (see \cite[Theorem 4.4]{BarzaKolyadaSoria}).

\begin{proof}[Proof of Lemma \ref{lem:dualitynew}]
The result is well known to experts. For convenience of the reader we give a simple proof which reducing the problem to the scalar case.
Let $f \in L^{p,q}(S;X)$ and let $g \in L^{p',q'}(S;Y)$ be such that $\|g\|_{L^{p',q'}(S;Y)} \leq 1$. Applying the H\"older inequality for Lorentz spaces $L^{p,q}(S)$ (see \cite[Theorem 3.5]{ONeil63}), we get
\begin{align*}
	\Big|\int_{S} \lb  f(s), g(s)\rb  \, \ud \mu(s)\Big| & \leq \int_{S} \|f(s)\|_X \|g(s)\|_{X^\ast} \, \ud \mu(s) \\
	& \leq \|f\|_{L^{p,q}(S;X)} \|g\|_{L^{p',q'}(S;X^\ast)} \leq  \|f\|_{L^{p,q}(S;X)}.
\end{align*}
This prove the estimate $\geq$ in \eqref{dualLorentznew}.
To prove the converse estimate, we may assume without loss of generality that $f \in L^{p,q}(S;X)$ is a simple function, that is, $f = \sum_{n=1}^N \ind_{A_n} \otimes x_n$, where the $A_n\subseteq S$ are measurable disjoint sets of finite measure and all $x_n \in X$ non-zero. Define $u(s) = \|f(s)\|_{X}$. By construction, the function $u \in L^{p,q}(S)$ with $\|u\|_{ L^{p,q}(S)} = \|f\|_{L^{p,q}(S;X)}$. Accordingly, by \cite[Chapter 4, Theorem 4.7, page 220]{Bennett-Sharpley88} there exists $v \in L^{p',q'}(S), v \geq 0,$ with $\|v\|_{L^{p',q'}(S)} \leq 1$ satisfying that
\begin{equation*}
	\|u\|_{ L^{p,q}(S)} \asymp \Big|\int_{S} u(s) v(s) \, \ud \mu(s) \Big|.
\end{equation*}
Since $Y$ is norming for $X$, for each $n \in \{1, \ldots, N\}$ there exists $x_n^* \in Y$ of norm one such that $\lb  x_n, x_n^\ast\rb  \geq \frac12 \|x_n\|_{X}$. Define $g(s) = \sum_{n=1}^N \ind_{A_n}(s) v(s) x_n^*$. Note that $\|g\|_{L^{p',q'}(S;Y)} \leq \|v\|_{L^{p',q'}(S)} \leq 1$ because $\|g(s)\|_{Y} \leq v(s)$, and, in addition, we have
\begin{equation*}
	\lb  f(s), g(s)\rb  = v(s) \sum_{n=1}^N \lb  x_n, x_n^\ast\rb  \ind_{A_n}(s) \asymp v(s) \sum_{n=1}^N \|x_n\|_{X} \ind_{A_n}(s) = v(s) u (s).
\end{equation*}
Therefore,
\begin{equation*}
	 \|f\|_{L^{p,q}(S;X)} \asymp \Big|\int_{S} \lb  f(s), g(s)\rb   \, \ud \mu(s) \Big|.
\end{equation*}
This finishes the proof of \eqref{dualLorentznew}.
\end{proof}

For later use, we shall also need the characterizations of the dual spaces of the vector-valued sequence spaces $\ell^p(\Z^d,w;X)$ and $\ell^{p,q}(\Z^d;X)$ given in Lemma \ref{lem:duality3} below. This is essentially known in the literature, but we could not find a reference for the particular result which we need.

As usual, if $(X_n)$ is a sequence of Banach spaces, we denote by $\ell^p(\Z^d; X_n)$ the space formed by all sequences $(x_n)$ with $x_n \in X_n$ such that
\begin{equation*}
 \|(x_n)\|_{\ell^p(\Z^d; X_n)} = \left(\sum_{n \in \Z^d} \|x_n\|_{X_n}^p\right)^{1/p} < \infty.
 \end{equation*}
 It is folklore that for all $p\in[1, \infty)$
 \begin{equation}\label{dualVectorSequences}
   (\ell^p(\Z^d;X_n))^\ast = \ell^{p'}(\Z^d;X_n^\ast)
   \end{equation}
   isometrically.

\begin{lemma}\label{lem:duality3}
	Let $1 < p, q < \infty$ and $w : \Z^d \to (0,\infty)$. Then, we have
	\begin{equation}\label{dualWeightedSequences}
		(\ell^p(\Z^d, w ; X))^\ast = \ell{^{p'}}(\Z^d, w^{-\frac{1}{p-1}};X^\ast) \quad (\text{isometry of norms})
	\end{equation}
	and
	\begin{equation}\label{dualLorentzSequences}
		(\ell^{p,q}(\Z^d;X))^\ast = \ell^{p',q'}(\Z^d;X^\ast) \quad (\text{equivalence of norms}).
	\end{equation}
\end{lemma}
\begin{proof}
	We start by showing \eqref{dualWeightedSequences}. Note that  $\ell^p(\Z^d,w;X) = \ell^p(\Z^d; X_n)$ with $X_n = w(n)^{1/p} X$. Applying now \eqref{dualVectorSequences}, we easily derive \eqref{dualWeightedSequences} because $(\lambda X)^\ast = \lambda^{-1} X^\ast$ (with equality of norms) for any $\lambda > 0$.
	
	In order to prove \eqref{dualLorentzSequences} we will make use of the well-known interpolation formula
	\begin{equation}\label{interpolationDuality}
		(A_0, A_1)_{\theta,q}^\ast = ((A_0)^\ast, (A_1)^\ast)_{\theta,q'},
	\end{equation}
	whenever $A_0 \cap A_1$ is dense in $A_0$ and $A_1$ (see \cite[Theorem 1.11.2]{Triebel95}). Let $1 < p_0 < p < p_1 < \infty$. We define $\theta \in (0,1)$ such that $\frac{1}{p} = \frac{1-\theta}{p_0} + \frac{\theta}{p_1}$. Then, by Lemma \ref{lem:interpolationLebesgue}, (\ref{interpolationDuality}) and (\ref{dualVectorSequences}), we have
	\begin{align*}
		(\ell^{p,q}(\Z^d;X))^\ast & = ((\ell^{p_0}(\Z^d;X))^\ast, (\ell^{p_1}(\Z^d;X))^\ast)_{\theta,q'} = (\ell^{p_0'}(\Z^d;X^\ast), \ell^{p_1'}(\Z^d;X^\ast))_{\theta, q'} \\
		&= \ell^{p',q'}(\Z^d;X^\ast).
	\end{align*}
	The proof is finished.
\end{proof}

\subsection{Technical estimates}

The following vector-valued extension of the Stein--Weiss theorem on fractional integration  holds (see \cite{Beckner08} for the scalar case).
\begin{lemma}[Stein--Weiss]\label{lem:SobolevEmb}
Let $1<u\leq v<\infty$, $0<\lambda<d$, $a<\frac{d}{v}$, $b<\frac{d}{u'}$, $a+b\geq 0$ and $\frac{d}{v}+\frac{d}{u'} = \lambda+a+b$. Then there exists a constant $C > 0$ such that for all $g\in L^u(\R^d,|\cdot|^{b u};X)$,
\begin{align*}
\||\cdot|^{-\lambda}*g\|_{L^v(\R^d,|\cdot|^{-a v};X)} \leq C \|g\|_{L^u(\R^d,|\cdot|^{b u};X)}.
\end{align*}
\end{lemma}
\begin{proof}
The result in the scalar case follows from \cite{Beckner08} for $g\in \Schw(\R^d)$ and by density for $g\in L^u(\R^d,|\cdot|^{b u})$. Therefore, $g\mapsto |\cdot|^{-\lambda}*g$ defines a bounded linear mapping from $L^u(\R^d,|\cdot|^{b u})$ into $L^v(\R^d,|\cdot|^{-a v})$ of norm $\leq C$. By positivity this mapping extends to the $X$-valued setting with the same norm estimate (see \cite[Theorem 2.1.3]{HyNeVeWe16}).
\end{proof}

\section{Inequalities of Hausdorff--Young, Paley, and Hardy-Littlewood}\label{sec:HYPHL}

We start by introducing the notions of type and cotype that we will work with in order to prove our main result Theorem \ref{thm:Pitt2intro}.

\begin{definition}\label{def:type}
	Let $1 < p \leq 2\leq q<\infty$ and let $d\in \N$. A Banach space $X$ is said to have
\begin{enumerate}[(i)]
 \item\label{it:type1} {\em Fourier type $p$ on $\R^d$}  if there is a constant $C\geq 0$ such that
\begin{equation*}
		\|\widehat{f}\|_{L^{p'}(\R^d;X)} \leq C \|f\|_{L^p(\R^d;X)}
	\end{equation*}
\item\label{it:type2} {\em Paley type $p$ on $\R^d$} if there is a constant $C\geq 0$ such that
	\begin{equation*}
		\|\widehat{f}\|_{L^{p',p}(\R^d;X)} \leq C \|f\|_{L^p(\R^d;X)}
	\end{equation*}
\item\label{it:type3} {\em HL type $p$ on $\R^d$} if there is a constant $C\geq 0$ such that
	\begin{equation*}
\|\widehat{f}\|_{L^{p}(\R^d, |\cdot|^{-d(2-p)};X)} \leq C \|f\|_{L^p(\R^d;X)}
	\end{equation*}
 \item\label{it:cotype1} {\em Fourier cotype $q$ on $\R^d$} if there is a constant $C\geq 0$ such that
	\begin{equation*}
		\|\widehat{f}\|_{L^{q}(\R^d;X)} \leq C \|f\|_{L^{q'}(\R^d;X)}
	\end{equation*}
	\item\label{it:cotype2} {\em Paley cotype $q$ on $\R^d$} if there is a constant $C\geq 0$ such that
	\begin{equation*}
	\|\widehat{f}\|_{L^{q}(\R^d;X)} \leq C \|f\|_{L^{q',q}(\R^d;X)}
	\end{equation*}
	\item\label{it:cotype3} {\em HL cotype $q$ on $\R^d$} if there is a constant $C\geq 0$ such that
	\begin{equation*}
\|\widehat{f}\|_{L^{q}(\R^d;X)} \leq C \|f\|_{L^q(\R^d, |\cdot|^{d(q-2)};X)}
	\end{equation*}
	\end{enumerate}
for all $f \in \Schw(\R^d;X)$. Whenever it is clear that the dimension $d$ is fixed we leave out the additional `on $\R^d$'.
\end{definition}

Instead of considering the Schwartz class, one could consider step functions $f:\R^d\to X$ to introduce all notions given in Definition \ref{def:type}. By density the definition of Fourier type $p$ can equivalently be stated as $\F$ defines a bounded linear mapping $\F:L^p(\R^d;X)\to L^{p'}(\R^d;X)$. The same holds for the other parts of the definition. In the above definition HL stands for Hardy--Littlewood. Trivially, $X$ has Fourier cotype $q$ if and only if $X$ has Fourier type $q'$. This definition is only added for notational convenience.

It is well known that the scalar-valued estimates (i.e., $X=\C$) given in Definition \ref{def:type} hold for all $p\in (1, 2]$ and $q\in [2, \infty)$. Moreover, this extends to the Hilbert space setting (see \cite[Theorem 2.1.9]{HyNeVeWe16}). In the scalar case the estimates \eqref{it:type1}, \eqref{it:type2}, \eqref{it:type3} are called the Hausdorff-Young, Paley, and Hardy-Littlewood inequalities, respectively, and this explains the terminology in the above definition. Even in the scalar-valued setting, the Paley and the Hardy-Littlewood inequalities in $L^p(\R^d)$ only hold true when $1 < p \leq 2$, and their converse estimates only remain valid in $L^q(\R^d)$ with $2 \leq q < \infty$. The estimates \eqref{it:type1} and \eqref{it:cotype1} also (trivially) hold for $p=1$ and $q=\infty$ for any Banach space $X$.

\begin{remark}\label{RemarkTypeCotype}\
\begin{enumerate}[(i)]	
\item\label{RemarkTypeCotype1} It is simple to check that Fourier (co)type is independent of the dimension $d$ (see \cite[Proposition 2.4.11]{HyNeVeWe16}). However, it is still unclear wether a similar result holds for Paley and HL (co)type. See Remark \ref{RemarkDimensionTrans} below.

\item It is worthwhile mentioning that the notion of Paley cotype $q$ for Fourier transform given in Definition \ref{def:type}\eqref{it:cotype2} differs from the same concept introduced in \cite[Definition 2.6(2)]{GaKaKo01} for Fourier series. More specifically, the definition of Paley cotype $q$ for Fourier transform given here corresponds to the notion of strong Paley cotype* $q$ for Fourier series of \cite[Definition 2.12(2)]{GaKaKo01}.

\item If $p=2$, then all six notions Fourier/Paley/HL type and cotype, coincide. Moreover, by Kwapie\'n's theorem \cite{Kwa} this holds if and only if $X$ is isomorphic to a Hilbert space (see also \cite[Theorem 7.3.5]{HyNeVeWe2}).
    \end{enumerate}
\end{remark}

Next we show that there exist some duality results between the notions of type and cotype introduced in Definition \ref{def:type}.

\begin{proposition}\label{prop:duality}
	Let $1 < p < 2$. Then
	 	\begin{enumerate}[\upshape(i)]
		\item\label{DualF} $X$ has Fourier type $p$ if and only if $X^\ast$ has Fourier cotype $p'$.
	 \item\label{DualP} $X$ has Paley type $p$ if and only if $X^\ast$ has Paley cotype $p'$.
	 \item\label{DualHL} $X$ has HL type $p$ if and only if $X^\ast$ has HL cotype $p'$.
	 \end{enumerate}
	 The corresponding results also hold true if we exchange the roles of $X$ and $X^\ast$.
\end{proposition}
\begin{proof}
\eqref{DualF}: This follows from \cite[Proposition 2.4.16]{HyNeVeWe16}.
	
\eqref{DualP}: Assume first that $X$ has Paley type $p$. Let $f : \R^d \to X$ and $g: \R^d \to X^\ast$ be simple functions. Using the equality
	\begin{equation}\label{FourierCom}
		\int_{\R^d} \lb  f(x), \widehat{g}(x)\rb \, \ud x = \int_{\R^d} \lb  \widehat{f}(x), g(x)\rb  \, \ud x
	\end{equation}
	of the Fourier transform and H\"older's inequality (see \cite[Theorem 3.5]{ONeil63}), we obtain
	\begin{align*}
		\Big|\int_{\R^d} \lb  f(x), \widehat{g}(x)\rb  \, \ud x\Big| &
\leq \|\widehat{f}\|_{L^{p',p}(\R^d;X)} \|g\|_{L^{p,p'}(\R^d;X^\ast)}
		\leq C \|f\|_{L^p(\R^d;X)}  \|g\|_{L^{p,p'}(\R^d;X^\ast)}
	\end{align*}
	where we have used that $X$ has Paley type $p$ in the last inequality. Taking the supremum over all  simple functions $f$ with $\|f\|_{L^p(\R^d;X)} \leq 1$ and using Lemma \ref{lem:duality0}, we derive that
	\begin{equation*}
	\|\widehat{g}\|_{L^{p'}(\R^d;X^\ast)} \leq C \|g\|_{L^{p,p'}(\R^d;X^\ast)},
	\end{equation*}
	which means that $X^\ast$ has Paley cotype $p'$.
	
	Conversely, suppose that $X^\ast$ has Paley cotype $p'$. Let $f \in L^{p}(\R^d;X)$. By \eqref{FourierCom}, we have
	\begin{equation*}
		\Big| \int_{\R^d} \lb  \widehat{f}(x), g(x)\rb  \,\ud x\Big| \leq C \|f\|_{L^{p}(\R^d;X)} \|g\|_{L^{p,p'}(\R^d;X^\ast)}.
	\end{equation*}
	Taking the supremum over all simple functions $g \in L^{p,p'}(\R^d;X^\ast)$ for which $\|g\|_{L^{p,p'}(\R^d;X^\ast)} \leq 1$ and using Lemma \ref{lem:dualitynew},
it follows that $X$ has Paley type $p$.
	
	It remains to show \eqref{DualHL}. Assume first that $X$ has HL type $p$. Then, applying again \eqref{FourierCom} and H\"older's inequality, we obtain
	\begin{align*}
		\Big|\int_{\R^d} \lb  f(x), \widehat{g}(x)\rb  \, \ud x \Big|  & \leq \|\widehat{f}\|_{L^p(\R^d, |\cdot|^{-d(2-p)};X)} \|g\|_{L^{p'}(\R^d, |\cdot|^{d(p'-2)};X^\ast)}  \\
		& \leq C \|f\|_{L^p(\R^d;X)} \|g\|_{L^{p'}(\R^d, |\cdot|^{d(p'-2)};X^\ast)}.
	\end{align*}
	Taking the supremum over all simple functions $f \in L^p(\R^d;X)$ with $\|f\|_{L^p(\R^d;X)} \leq 1$, it follows from Lemma \ref{lem:duality0} that
$X^\ast$ has HL cotype $p'$. The converse direction may be proved similarly.
\end{proof}

Next we study the relationships between Fourier, Paley, and HL type (as well as, their cotype counterparts). We first state some elementary relations.

\begin{proposition}\label{prop:relationsPFHL}
	Let $1 < p < 2 < q < \infty$.
	\begin{enumerate}[\upshape(i)]
	 \item\label{it:relationsPFHL1} If $X$ has Paley type $p$, then $X$ has Fourier type $p$ and HL type $p$.
	 \item \label{it:relationsPFHL2} If $X$ has Paley cotype $q$, then $X$ has Fourier cotype $q$ and HL cotype $q$.
	 \end{enumerate}
\end{proposition}

In order to show Proposition \ref{prop:relationsPFHL} we will make use of the following Hardy-Littlewood rearrangement inequalities: Let $f$ and $w$ be non-negative measurable functions on $\R^d$, then
\begin{equation}\label{HL}
	\int_{\R^d} f(x) w(x) \, \ud x \leq \int_0^\infty f^\ast(t) w^\ast(t) \,dt
\end{equation}
and
\begin{equation}\label{HLConverse}
	\int_0^\infty f^\ast(t) \frac{1}{(1/w)^\ast(t)} \, dt \leq \int_{\R^d} f(x) w(x) \, \ud x.
\end{equation}
For a proof of the classical inequality \eqref{HL} we refer e.g.~to \cite[Chapter 2, Theorem 2.2, page 44]{Bennett-Sharpley88}. On the other hand, a more general version of \eqref{HLConverse} can be found in \cite[Corollary 2.5]{Heinig84} (see also \cite[(2.4)]{Benedetto-Heinig03}).

\begin{proof}[Proof of Proposition \ref{prop:relationsPFHL}]
	 It is clear that Paley type $p$ implies Fourier type $p$ because $L^{p',p}(\R^d;X) \hookrightarrow L^{p'}(\R^d;X)$ (see, e.g., \cite[Chapter 4, Proposition 4.2, page 217]{Bennett-Sharpley88}). On the other hand, using \eqref{HL} and basic properties for rearrangements, we get
	 \begin{equation}\label{emb}
	 	\int_{\mathbb{R}^d} |\xi|^{-d(2-p)} \|\wh{f}(\xi)\|_X^p \, \ud \xi \leq C \int_0^\infty t^{p/p'} (\|\wh{f}(\cdot)\|_X)^\ast(t)^p \, \frac{dt}{t}
	 \end{equation}
	 where we have used that $p < 2$. Hence, if $X$ has Paley type $p$ then
	  \begin{equation*}
	 	\int_{\mathbb{R}^d} |\xi|^{-d(2-p)} \|\wh{f}(\xi)\|_X^p\, \ud \xi \leq C \|f\|_{L^p(\mathbb{R}^d;X)}^p,
	 \end{equation*}
	 that is, $X$ has HL type $p$. The proof of \eqref{it:relationsPFHL1} is finished.
	
	 Let us prove \eqref{it:relationsPFHL2}. Assume that $X$ has Paley cotype $q$. Since $q' < q$, we have $L^{q'}(\R^d;X) \hookrightarrow L^{q',q}(\R^d;X)$. Consequently, $X$ has Fourier cotype $q$.
	
	 In order to show that $X$ has HL cotype $q$, it suffices to prove
	 \begin{equation}\label{emb2}
	 	\|f\|_{L^{q',q}(\R^d;X)} \leq C \|f\|_{L^q(\R^d, |\cdot|^{d(q-2)};X)} \quad \text{for all} \quad f \in L^q(\R^d, |\cdot|^{d(q-2)};X).
	 \end{equation}
	 The proof of this estimate is a simple application of \eqref{HLConverse} with $w(x) = |x|^{d(q-2)}$ for $q > 2$. 	An alternative proof of \eqref{emb2} can be obtained from \eqref{emb} by using duality (see Lemmas \ref{lem:duality0} and \ref{lem:dualitynew}).
\end{proof}

\begin{proposition}\label{prop:WeakFourier}
	Assume that $X$ has Fourier type $p_0 \in (1,2]$. Then
	\begin{enumerate}[\upshape(i)]
		\item\label{it:weakFtype} $X$ has Fourier type $p$ for any $p \in (1, p_0)$.
		\item\label{it:weakPtype} $X$ has Paley type $p$ for any $p \in (1,p_0)$ and Paley cotype $q$ for any $q \in (p_0', \infty)$.
		\item\label{it:weakHtype} $X$ has HL type $p$ for any $p \in (1,p_0)$ and HL cotype $q$ for any $q \in (p_0', \infty)$.
	\end{enumerate}
\end{proposition}

	The statement \eqref{it:weakFtype} of the previous proposition is well known (see, e.g., \cite[page 111]{HyNeVeWe16}). While the first part of \eqref{it:weakPtype} was already established in \cite[page 220]{Konig91} for the Fourier coefficient map on $\T$. For convenience, we shall include here the proofs of these assertions.

\begin{proof}[Proof of Proposition \ref{prop:WeakFourier}]
Note that
	\begin{equation}\label{3.11}
		\F : L^{p_0}(\R^d;X) \to L^{p_0'}(\R^d;X) \ \ \text{and} \ \ 		\F : L^{1}(\R^d;X) \to L^{\infty}(\R^d;X).
	\end{equation}
Now given any $1 < p < p_0$, there is $\theta \in (0,1)$ with $1/p = 1-\theta + \theta/p_0$. Let $1 \leq r \leq \infty$. Interpolating the  boundedness estimates of the operator $\F$ given in \eqref{3.11} with parameters $(\theta,r)$ and applying Lemma \ref{lem:interpolationLebesgue}, we get
\begin{equation}\label{Konig}
		\F: L^{p,r}(\R^d;X) \to L^{p',r}(\R^d;X).
	\end{equation}
	In particular, putting $r = p$ we obtain the first part of \eqref{it:weakPtype}. In addition, since $p < p'$ we have that $L^{p',p}(\R^d;X) \hookrightarrow L^{p'}(\R^d;X)$, and thus,
\eqref{it:weakFtype} follows as well.

	On the other hand, suppose that $q > p_0'$. If we interpolate \eqref{3.11} but now with the parameters $(\theta, q) = (p_0'/q,q)$, we obtain
		$\F: L^{q',q}(\R^d;X) \to L^{q}(\R^d;X)$,
	which means that $X$ has Paley cotype $q$. This finishes the proof of \eqref{it:weakPtype}.
	
	Finally, \eqref{it:weakHtype} is an immediate consequence of \eqref{it:weakPtype} and Proposition \ref{prop:relationsPFHL}.
	
\end{proof}

\begin{remark}\label{rem:limit}
	It will be shown in Examples \ref{ex:Paley} and \ref{ex:PaleyCotype} below that the statements given in Proposition \ref{prop:WeakFourier}\eqref{it:weakPtype} and \eqref{it:weakHtype} are no longer true in the extreme cases $p = p_0$ and $q=p_0'$.
\end{remark}

In light of Proposition \ref{prop:WeakFourier}\eqref{it:weakFtype}, if $X$ has Fourier type $p$ then $X$ has Fourier type $r$ for all $r \in [1,p]$. Accordingly we study such a question for the notions of HL and Paley (co)type.

\begin{proposition}\label{prop:HLprop}
	Let $p\in (1, 2]$ and $q\in [2, \infty)$. We have
	\begin{enumerate}[\upshape(i)]
\item\label{it:HLtypepr} If $X$ has HL type $p$, then $X$ has HL type $r$ for all $1 < r \leq p \leq 2$.
\item\label{it:HLcotypeqr}  If $X$ has HL cotype $q$, then $X$ has HL cotype $r$ for all $2 \leq q \leq r < \infty$.
\item\label{it:Paleytypepr} If $X$ has Paley type $p$, then $X$ has Paley type $r$ for all $1<r\leq p \leq2$.
\item\label{it:Paleycotypeqr} If $X$ has Paley cotype $q$, then $X$ has Paley cotype $r$ for all $2 \leq q \leq r < \infty$.
	\end{enumerate}
\end{proposition}
\begin{proof}
\eqref{it:HLtypepr}: The fact that $X$ has HL type $p$ means that the map $f \to |\xi|^d \widehat{f}$ is bounded from $L^p(\R^d;X)$ into $L^p(\R^d, |\cdot|^{-2d};X)$. Furthermore, working with the measure $\mu$ given by $d \mu(\xi) = |\xi|^{-2 d} \ud \xi$ we have for $t > 0$,
	\begin{align*}
		\mu \{\xi \in \R^d : |\xi|^d \|\widehat{f}(\xi)\|_X > t\} &= \int_{ \{\xi \in \R^d: |\xi|^d \|\widehat{f}(\xi)\|_X > t\}} |\xi|^{-2 d} \, \ud \xi \\
		& \leq \int_{\left\{\xi \in \R^d : |\xi|^d \geq C t \|f\|_{L^1(\R^d;X)}^{-1}\right\}}|\xi|^{-2 d} \, \ud \xi \\
		& \leq C \, \frac{\|f\|_{L^1(\R^d;X)}}{t}
	\end{align*}
	where the first inequality is an immediate consequence of the fact that the Fourier transform of a function $f \in L^1(\R^d;X)$ is bounded. Hence, we have shown that the map $f \to |\xi|^d \widehat{f}$ is weak (1,1)-bounded. Given $1 < r < p$, we choose $\theta \in (0,1)$ such that $1/r = 1-\theta + \theta/p$, and applying now the vector-valued version of the Marcinkiewicz interpolation theorem (see \cite[Theorem 2.2.3]{HyNeVeWe16}) we obtain that
	\begin{equation*}
		f \in L^r(\R^d;X) \to |\xi|^d \widehat{f} \in L^r(\R^d, |\cdot|^{-2 d} ; X)
	\end{equation*}
	or equivalently, $X$ has HL type $r$. The proof of \eqref{it:HLtypepr} is completed.
	
\eqref{it:HLcotypeqr}: This follows from \eqref{it:HLtypepr} and duality (see Proposition \ref{prop:duality}\eqref{DualHL}).

\eqref{it:Paleytypepr}: Fix $1<r<p$. By Proposition \ref{prop:relationsPFHL}\eqref{it:relationsPFHL1} $X$ has Fourier type $p$ and hence by Proposition \ref{prop:WeakFourier}\eqref{it:weakPtype} $X$ has Paley type $r$.
	
\eqref{it:Paleycotypeqr}: This can be proved in the same way as \eqref{it:Paleytypepr} or by using duality again.
\end{proof}

\begin{proposition}\label{prop:LpXPaley}
Let $(S,\Sigma,\mu)$ be a $\sigma$-finite measure space.
Let $p\in (1, 2]$ and $q\in [2, \infty)$.
\begin{enumerate}[\upshape(i)]
\item\label{it:LpXPaley1} If $X$ has Paley type $p$, then $L^p(S;X)$ has Paley type $p$.
\item\label{it:LpXPaley2} If $X$ has Paley cotype $q$, then $L^q(S;X)$ has Paley cotype $q$.
\end{enumerate}
The same holds for HL (co)type.
\end{proposition}
A more general result holds for Fourier (co)type (see \cite[Example 2.4.14]{HyNeVeWe16}).
\begin{proof}
\eqref{it:LpXPaley1}:  Firstly, we claim that
\begin{equation}\label{MinkGen}
	L^p(S; L^{p',p}(\R^d;X)) \hookrightarrow L^{p',p}(\R^d; L^p(S;X)).
\end{equation}
This is obvious if $p=2$. Assume $p < 2$ and so, $p' > p$. Let $p < p_0 < p' < p_1$ and $\theta \in (0,1)$ such that $\frac{1}{p'} = \frac{1-\theta}{p_0} + \frac{\theta}{p_1}$. Applying Minkowski's inequality, $L^p(S; L^{p_i}(\R^d;X)) \hookrightarrow L^{p_i}(\R^d; L^p(S;X)), \, i=0, 1$. Therefore, by Lemmas \ref{lem:interpolationLebesgue} and \ref{lem:interpolationLebesgue2},
\begin{align*}
	L^p(S; L^{p',p}(\R^d;X)) & = L^p(S; (L^{p_0}(\R^d;X), L^{p_1}(\R^d;X))_{\theta, p} ) \\
	& = (L^p(S; L^{p_0}(\R^d;X)), L^p(S; L^{p_1}(\R^d;X)))_{\theta,p} \\
	& \hookrightarrow (L^{p_0}(\R^d; L^p(S;X)), L^{p_1}(\R^d; L^p(S;X)))_{\theta,p} \\
	& = L^{p',p}(\R^s; L^p(S;X)).
\end{align*}

It follows from (\ref{MinkGen}) and the assumption that $X$ has Paley type $p$ that

\begin{align*}
\|\wh{f}\|_{L^{p',p}(\R^d;L^p(S;X))} & \leq C \, \|\wh{f}\|_{L^p(S;L^{p',p}(\R^d;X)}
\\ & \leq C \, \|f\|_{L^p(S;L^{p}(\R^d;X))}  = C \, \|f\|_{L^{p}(\R^d;L^p(S;X))}.
\end{align*}

The assertion \eqref{it:LpXPaley2} can be proved in a similar way.

The proofs of the corresponding statements for HL (co)type are easier and left to the reader.
\end{proof}

It is well known that the space $L^p(S)$ with $p\in (1, \infty)$ has Fourier type $\min\{p,p'\}$ (or equivalently, Fourier cotype $\max\{p,p'\}$). This also follows from Propositions \ref{prop:relationsPFHL} and \ref{prop:LpXPaley}. Therefore, Propositions \ref{prop:LpXPaley} and \ref{prop:WeakFourier} lead to the following example.
\begin{example}\label{ex:PaleyLp}
Let $(S, \Sigma,\mu)$ be a $\sigma$-finite measure space. Let $p\in (1, 2]$ and $q\in [2, \infty)$. Then
\begin{enumerate}
\item $L^p(S)$ has Paley/HL type $p$ and Paley/HL cotype $r$ for any $r\in (p',\infty)$.
\item $L^q(S)$ has Paley/HL cotype $q$ and  Paley/HL type $r$ for any $r\in (1,q')$.
\end{enumerate}
In Remark \ref{rem:Lpsharp} below we will see that these statements are optimal.
\end{example}

\begin{remark}\label{rem:interFourierPaley}
Motivated by the Hausdorff-Young type inequality
\begin{equation}\label{Calderon}
\F : L^{p,q}(\R^d) \to L^{p',q}(\R^d), \quad 1 < p < 2,\quad 1 \leq q \leq \infty,
\end{equation}
(see \cite{Calderon} and \cite{Herz}), one could also study the boundedness properties of the Fourier transform as a mapping from $\F:L^{p,r}(\R^d;X)\to L^{p',s}(\R^d;X), \, 1 \leq r \leq s \leq \infty$. Note that Fourier type $p$ corresponds to $r=p, \, s=p'$ and Paley type $p$ corresponds to $r=s=p$.
The ranges $p\leq r\leq s\leq p'$ can be considered as intermediate notions between Fourier type and Paley type. Moreover, the ranges $r\in [1, p]$ and $s\in [p', \infty]$ are also interesting since they lead to a wider class of spaces than Fourier type $p$.

In particular in similar way as in Proposition \ref{prop:LpXPaley}  one can prove the following. If $X$ has Paley type $p$ and $r\in [p,p']$, then setting $Y = L^r(S;X)$ one has $\F:L^{p}(\R^d;Y)\to L^{p',r}(\R^d;Y)$. We leave the details to the reader.
\end{remark}

We close this section with interpolation properties of the notions introduced in Definition \ref{def:type}. It is well known that Fourier type behaves well under interpolation (see \cite[Proposition 2.4.17]{HyNeVeWe16}). To be more precise, let $(X_0,X_1)$ be a Banach couple such that $X_i$ has Fourier type $p_i, \, i=0,1,$ with $1 \leq p_0 \leq p_1 \leq 2$. Let $0 < \theta < 1$ and $1/p = (1-\theta)/p_0 + \theta/p_1$. Then, $(X_0,X_1)_{\theta,p}$ and $[X_0,X_1]_\theta$ both are of Fourier type $p$. The corresponding assertions for Paley (co)type and the real interpolation method also hold true; see \cite[Corollaries 5.3 and 5.9]{GaKaKo01}.

Next we show the stability properties of the HL (co)type under real and complex interpolation.

\begin{proposition}
	Let $(X_0,X_1)$ be a Banach spaces such that $X_i$ has HL type $p_i, i=0,1,$ with $1 < p_0 \leq p_1 \leq 2$. Let $0 < \theta < 1$ and $1/p = (1-\theta)/p_0 + \theta/p_1$. Then, $(X_0,X_1)_{\theta,p}$ and $[X_0,X_1]_\theta$ both are of HL type $p$.
	
	The corresponding result for HL cotype also holds true.
\end{proposition}
\begin{proof}
	First we make the following elementary observation: the space $X$ has HL type $p$ if and only if
	\begin{equation*}
		T : L^p(\R^d; X) \to L^p(\R^d, |\cdot|^{-2 d};X), \quad T f (x) = |x|^d \widehat{f}(x).
	\end{equation*}
	
	By assumptions, we have
	\begin{equation*}
		T : L^{p_i}(\R^d; X_i) \to L^{p_i}(\R^d, |\cdot|^{-2 d};X_i), \quad i = 0, 1,
	\end{equation*}
	and then
	\begin{equation}\label{17}
		T : (L^{p_0}(\R^d;X_0) , L^{p_1}(\R^d;X_1) )_{\theta,p} \to (L^{p_0}(\R^d, |\cdot|^{-2 d};X_0), L^{p_1}(\R^d, |\cdot|^{-2 d};X_1))_{\theta,p}.
	\end{equation}
	Since (see Lemma \ref{lem:interpolationLebesgue2})
	\begin{equation*}
		(L^{p_0}(\R^d;X_0) , L^{p_1}(\R^d;X_1) )_{\theta,p} = L^p(\R^d; (X_0,X_1)_{\theta, p})
	\end{equation*}
	and
	\begin{equation*}
		 (L^{p_0}(\R^d, |\cdot|^{-2 d};X_0), L^{p_1}(\R^d, |\cdot|^{-2 d};X_1))_{\theta,p} = L^p(\R^d, |\cdot|^{-2 d}; (X_0, X_1)_{\theta,p}),
	\end{equation*}
	it follows from (\ref{17}) that $T : L^p(\R^d; (X_0,X_1)_{\theta, p}) \to L^p(\R^d, |\cdot|^{-2 d};(X_0,X_1)_{\theta, p}).$ Hence, $(X_0,X_1)_{\theta, p}$ has HL type $p$. Working with the complex interpolation method one can proceed in the same way.
	
	The proof for HL cotype is similar and left to the reader.
\end{proof}

\section{Transference principles and examples}\label{sec:trans}

The periodic Fourier transform of $f \in L^1(\T^d;X)$ is given by
\begin{equation*}
	\widehat{f}(n) = \F(f)(n) = \int_{\T^d} f(x) e^{-2\pi i x \cdot n} \, \ud x, \quad n \in \Z^d.
\end{equation*}

It is well known that the notion of Fourier type given in Definition \ref{def:type}\eqref{it:type1} can equivalently be introduced through the periodic Fourier transform. To be more precise, the following transference principle
holds.

\begin{lemma}\label{lem:ClassicalTrans}
	Let $1 < p \leq 2$. The following assertions are equivalent:
	\begin{enumerate}[\upshape(i)]
		\item $\F$ extends to a bounded operator from $L^p(\R^d;X)$ into $L^{p'}(\R^d;X)$.
		\item $\F$ extends to a bounded operator from $\ell^p(\Z^d;X)$ into $L^{p'}(\T^d;X)$.
		\item $\F$ extends to a bounded operator from $L^p(\T^d;X)$ into $\ell^{p'}(\Z^d;X)$.
	\end{enumerate}
\end{lemma}

The proof of Lemma \ref{lem:ClassicalTrans} may be found in \cite[Proposition 2.4.20]{HyNeVeWe16}.

The aim of this section is to show the corresponding transference principles for the notions of Paley type and HL type given in Definition \ref{def:type}\eqref{it:type2} and \eqref{it:type3}, respectively.

We first deal with the transference principle for HL type. In fact, we shall show below a more general result which can be applied to the Pitt's inequality
\begin{equation*}
	\|\widehat{f}\|_{L^q(\R^d, |\cdot|^{-\gamma q};X)} \leq C \|f\|_{L^p(\R^d, |\cdot|^{\beta p};X)}
\end{equation*}
where $1 < p \leq q < \infty$,
	\begin{equation}\label{AssumpPitt}
 \max\left\{0, d \left(\frac{1}{p} + \frac{1}{q} - 1\right) \right\} \leq \gamma < \frac{d}{q} \quad \text{and} \quad \beta - \gamma = d \Big(1- \frac{1}{p} - \frac{1}{q}\Big).
	\end{equation}
	The validity of such an inequality will be studied in detail in Section \ref{Pitt's inequality}.

\begin{proposition}\label{prop:WeightedTrans}
	Let $ 1 < p , q < \infty$ and $\beta, \gamma \geq 0$ with $\beta - \gamma = d \Big(1- \frac{1}{p} - \frac{1}{q}\Big)$. The following assertions are equivalent:
	\begin{enumerate}[\upshape(i)]
		\item \label{it:TPR} $\F$ extends to a bounded operator from $L^p(\R^d, |\cdot|^{\beta p};X)$ into $L^{q}(\R^d, |\cdot|^{-\gamma q};X)$.
		\item \label{it:TPZ} $\F$ extends to a bounded operator from $\ell^p(\Z^d, (|n| + 1)^{\beta p};X)$ into $L^{q}(\T^d, |\cdot|^{-\gamma q};X)$.
		\item \label{it:TPT} $\F$ extends to a bounded operator from $L^p(\T^d, |\cdot|^{\beta p};X)$ into $\ell^{q}(\Z^d, (|n|+1)^{-\gamma q};X)$.
	\end{enumerate}
\end{proposition}
\begin{proof}
Set $g(\xi) = \F (\ind_{[-\frac{1}{2}, \frac{1}{2}]^d})(\xi)= \prod_{n=1}^d \frac{\sin(\pi \xi_n)}{\pi \xi_n}$.
Then
\begin{align}\label{aux1}
\sum_{m \in \Z^d} |g(\xi + m)|^q |\xi + m|^{-\gamma q}  |\xi|^{\gamma q} &\geq |g(\xi)|^q \geq C, \quad \xi \in [-\tfrac{1}{2}, \tfrac{1}{2}]^d,
\\
\label{eq:propgineq}
\sum_{m \in \Z^d} |g(\xi + m)|^q |\xi + m|^{-\gamma q}  |\xi|^{\gamma q} &\leq  \sum_{m \in \Z^d}  |g(\xi + m)|^q \leq C, \quad \xi \in [-\tfrac{1}{2}, \tfrac{1}{2}]^d,
\end{align}
where the last estimate follows from periodicity and $q>1$.

	\eqref{it:TPR} $\implies$ \eqref{it:TPZ}:
	For $x_k \in X, |k| \leq n$, we define
	\begin{equation*}
		f (t) = \sum_{|k| \leq n} \ind_{[-\frac{1}{2}, \frac{1}{2}]^d} (t + k) x_k, \quad t \in \R^d.
	\end{equation*}
We can write
	\begin{equation}\label{prop:WeightedTrans1new2}
		\widehat{f}(\xi) = g(\xi) \sum_{|k| \leq n} e_k(\xi) x_k, \quad \xi \in \R^d,
	\end{equation}
	 where $e_k(\xi) = e^{2 \pi i k \cdot \xi}$. Elementary computations yield that
	\begin{align}
		\|\widehat{f}\|_{L^{q}(\R^d, |\cdot|^{-\gamma q};X)}^q & = \int_{\R^d} |g(\xi)|^q |\xi|^{-\gamma q} \Big\| \sum_{|k| \leq n} e_k(\xi) x_k\Big\|_X^q \, \ud \xi \nonumber \\
		& = \sum_{m \in \Z^d} \int_{[-\frac{1}{2}, \frac{1}{2}]^d -m} |g(\xi)|^q |\xi|^{-\gamma q} \Big\| \sum_{|k| \leq n} e_k(\xi) x_k\Big\|_X^q \, \ud \xi \nonumber\\
		& = \sum_{m \in \Z^d} \int_{[-\frac{1}{2}, \frac{1}{2}]^d} |g(\xi + m)|^q |\xi + m|^{-\gamma q} \Big\| \sum_{|k| \leq n} e_k(\xi) x_k\Big\|_X^q \, \ud \xi \label{prop:WeightedTrans1}.
	\end{align}
From \eqref{aux1} and \eqref{prop:WeightedTrans1} we obtain the estimate
	\begin{equation*}
			\|\widehat{f}\|_{L^{q}(\R^d, |\cdot|^{-\gamma q};X)}^q \geq C \int_{[-\frac{1}{2}, \frac{1}{2}]^d} |\xi|^{-\gamma q} \Big\| \sum_{|k| \leq n} e_k(\xi) x_k\Big\|_X^q \, \ud \xi.
	\end{equation*}
	On the other hand, we have
	\begin{equation*}
		\|f\|_{L^p(\R^d, |\cdot|^{\beta p};X)}^p = \sum_{|k| \leq n} \int_{[-\frac{1}{2}, \frac{1}{2}]^d -k} |t|^{\beta p} \|x_k\|_X^p \, \ud t \asymp \sum_{|k| \leq n} \|x_k\|_X^p (1 + |k|)^{\beta p}.
	\end{equation*}
	Applying now that \eqref{it:TPR} holds true, we arrive at
	\begin{equation*}
		\Big\| \sum_{|k| \leq n} e_k x_k\Big\|_{L^q(\T^d, |\cdot|^{-\gamma q};X)} \leq C \|(x_k)_{|k| \leq n}\|_{\ell^p(\Z^d, (1 + |k|)^{\beta p};X)},
	\end{equation*}
	which implies \eqref{it:TPZ} by a density argument.
	
	\eqref{it:TPZ} $\implies$ \eqref{it:TPR}: To prove \eqref{it:TPR}, we first observe that it suffices to check that the inequality
	\begin{equation}\label{prop:WeightedTrans3new}
		\|\widehat{f}\|_{L^{q}(\R^d, |\cdot|^{-\gamma q};X)} \leq C \|f\|_{L^p(\R^d, |\cdot|^{\beta p};X)}
	\end{equation}
	holds for all functions $f$ given by
	\begin{equation*}
		f (t) =  \sum_{|k| \leq n} \ind_{[-\frac{1}{2}, \frac{1}{2}]^d} (t a^{-1} + k) x_k, \quad t \in \R^d, \quad a > 0.
	\end{equation*}
	We first remark that
	\begin{equation}\label{aux4}
		\widehat{f}(\xi) = a^{d} g(a \xi) \sum_{|k| \leq n} e_k(a \xi) x_k, \quad \xi \in \R^d.
	\end{equation}
	Then, by a change of variables, we have
		\begin{align}
		\|\widehat{f}\|_{L^{q}(\R^d, |\cdot|^{-\gamma q};X)} & = a^{d+\gamma -d/q} \left(\int_{\R^d} |g(\xi)|^q |\xi|^{-\gamma q} \Big\| \sum_{|k| \leq n} e_k(\xi) x_k\Big\|_X^q \, \ud \xi \right)^{1/q} \nonumber \\
		& \hspace{-1.5cm}=  a^{d+\gamma -d/q} \left(\int_{[-\frac{1}{2}, \frac{1}{2}]^d} \sum_{m \in \Z^d}  |g(\xi + m)|^q |\xi + m|^{-\gamma q} \Big\| \sum_{|k| \leq n} e_k(\xi) x_k\Big\|_X^q \, \ud \xi \right)^{1/q}. \label{aux2}
	\end{align}
	Inserting \eqref{eq:propgineq} into \eqref{aux2} we get
	\begin{equation}\label{prop:WeightedTrans3}
		\|\widehat{f}\|_{L^{q}(\R^d, |\cdot|^{-\gamma q};X)} \leq C  a^{d+\gamma -d/q} \left(\int_{[-\frac{1}{2}, \frac{1}{2}]^d} |\xi|^{-\gamma q} \Big\| \sum_{|k| \leq n} e_k(\xi) x_k\Big\|_X^q \, \ud \xi\right)^{1/q}.
	\end{equation}
	Furthermore, we have
	\begin{align}
			\|f\|_{L^p(\R^d, |\cdot|^{\beta p};X)} & =  \left(\sum_{|k| \leq n} \int_{[-\frac{a}{2}, \frac{a}{2}]^d -k a} |t|^{\beta p} \|x_k\|_X^p  \, \ud t \right)^{1/p}  \nonumber \\
			& \geq C a^{ \beta + d/p} \left(\sum_{|k| \leq n} \|x_k\|_X^p (1 + |k|)^{\beta p} \right)^{1/p} \nonumber \\
			& = C a^{ d + \gamma -d/q} \left(\sum_{|k| \leq n} \|x_k\|_X^p (1 + |k|)^{\beta p} \right)^{1/p} .\label{aux3}
	\end{align}
	By the assumption, \eqref{prop:WeightedTrans3} and \eqref{aux3}, we easily derive the desired estimate \eqref{prop:WeightedTrans3new}.
	
	\eqref{it:TPR} $\iff$ \eqref{it:TPT}: By Lemma \ref{lem:duality0} and \eqref{dualWeightedSequences}, the boundedness result
	\begin{equation}\label{prop:WeightedTrans4}
		\F : L^p(\R^d, |\cdot|^{\beta p};X) \to L^{q}(\R^d, |\cdot|^{-\gamma q};X)
	\end{equation}
	turns out to be equivalent to
	\begin{equation}\label{prop:WeightedTrans5}
		\F : L^{q'}(\R^d, |\cdot|^{\gamma q'};X^\ast) \to L^{p'}(\R^d, |\cdot|^{-\beta p'};X^\ast),
	\end{equation}
	while
	 \begin{equation}\label{prop:WeightedTrans6}
		\F:  \ell^{q'}(\Z^d, (|n|+1)^{\gamma q'};X^\ast) \to L^{p'}(\T^d, |\cdot|^{-\beta p'};X^\ast)
	\end{equation}
	can be characterized in terms of the boundedness estimate
		\begin{equation}\label{prop:WeightedTrans7}
		\F: L^p(\T^d, |\cdot|^{\beta p};X) \to \ell^{q}(\Z^d, (|n|+1)^{-\gamma q};X).
	\end{equation}
	Furthermore, we have already shown that \eqref{prop:WeightedTrans5} and \eqref{prop:WeightedTrans6} are equivalent, which yields the equivalence between the statements \eqref{prop:WeightedTrans4} and \eqref{prop:WeightedTrans7}.
\end{proof}

Specializing Proposition \ref{prop:WeightedTrans} to the case $p=q \leq 2, \beta =0$, and $\gamma = d(2/p - 1)$, we obtain the following characterizations of the notion of HL type.

\begin{corollary}\label{cor:HLTrans}
	Let $1 < p \leq 2$. Then the following assertions are equivalent:
	\begin{enumerate}[\upshape(1)]
		\item $X$ has HL type $p$.
		\item $\F$ extends to a bounded operator from $\ell^p(\Z^d;X)$ into $L^{p}(\T^d, |\cdot|^{- d(2-p)};X)$.
		\item $\F$ extends to a bounded operator from $L^p(\T^d;X)$ into $\ell^{p}(\Z^d, (|n|+1)^{-d(2-p)};X)$.
	\end{enumerate}
\end{corollary}

The corresponding result for HL cotype also follows from Proposition \ref{prop:WeightedTrans} by taking $p=q \geq 2, \gamma =0$, and $\beta = d(1 - 2/q)$. It reads as follows.

\begin{corollary}\label{cor:HLCoTrans}
	Let $2 \leq q < \infty$. Then the following assertions are equivalent:
	\begin{enumerate}[\upshape(i)]
		\item $X$ has HL cotype $q$.
		\item $\F$ extends to a bounded operator from $\ell^q(\Z^d, (|n|+1)^{d(q-2)};X)$ into $L^{q}(\T^d;X)$.
		\item $\F$ extends to a bounded operator from $L^q(\T^d, |\cdot|^{d(q-2)};X)$ into $\ell^{q}(\Z^d;X)$.
	\end{enumerate}
\end{corollary}

Next we prove a transference result for Paley (co)type between $\R^d$ and $\T^d$.

\begin{proposition}\label{prop:PTrans}
	Let $1 < p \leq 2$. Then the following assertions are equivalent:
	\begin{enumerate}[\upshape(i)]
		\item \label{it:TPPalR}$X$ has Paley type $p$.
		\item \label{it:TPPalZ}$\F$ extends to a bounded operator from $\ell^p(\Z^d;X)$ into $L^{p',p}(\T^d;X)$.
		\item \label{it:TPPalT}$\F$ extends to a bounded operator from $L^p(\T^d;X)$ into $\ell^{p',p}(\Z^d;X)$.
	\end{enumerate}
\end{proposition}
 \begin{proof}
 The equivalence between \eqref{it:TPPalZ} and \eqref{it:TPPalT}
 was already shown in \cite[Theorem 3.2]{GaKaKo01}. Then, it suffices to show that \eqref{it:TPPalR} can be characterized in terms of \eqref{it:TPPalT}.

 	\eqref{it:TPPalR} $\implies$ \eqref{it:TPPalT}: Let
	\begin{equation*}
		f(\xi) = g(\xi) \sum_{|k| \leq n} e_k(\xi) x_k, \quad \xi \in \R^d,
	\end{equation*}
	where $x_k \in X$ for $|k| \leq n$, and $g(\xi) = \F (\ind_{[-\frac{1}{2}, \frac{1}{2}]^d})(\xi)$ (see \eqref{prop:WeightedTrans1new2}). We have
	\begin{align}
		\|f\|_{L^{p}(\R^d;X)}^p & = \int_{\R^d} |g(\xi)|^p \Big\| \sum_{|k| \leq n} e_k(\xi) x_k\Big\|_X^p \, \ud \xi \nonumber \\
		& = \sum_{m \in \Z^d} \int_{[-\frac{1}{2}, \frac{1}{2}]^d} |g(\xi + m)|^p \Big\| \sum_{|k| \leq n} e_k(\xi) x_k\Big\|_X^p \, \ud \xi \nonumber\\
		& \leq C  \int_{[-\frac{1}{2}, \frac{1}{2}]^d} \Big\| \sum_{|k| \leq n} e_k(\xi) x_k\Big\|_X^p \, \ud \xi \label{PTrans1}
	\end{align}
	where in the last step we have used \eqref{eq:propgineq}

On the other hand, it is clear that
	\begin{equation*}
		\widehat{f} (t) = \sum_{|k| \leq n} \ind_{[-\frac{1}{2}, \frac{1}{2}]^d} (t + k) x_k, \quad t \in \R^d.
	\end{equation*}
	Let $M_n$ be the number of $k \in \Z^d$ such that $|k| \leq n$, and denote by $(x_j^\ast)_{j=1}^{M_n}$ the non-increasing rearrangement of the sequence $(\|x_k\|_X)_{|k| \leq n}$. Note that
	\begin{equation*}
		\widehat{f}^\ast(t) = \sum_{j=1}^{M_n} x_j^\ast \ind_{[j-1,j)}(t), \quad t > 0,
	\end{equation*}
	and so,
	\begin{equation}\label{PTrans2}
		\|\widehat{f}\|_{L^{p',p}(\R^d;X)}^p = \sum_{j=1}^{M_n} \int_{j-1}^j t^{p/p'-1} \widehat{f}^\ast(t)^p dt \asymp \sum_{j=1}^{M_n} (j^{1/p'} x_j^\ast)^p \frac{1}{j}.
	\end{equation}
	Hence, it follows from \eqref{it:TPPalR} together with the estimates \eqref{PTrans1} and \eqref{PTrans2} that
	\begin{equation*}
		\left(\sum_{j=1}^{M_n} (j^{1/p'} x_j^\ast)^p \frac{1}{j}\right)^{1/p} \leq C  \left\| \sum_{|k| \leq n} e_k x_k\right\|_{L^p(\T^d;X)}.
	\end{equation*}
	We conclude \eqref{it:TPPalT} by using a density argument.
	
	\eqref{it:TPPalT} $\implies$ \eqref{it:TPPalR}: It turns out that \eqref{it:TPPalT} is equivalent to say
	$\F$ extends to a bounded operator from $\ell^{p,p'}(\Z^d;X^\ast)$ into $L^{p'}(\T^d;X^\ast)$, that is,
	\begin{equation}\label{aux5}
		\Big\|\sum_{|k| \leq n} e_k(\xi) y_k\Big\|_{L^{p'}(\T^d;X^\ast)} \leq C \|(y_k)_{|k| \leq n}\|_{\ell^{p,p'}(\Z^d;X^\ast)}.
	\end{equation}
	This is an immediate consequence of Lemma \ref{lem:duality0} and \eqref{dualLorentzSequences}.
	
	  By density and Proposition \ref{prop:duality}(ii) it is enough to prove that the following inequality
	\begin{equation}\label{aux6}
		\|\widehat{f}\|_{L^{p'}(\R^d;X^\ast)} \leq C \|f\|_{L^{p,p'}(\R^d;X^\ast)}
	\end{equation}
	holds for all $f$ given by
	\begin{equation*}
		f (t) = \sum_{|k| \leq n} \ind_{[-\frac{1}{2}, \frac{1}{2}]^d} (t a^{-1} + k) y_k, \quad t \in \R^d, \quad a > 0,
	\end{equation*}
	where $y_k \in X^\ast, |k| \leq n,$ are arbitrary.
	Clearly, we have
	\begin{equation*}
		\widehat{f}(\xi) = a^d g(a \xi) \sum_{|k| \leq n} e_k(a \xi) y_k, \quad \xi \in \R^d,
	\end{equation*}
	with $g(\xi) = \F (\ind_{[-\frac{1}{2}, \frac{1}{2}]^d})(\xi)$ (see \eqref{aux4}). Therefore, setting $(y_j^\ast)_{j=1}^{M_n}$ the non-increasing rearrangement of the sequence $(\|y_k\|_{X^\ast})_{|k| \leq n}$ (recall that $M_n$ denotes the cardinality of the set $\{k \in \Z^d : |k| \leq n\}$), we obtain
	\begin{equation}\label{PTrans2new}
		\|f\|_{L^{p,p'}(\R^d;X^\ast)}^{p'} = \sum_{j=1}^{M_n} \int_{(j-1) a^d }^{j a^d} t^{p'/p-1} \widehat{f}^\ast(t)^{p'} dt \asymp a^{d p'/p}\sum_{j=1}^{M_n} (j^{1/p} y_j^\ast)^{p'} \frac{1}{j}.
	\end{equation}
	On the other hand by \eqref{eq:propgineq}, we have
	\begin{align}
		\|\widehat{f}\|_{L^{p'}(\R^d;X^\ast)}^{p'} & = a^{d p'/p} \int_{\R^d} |g(\xi)|^{p'} \Big\| \sum_{|k| \leq n} e_k(\xi) y_k\Big\|_{X^\ast}^{p'} \, \ud \xi \nonumber \\
		& = a^{d p'/p} \sum_{m \in \Z^d} \int_{[-\frac{1}{2}, \frac{1}{2}]^d} |g(\xi + m)|^{p'} \Big\| \sum_{|k| \leq n} e_k(\xi) y_k\Big\|_{X^\ast}^{p'}  \, \ud \xi \nonumber\\
		& \leq C a^{d p'/p}  \int_{[-\frac{1}{2}, \frac{1}{2}]^d} \Big\| \sum_{|k| \leq n} e_k(\xi) y_k\Big\|_{X^\ast}^{p'} \, \ud \xi \label{PTrans1new}
	\end{align}
	
	According to  \eqref{PTrans1new}, \eqref{aux5} and \eqref{PTrans2new} we get
	\begin{align*}
		\|\widehat{f}\|_{L^{p'}(\R^d;X^\ast)} &\leq C a^{d/p} \Big\| \sum_{|k| \leq n} e_k(\xi) y_k \Big\|_{L^{p'}(\T^d;X^\ast)} \leq C a^{d/p} \|(y_k)_{|k| \leq n}\|_{\ell^{p,p'}(\Z^d;X^\ast)} \\
		& \leq C \|f\|_{L^{p,p'}(\R^d;X^\ast)}.
	\end{align*}	
	This gives the proof of the desired estimate \eqref{aux6}.

 \end{proof}

 As a consequence of Proposition \ref{prop:PTrans} by applying Proposition \ref{prop:duality}\eqref{DualP} together with Lemmas \ref{lem:duality0} and \ref{lem:dualitynew}, we obtain the transference principle for Paley cotype.

\begin{proposition}\label{prop:PCoTrans}
	Let $2 \leq q < \infty$. Then the following assertions are equivalent:
	\begin{enumerate}[\upshape(1)]
		\item $X$ has Paley cotype $q$.
		\item $\F$ extends to a bounded operator from $\ell^{q',q}(\Z^d;X)$ into $L^{q}(\T^d;X)$.
		\item $\F$ extends to a bounded operator from $L^{q',q}(\T^d;X)$ into $\ell^{q}(\Z^d;X)$.
	\end{enumerate}
\end{proposition}

\begin{remark}\label{RemarkDimensionTrans}
	As already mentioned in Remark \ref{RemarkTypeCotype}\eqref{RemarkTypeCotype1}, it is well known that transference principles with respect to dimension hold for Fourier type. It would be desirable to know wether these transference results still hold true for HL (co)type and Paley (co)type. Next we get some partial results. In view of Propositions \ref{prop:PTrans} and \ref{prop:PCoTrans}, it is not hard to check that Paley (co)type $p$ for dimension $d > 1$ implies $d-i$-dimensional Paley (co)type $p$ for $i = 1, \ldots, d-1$. On the other hand, it is easily seen from definition that HL (co)type $p$ in the $1$-dimensional setting implies HL (co)type $p$ in the $d$-dimensional setting.
\end{remark}

In the rest of this section, we will use the above transference results to obtain simple examples which prove the sharpness of Paley/HL (co)type of the $\ell^p$-spaces (see Example \ref{ex:PaleyLp}).

\begin{example}\label{ex:Paley}
Let $1 < p < 2$. It is well known that the space $X =\ell{^{p'}}(\Z^d)$ has Fourier type $p$. In particular, $X$ has Paley/HL type $q$ for any $q < p$ (see Proposition  \ref{prop:WeakFourier}). However, as can be seen in \cite[page 220]{Konig91} and \cite[Proposition 4.6]{GaKaKo01}, one can construct a function $f : \T^d \to X$ such that $f \in L^p(\T^d; X)$ but its sequence of Fourier coefficients $(\widehat{f}(n))$ does not belong to $\ell^{p',p}(\Z^d; X)$. Therefore, in virtue of Proposition \ref{prop:PTrans}, we derive that $X$ is not of Paley type $p$. We will improve this by showing that $X$ does not have HL type $p$ (see Proposition \ref{prop:relationsPFHL}). Indeed, let $f : \T^d \to X$ defined by
	\begin{equation*}
		f(t) = (|n|^{-d/p'} (1 + \log |n|)^{-\varepsilon} e^{2 \pi i n\cdot t} \one_{n\neq 0}), \quad t \in \T^d,
	\end{equation*}
	where $1/p' < \varepsilon < 1/p$. We have
	\begin{equation*}
		\|f\|_{L^p(\T^d;X)} \asymp \left(\sum_{n\neq 0} (1 + \log |n|)^{-\varepsilon p'} \frac{1}{|n|^d}\right)^{1/p'} < \infty.
	\end{equation*}
	On the other hand, since $\|\widehat{f}(n)\|_X = |n|^{-d/p'} (1 + \log |n|)^{-\varepsilon}$ for $n \neq 0$, we get
	\begin{equation*}
		\|\widehat{f}\|_{\ell^p(\Z^d,(|n|+1)^{-d(2-p)}; X)} \eqsim \left(\sum_{n\neq 0} (1 + \log |n|)^{-\varepsilon p} \frac{1}{|n|^d} \right)^{1/p} = \infty.
	\end{equation*}
	Then, according to Corollary \ref{cor:HLTrans}, $X$ does not have HL type $p$. It is worth mentioning that this example reveals some further features of the notions of Paley type and HL type. Firstly, in general, the indices
	\begin{equation*}
		\sup \{p : X \quad \text{has Paley type} \quad  p\} \quad \text{and} \quad \sup \{p : X \quad \text{has HL type} \quad  p\}
	\end{equation*}
	are not attained. Recall that the corresponding result for Fourier type is already known (see \cite[page 222]{Konig91}). Secondly, Proposition \ref{prop:WeakFourier} does not hold in the limiting case $p=p_0$ (see Remark \ref{rem:limit}).
\end{example}

\begin{example}\label{ex:Paley2}
	Let $X = \ell^p(\Z^d), \,1 < p < 2$. By Example \ref{ex:PaleyLp} $X$ has Paley type $p$. Alternatively this follows from \cite[Theorem 4.2]{GaKaKo01} and Proposition \ref{prop:PTrans}.
Then, Proposition \ref{prop:relationsPFHL} implies that $X$ has Fourier type $p$ (see e.g.~\cite{Konig91}) and HL type $p$. Moreover, all these statements are sharp. More specifically, it is known that $X$ has exactly Fourier type $p$. Consequently, $p$ is the best possible index $q$ such that $X$ has Paley type $q$. Otherwise, one may apply Proposition \ref{prop:relationsPFHL} in order to derive that $X$ has Fourier type $q > p$, which is not true. Next we show that $X$ has no better HL type than $p$. We will proceed by contradiction. Let us assume that $X$ has HL type $q$ for some $q > p$. By Proposition \ref{prop:duality}, this is equivalent to saying that $X^\ast = \ell^{p'}(\Z^d)$ has HL cotype $q'$, that is,
	\begin{equation*}
		\|\F f\|_{\ell^{q'}(\Z^d;X^\ast)} \leq C \|f\|_{L^{q'}(\T^d, |\cdot|^{d(q'-2)};X^\ast)}
	\end{equation*}
	(see Corollary \ref{cor:HLCoTrans}). However, this inequality is not true which yields the desired contradiction. Indeed, let
	\begin{equation*}
		f(t) = (|n|^{-\varepsilon} e^{2 \pi i n\cdot t} \one_{n\neq 0}), \quad t \in \T^d,
	\end{equation*}
	where $d/p' < \varepsilon < d/q'$. We have
	\begin{equation*}
		 \|f\|_{L^{q'}(\T^d, |\cdot|^{d(q'-2)};X^\ast)} \asymp \left(\sum_{n\neq 0} |n|^{-\varepsilon p'}\right)^{1/p'} < \infty,
	\end{equation*}
	but
	\begin{equation*}
		\|\wh{f}\|_{\ell^{q'}(\Z^d;X^\ast)} =  \left(\sum_{n\neq 0}|n|^{-\varepsilon q'}\right)^{1/q'} = \infty.
	\end{equation*}
\end{example}

As a consequence of the preceding examples and Proposition \ref{prop:duality} we obtain the corresponding results for cotype.
\begin{example}\label{ex:PaleyCotype}
	The space $X = \ell^p(\Z^d), \,1 < p < 2$, is of Fourier type exactly $p$, Paley/HL cotype $q$ for any $q > p'$. Moreover, $X$ has neither Paley cotype $p'$ nor HL cotype $p'$.
\end{example}

\begin{example}\label{ex:PaleyCotype2}
	Let $X = \ell^p(\Z^d), \, 2 < p < \infty$. The space $X$ is of Paley/HL cotype $p$ and this is optimal.
\end{example}

\begin{remark}\label{rem:Lpsharp}
If $S$ in Example \ref{ex:PaleyLp} does not consist of finitely many atoms, then the example is optimal. Indeed, this follows from the above examples since $L^p(S)$ contains an isometric copy of $\ell^p(\Z^d)$.
\end{remark}

\section{Pitt's inequality for the vector-valued Fourier transform}\label{Pitt's inequality}

In this section we are concerned with weighted inequalities for the vector-valued Fourier transform. More precisely, we will prove that under some geometric conditions on $X$ the Pitt's inequality obeys
\begin{equation}\label{PittGeneral}
	\|\widehat{f}\|_{L^q(\R^d, |\cdot|^{-\gamma q};X)} \leq C \|f\|_{L^p(\R^d, |\cdot|^{\beta p};X)}
\end{equation}
for certain indices $p, q, \gamma$, and $\beta$.

Let us recall from the introduction that the Pitt's inequality in the scalar case (i.e., $X=\C$) claims that if $1 < p \leq q < \infty$ and
	\begin{equation}\label{PittAssumpScalar}
\max\left\{0, d \left(\frac{1}{p} + \frac{1}{q} - 1\right) \right\} \leq \gamma < \frac{d}{q} \quad \text{and} \quad \beta - \gamma = d \Big(1- \frac{1}{p} - \frac{1}{q}\Big),
	\end{equation}
	then
	\begin{equation}\label{PittScalar}
	\|\widehat{f}\|_{L^q(\R^d, |\cdot|^{-\gamma q})} \leq C \|f\|_{L^p(\R^d, |\cdot|^{\beta p})}.
\end{equation}
Notice that the conditions \eqref{PittAssumpScalar} can be rewritten in terms of $\beta$ as
\begin{equation*}
 \max\left\{ 0, d \left(1 - \frac{1}{p} - \frac{1}{q}\right)\right\} \leq \beta <\frac{d}{p'}  \quad \text{and} \quad \beta - \gamma = d \Big(1- \frac{1}{p} - \frac{1}{q}\Big).
 \end{equation*}
 Further, since
 \begin{equation*}
 	\max\left\{0, d \left(\frac{1}{p} + \frac{1}{q} - 1\right) \right\} = \max\left\{0, d \left(\frac{1}{\min\{2, p\}} + \frac{1}{q} - 1\right) \right\},
 \end{equation*}
 the conditions on $\gamma$ given in (\ref{PittAssumpScalar}) read as
 	\begin{equation}\label{PittAssumpScalar*}
\max\left\{0, d \left(\frac{1}{\min\{2, p\}} + \frac{1}{q} - 1\right) \right\} \leq \gamma < \frac{d}{q}.
	\end{equation}

By general results for vector-valued extensions the scalar version \eqref{PittScalar} implies that \eqref{PittGeneral} holds for the same set of parameters \eqref{PittAssumpScalar} if $X$ is a Hilbert space (see \cite[Theorem 2.1.9]{HyNeVeWe16}).

Let us briefly explain our strategy to prove Theorem \ref{thm:Pitt2intro}. First we will focus on the extreme case $p=q$ (see Theorem \ref{thm:Pitt}). Then, we shall deal with the case $p < q$. Here, we shall distinguish three possible situations according to the fact that
\begin{equation*}
\max\left\{0, d \left(\tfrac{1}{\min\{p, p_0\}} + \tfrac{1}{q} - 1\right) \right\} < \gamma < \frac{d}{q}
\end{equation*}
holds if and only if one of the following conditions holds
\begin{enumerate}[\upshape(i)]
\item\label{it:Pittintroi} $p \in (1,p_0] \cup [p_0',\infty)$, and $\max\left\{0, d \left(\frac{1}{p} + \frac{1}{q} - 1\right) \right\} < \gamma < \frac{d}{q}$,
\item\label{it:Pittintroii} $p \in (p_0, p_0')$, $p_0' \leq q$, and $0 < \gamma < \frac{d}{q}$,
\item\label{it:Pittintroiii}  $p \in (p_0, p_0')$, $q < p_0'$, and $d\left(\frac{1}{p_0} + \frac{1}{q} - 1\right) < \gamma < \frac{d}{q} $.
\end{enumerate}
The cases \eqref{it:Pittintroi}, \eqref{it:Pittintroii} and \eqref{it:Pittintroiii} will be shown in Theorems \ref{thm:Pitt2}, \ref{thm:Pitt3} and \ref{thm:Pitt4}, respectively.

\begin{theorem}\label{thm:Pitt}
	Let $X$ be of Fourier type $p_0 \in (1,2]$. Assume that
	\begin{equation*}
	1 < q < \infty, \ \max\left\{0, d \left(\frac{1}{\min\{p_0,q\}} + \frac{1}{q} -1 \right) \right\} < \gamma < \frac{d}{q} \quad \text{and} \quad \beta - \gamma = d - \frac{2 d}{q}.
	\end{equation*}
	Then,
	\begin{equation}\label{PittVectorThm}
	\|\widehat{f}\|_{L^q(\R^d, |\cdot|^{-\gamma q};X)} \leq C \|f\|_{L^q(\R^d, |\cdot|^{\beta q};X)}.
\end{equation}
\end{theorem}

\begin{remark}
(i) We will see in Remark \ref{rem:SharpnessPitt} below that (\ref{PittVectorThm}) may be false in the extreme case $\gamma =  \max\left\{0, d \left(\frac{1}{\min\{p_0,q\}} + \frac{1}{q} -1 \right) \right\}$.

(ii) In the special case that $X = \C$ in Theorem \ref{thm:Pitt} ($p_0=2$) we recover the classical inequality \eqref{PittScalar} for $p=q$ under the assumptions \eqref{PittAssumpScalar}.
\end{remark}

\begin{proof}[Proof of Theorem \ref{thm:Pitt}]
The proof is divided into a number of steps.

\textsc{Step 1:} First we shall prove the following result:  Let $1 < p < p_0$. Then,
	\begin{equation}\label{aux7}
	\|\widehat{f}\|_{L^q(\R^d, |\cdot|^{-\gamma q};X)} \leq C \|f\|_{L^q(\R^d, |\cdot|^{\beta q};X)}
\end{equation}
where
\begin{equation}\label{aux8}
	q \in (p, p'), \quad \beta = d \Big(\frac{1}{p}-\frac{1}{q}\Big),  \quad \text{and}  \quad \gamma = d \Big(\frac{1}{q}-\frac{1}{p'}\Big).
\end{equation}

By Proposition \ref{prop:WeakFourier}\eqref{it:weakHtype}, we have
	\begin{equation*}
		\F : L^p(\R^d;X) \to L^p(\R^d, |\cdot|^{-d(2-p)};X) \ \ \text{and} \ \ \F: L^{p'}(\R^d, |\cdot|^{d(p'-2)};X) \to L^{p'}(\R^d;X).
	\end{equation*}
	Let $p < q < p'$. Then, we define $\theta \in (0,1)$ such that $1/q = (1-\theta)/p + \theta/p'$ and apply complex interpolation and Lemma \ref{lem:interpolationWeighted} to get
that $\F$ maps the space
	\begin{equation*}
		[ L^p(\R^d;X), L^{p'}(\R^d, |\cdot|^{d(p'-2)};X)]_\theta = L^q(\R^d, |\cdot|^{\beta q};X)
	\end{equation*}
boundedly into
	\begin{equation*}
		[L^p(\R^d, |\cdot|^{-d(2-p)};X)  , L^{p'}(\R^d;X)]_\theta = L^q(\R^d, |\cdot|^{-\gamma q};X).
	\end{equation*}
Hence, we obtain the desired estimate \eqref{aux7} with \eqref{aux8}.

\textsc{Step 2:} Suppose that $p_0 \leq q \leq p_0'$. According to Step 1, the inequality \eqref{PittVectorThm} holds true if $ \beta -\gamma=  d -\frac{2 d}{q}$ and, in addition,
\begin{equation*}
	 \gamma = d \left(\frac{1}{q} - \frac{1}{p'}\right) \quad \text{for any} \quad 1 < p < p_0,
\end{equation*}
or equivalently,
\begin{equation*}
	\max\left\{0, d \left(\frac{1}{\min\{p_0,q\}} + \frac{1}{q} -1 \right) \right\} = d\left(\frac{1}{q} - \frac{1}{p_0'}\right) < \gamma < \frac{d}{q}.
\end{equation*}

\textsc{Step 3:}  We treat the case $q < p_0$. Applying Step 1 we get the inequality \eqref{PittVectorThm} with the constraint $ \beta -\gamma=  d -\frac{2 d}{q}$ for any
\begin{equation*}
	 \gamma = d \left(\frac{1}{q} - \frac{1}{p'}\right) \quad \text{with} \quad 1 < p < q.
\end{equation*}
Note that the latter condition can be written as
\begin{equation*}
	\max\left\{0, d \left(\frac{1}{\min\{p_0,q\}} + \frac{1}{q} -1 \right) \right\}  = d \left(\frac{1}{q} - \frac{1}{q'}\right)< \gamma < \frac{d}{q} \quad \text{and} \quad q' = \max\{p_0',q,q'\}.
\end{equation*}

\textsc{Step 4:}  We deal with the case $q > p_0'$. Take any $1 < p < q'$. Note that, in particular, $p < p_0$. Then, by Step 1, we obtain the desired inequality under the assumptions $ \beta -\gamma=  d -\frac{2 d}{q}$ and $\gamma = d \left(\frac{1}{q} - \frac{1}{p'}\right)$. Since $q = \max\{p_0',q,q'\}$, a simple change of variables yields the desired result if $\beta -\gamma=  d -\frac{2 d}{q}$ and $\max\left\{0, d \left(\frac{1}{\min\{p_0,q\}} + \frac{1}{q} -1 \right) \right\} =0 < \gamma < \frac{d}{q}$.

\end{proof}

\begin{remark}\label{rem:SharpnessPitt}
	 In sharp contrast with the scalar and Hilbert space case (i.e. $p_0 = 2$), working with a Banach space $X$ which is not isomorphic to a Hilbert space (i.e.\ $p_0\in [1, 2)$) the inequality \eqref{PittVectorThm} may not be true in the extreme case $\gamma = \max\left\{0, d \left(\frac{1}{\min\{p_0,q\}} + \frac{1}{q} -1 \right) \right\}$. As we will show below, its validity depends on the relationships between $q$ and the Fourier type $p_0$ of $X$. Let us distinguish three possible cases:

{\em Case $q < p_0$:} Thus, $\gamma = d \left(\frac{1}{q}-\frac{1}{q'}\right) = d\left(\frac{2}{q}-1\right)$, and $\beta = 0$. Note that \eqref{PittVectorThm} holds true because $X$ has HL type $q$ (see Proposition \ref{prop:WeakFourier}\eqref{it:weakHtype}).

{\em Case $q > p_0'$:} Thus, $\gamma = 0$, and $\beta = d \left(1-\frac{2}{q}\right) $. Note that \eqref{PittVectorThm} holds true because $X$ has HL cotype $q$ (see Proposition \ref{prop:WeakFourier}\eqref{it:weakHtype}).

{\em Case $q \in [p_0, p_0'], \, p_0 \neq 2$:} Then the extreme case $\gamma = \max\left\{0, d \left(\frac{1}{\min\{p_0,q\}} + \frac{1}{q} -1 \right) \right\}  = d\Big(\frac{1}{q} - \frac{1}{p_0'}\Big)$ cannot be attained. Assume first that $q = p_0 \neq 2$. Therefore, the validity of the inequality \eqref{PittVectorThm} (with $\gamma = d (2/p_0 - 1)$ and so, $\beta = 0$) is equivalent to saying that $X$ has HL type $p_0$ whenever $X$ has Fourier type $p_0$. As it was shown in Example \ref{ex:Paley}, this latter statement is not true. Suppose now that the inequality \eqref{PittVectorThm} holds for $q=p_0' \neq 2$ with $\gamma = 0$ (and so, $\beta = d(1-2/p_0')$), which means that $X$ has HL cotype $p_0'$ whenever $X$ has Fourier type $p_0$. This is not true as can be seen in Example \ref{ex:PaleyCotype}. Finally, we deal with the case $q \in (p_0, p_0')$ and $\gamma = d\Big(\frac{1}{p_0} + \frac{1}{q} - 1\Big)$ (and thus, $\beta = d(1/p_0 -1/q)$). Under these assumptions on the parameters, the inequality \eqref{PittVectorThm} is not true. Indeed, let $X = \ell{^{p_0'}}(\Z^d)$ and
	\begin{equation*}
		f(t) = (|n|^{-d/p_0'} (1 + \log |n|)^{-\varepsilon} e^{2 \pi i n\cdot t} \one_{n\neq 0}), \quad t \in \T^d,
	\end{equation*}
	where $1/p_0' < \varepsilon < 1/q$. Clearly, we have $\|(\widehat{f}(n))\|_X = |n|^{-d/p_0'} (1 + \log |n|)^{-\varepsilon}$ for $n \neq 0$. Then,
	\begin{equation*}
		\|f\|_{L^q(\T^d, |\cdot|^{(1/p_0-1/q)dq};X)} \asymp \left(\sum_{n\neq 0} (1 + \log |n|)^{-\varepsilon p_0'} \frac{1}{|n|^d}\right)^{1/p_0'} < \infty
	\end{equation*}
	and
	\begin{equation*}
		\|(\widehat{f}(n))\|_{\ell^q(\Z^d, (1+|n|)^{-(1/p_0 +1/q-1) dq};X)} \eqsim \left(\sum_{n\neq 0} (1 + \log |n|)^{-\varepsilon q} \frac{1}{|n|^d}\right)^{1/q} = \infty.
	\end{equation*}
	Finally, by Proposition \ref{prop:WeightedTrans} we derive that the inequality \eqref{PittVectorThm} is not satisfied.
	\end{remark}

Next we study the Pitt's inequality \eqref{PittGeneral} for integrability parameters $p$ and $q$ with $p < q$. Our first result in this direction reads as follows.

\begin{theorem}\label{thm:Pitt2}
	Let $X$ be of Fourier type $p_0 \in (1, 2]$. Assume that
		\begin{equation*}
	p \in (1, p_0] \cup [p_0',\infty), \quad p < q < \infty,
	\end{equation*}	
	and
	\begin{equation}\label{thm:Pitt2Assumptions}
	\max\left\{0, d \left(\frac{1}{p} + \frac{1}{q} - 1\right) \right\} \leq \gamma < \frac{d}{q} , \quad \beta - \gamma = d \Big(1- \frac{1}{p} - \frac{1}{q}\Big).
	\end{equation}	
	 Then, we have
		\begin{equation}\label{PittVectorThm2}
	\|\widehat{f}\|_{L^q(\R^d, |\cdot|^{-\gamma q};X)} \leq C \|f\|_{L^p(\R^d, |\cdot|^{\beta p};X)}.
\end{equation}
\end{theorem}

\begin{remark}\
\begin{enumerate}[\upshape(i)]
	\item In case that $X = \C$ in Theorem \ref{thm:Pitt2} (and so, $p_0 = 2$) we get the Pitt's inequalities \eqref{PittScalar} in the full range of parameters \eqref{PittAssumpScalar}. The same holds if $X$ is a Hilbert space.
	\item We will see in Remark \ref{rem:SharpnessPitt2} below that the assumption $p \in (1, p_0] \cup [p_0',\infty)$ is indeed necessary in Theorem \ref{thm:Pitt2}.
	\end{enumerate}
\end{remark}

\begin{proof}[Proof of Theorem \ref{thm:Pitt2}]
	We shall break up the proof into several cases.
	
	\textsc{Case 1:} $p \leq p_0$ and $p '\leq q$. By Lemma \ref{lem:SobolevEmb} applied with $u=p'$, $v = q$, $a = \gamma$, $b =0$, $\lambda =d - \beta$ and to the function $g = (-\Delta)^{\beta/2} \wh f$ (see \cite[Section 6.1]{Grafakos09}) we have that
	\begin{align*}
		\|\wh f\|_{L^q(\R^d, |\cdot|^{-\gamma q};X)}  & = c_{d,\beta}\||\cdot|^{-\lambda}*g\|_{L^q(\R^d, |\cdot|^{-\gamma q};X)}
\\ & \leq C\|g\|_{L^{p'}(\R^d;X)}
\leq C \|\F^{-1} g \|_{L^{p}(\R^d;X)} = (2\pi)^{\beta} C \|f\|_{L^p(\R^d, |\cdot|^{\beta p};X)}
	\end{align*}
	where we have also used that $X$ is of Fourier type $p \leq p_0$ (see Proposition \ref{prop:WeakFourier}).

		\textsc{Case 2:} $p \leq p_0$ and $p < q < p'$. Using Case 1, we know that
			\begin{equation}\label{3.25}
	\|\widehat{f}\|_{L^{p'}(\R^d, |\cdot|^{-\beta p'};X)} \leq C \|f\|_{L^p(\R^d, |\cdot|^{\beta p};X)}
\end{equation}
for any $0 \leq \beta < d/p'$. On the other hand, by Theorem \ref{thm:Pitt}, we have
		\begin{equation}\label{3.26}
	\|\widehat{f}\|_{L^{p}(\R^d, |\cdot|^{-\mu p};X)} \leq C \|f\|_{L^p(\R^d, |\cdot|^{\beta p};X)}
\end{equation}
whenever $0 \leq \beta < d/p'$ and $\beta - \mu = d(1-2/p)$. Choose $\theta \in (0,1)$, so that $1/q = (1-\theta)/p' + \theta/p$ and we interpolate between \eqref{3.25} and \eqref{3.26} (see Lemma \ref{lem:interpolationWeighted}) in order to get that
\begin{equation}
	\|\widehat{f}\|_{L^{q}(\R^d, |\cdot|^{-\gamma q};X)} \leq C \|f\|_{L^p(\R^d, |\cdot|^{\beta p};X)}
\end{equation}
with $\beta - \gamma = d(1 - 1/p -1/q)$ and $0 \leq \beta < d/p'$ (or equivalently, $d (1/p + 1/q -1) = \max\left\{0, d \left(1/p + 1/q-1 \right) \right\} \leq \gamma < d/q$).

Cases 1 and 2 together give \eqref{PittVectorThm2} for $p \leq p_0$ and $p < q$.

\textsc{Case 3:} $p_0' \leq p < q$. We will apply duality from the previous cases to extend the inequality  \eqref{PittVectorThm2} to the range of values $p_0' \leq p < q$. By assumptions $X^\ast$ has Fourier cotype $p_0'$ (see Proposition \ref{prop:duality}\eqref{DualF}) and $q' < p' \leq p_0$, and therefore
\begin{equation}\label{3.28}
	\F: L^{q'}(\R^d, |\cdot|^{\gamma q'};X^\ast) \to L^{p'}(\R^d, |\cdot|^{-\beta p'};X^\ast)
\end{equation}
with $d (1/p' +1/q'-1) = \max\{0, d(1/p' + 1/q' -1)\} \leq\beta < d/p'$ and $-\beta + \gamma = d(1-1/p'-1/q')$. Let $f:\R^d \to X$ and $g:\R^d \to X^\ast$ be simple functions. Then, applying H\"older's inequality, \eqref{3.28} and \eqref{FourierCom}, we get
\begin{align*}
	\Big|\int_{\R^d} \lb \widehat{f}(x),g(x)\rb  \, \ud x\Big| &  \leq \int_{\R^d} \|f(x)\|_X \|\widehat{g}(x)\|_{X^\ast} \, \ud x \\
	& \leq \|f\|_{L^p(\R^d, |\cdot|^{\beta p};X)} \|\widehat{g}\|_{L^{p'}(\R^d, |\cdot|^{-\beta p'};X^\ast)} \\
	& \leq C  \|f\|_{L^p(\R^d, |\cdot|^{\beta p};X)}  \|g\|_{L^{q'}(\R^d, |\cdot|^{\gamma q'};X^\ast)}.
\end{align*}
By taking the supremum over all simple functions $g:\R^d \to X^\ast$ such that $ \|g\|_{L^{q'}(\R^d, |\cdot|^{\gamma q'};X^\ast)} \leq 1$ and applying Lemma \ref{lem:duality0}, we derive at
\begin{equation*}
	\|\widehat{f}\|_{L^q(\R^d, |\cdot|^{-\gamma q}, X)} \leq C  \|f\|_{L^p(\R^d, |\cdot|^{\beta p};X)}
\end{equation*}
whenever $0 = \max\{0, d(1/p + 1/q -1)\} \leq \gamma < d/q$ and $\beta - \gamma = d (1 - 1/p -1/q)$.

\end{proof}

It turns out that the validity of the Pitt's inequality \eqref{PittVectorThm2} for a Banach space $X$ of Fourier type $p_0$ becomes more intricate when $p \in (p_0, p_0')$. More precisely, the ranges of $\gamma$ and $\beta$ for which \eqref{PittVectorThm2} holds will depend on the relations between the integrability parameter $q$ and $p_0'$. Our next result treats the case $p_0' \leq q$.

\begin{theorem}\label{thm:Pitt3}
	Let $X$ be of Fourier type $p_0 \in (1, 2]$. Assume that
		\begin{equation}\label{thm:Pitt3Assumptions}
	p \in (p_0, p_0'), \quad p_0' \leq q < \infty,
	\end{equation}	
	and
	\begin{equation}\label{thm:Pitt3Assumptions2}
	0 \leq \gamma < \frac{d}{q} , \quad \beta - \gamma = d \Big(1- \frac{1}{p} - \frac{1}{q}\Big).
	\end{equation}	
	 Then, we have
		\begin{equation*}
	\|\widehat{f}\|_{L^q(\R^d, |\cdot|^{-\gamma q};X)} \leq C \|f\|_{L^p(\R^d, |\cdot|^{\beta p};X)}.
\end{equation*}
\end{theorem}
\begin{proof}
	By assumptions and Proposition \ref{prop:duality} we have that $X^\ast$ has Fourier type $p_0$ and
	\begin{equation*}
	 q' \leq p_0 < p', \quad 0 < d \left(\frac{1}{q'} + \frac{1}{p'} -1\right)\leq \beta < \frac{d}{p'}, \quad \gamma - \beta = d \left(1 - \frac{1}{q'} - \frac{1}{p'}\right).
	 \end{equation*}
	 Hence, applying Theorem \ref{thm:Pitt2} we obtain
		\begin{equation*}
	\|\widehat{f}\|_{L^{p'}(\R^d, |\cdot|^{-\beta p'};X^\ast)} \leq C \|f\|_{L^{q'}(\R^d, |\cdot|^{\gamma q'};X^\ast)}.
\end{equation*}
Dualising with $L^p(\R^d, |\cdot|^{\beta p};X)$ and $L^q(\R^d, |\cdot|^{-\gamma q};X)$ (see Lemma \ref{lem:duality0}), this is seen to be equivalent to the following estimate
\begin{equation*}
	\|\widehat{f}\|_{L^q(\R^d, |\cdot|^{-\gamma q};X)} \leq C \|f\|_{L^p(\R^d, |\cdot|^{\beta p};X)}.
\end{equation*}
\end{proof}

It remains to investigate the Pitt's inequality
	\begin{equation}\label{PittVectorThm2New}
	\|\widehat{f}\|_{L^q(\R^d, |\cdot|^{-\gamma q};X)} \leq C \|f\|_{L^p(\R^d, |\cdot|^{\beta p};X)}
\end{equation}
 for
 \begin{equation*}
  p \in (p_0,p_0') \quad \text{and} \quad p < q < p_0'.
 \end{equation*}
   In such a situation, we will see that $p_0$ plays a key role. In fact, we will show in Theorem \ref{thm:Pitt4} below that \eqref{PittVectorThm2New} holds if $d\left(\frac{1}{p_0} + \frac{1}{q} - 1\right)  < \gamma < \frac{d}{q}$. Furthermore, we will prove in Remark \ref{rem:SharpnessPitt2} below that this result is sharp, i.e., \eqref{PittVectorThm2New} fails to be true if $\gamma \leq d\left(\frac{1}{p_0} + \frac{1}{q} - 1\right)$.

\begin{theorem}\label{thm:Pitt4}
	Let $X$ be of Fourier type $p_0 \in (1, 2]$. Assume that
		\begin{equation*}
	p \in (p_0, p_0'), \quad p < q < p_0',
	\end{equation*}	
	and
	\begin{equation}\label{thm:Pitt4Assumptions2}
	d\left(\frac{1}{p_0} + \frac{1}{q} - 1 \right) < \gamma < \frac{d}{q} , \quad \beta - \gamma = d \Big(1- \frac{1}{p} - \frac{1}{q}\Big).
	\end{equation}	
	 Then, we have
		\begin{equation}\label{aux13}
	\|\widehat{f}\|_{L^q(\R^d, |\cdot|^{-\gamma q};X)} \leq C \|f\|_{L^p(\R^d, |\cdot|^{\beta p};X)}.
\end{equation}
\end{theorem}
\begin{proof}

Let $0< \eta < 1$ be arbitrary and define
	\begin{equation*}
		\gamma_0 = (1 - \eta) d \left(\frac{1}{p_0} +\frac{1}{p} -  1 \right) + \frac{\eta d}{p}  \quad \text{and} \quad \gamma_1 = \frac{\eta d}{p_0'}.
	\end{equation*}
	In particular, we have $0 < d \left(\frac{1}{p_0} +  \frac{1}{p} -1 \right) < \gamma_0 < \frac{d}{p}$ and $0 < \gamma_1 < \frac{d}{p_0'}$. Applying Theorem \ref{thm:Pitt} we get
	\begin{equation}\label{aux8b}
		\F : L^p(\R^d, |\cdot|^{\beta_0 p};X) \to L^p(\R^d, |\cdot|^{-\gamma_0 p};X), \quad \beta_0 - \gamma_0 = d \left(1 - \frac{2}{p}\right) .
	\end{equation}
	On the other hand, invoking Theorem \ref{thm:Pitt3} we derive
		\begin{equation}\label{aux9}
		\F : L^p(\R^d, |\cdot|^{\beta_1 p};X) \to L^{p_0'}(\R^d, |\cdot|^{-\gamma_1 p_0'};X), \quad \beta_1 - \gamma_1 = d \left(1 - \frac{1}{p} - \frac{1}{p_0'}\right) .
	\end{equation}
	Since $q \in (p,p_0')$ there exists $\theta \in (0,1)$ such that $\frac{1}{q} = \frac{1-\theta}{p} + \frac{\theta}{p_0'}$. Then, applying the interpolation property of the complex interpolation method to \eqref{aux8b} and \eqref{aux9} and Lemma \ref{lem:interpolationWeighted}, one gets that
$\F$ maps
\begin{equation}\label{aux10}
[L^p(\R^d, |\cdot|^{\beta_0 p};X) , L^p(\R^d, |\cdot|^{\beta_1 p};X)]_\theta  = L^p(\R^d, |\cdot|^{\beta p};X)
\end{equation}
into the space
\begin{equation}\label{aux12}
[ L^p(\R^d, |\cdot|^{-\gamma_0 p};X) , L^{p_0'}(\R^d, |\cdot|^{-\gamma_1 p_0'};X)]_\theta = L^q(\R^d, |\cdot|^{-\gamma q};X),
	\end{equation}
where $\beta = (1-\theta)\beta_0+ \theta \beta_1$ and $\gamma= (1-\theta)\gamma_0 + \theta \gamma_1$. Then the second part of \eqref{thm:Pitt4Assumptions2} holds. Moreover, we have that $\gamma = (1-\theta)\gamma_0 + \theta \gamma_1$ satisfies
\begin{align*}
\gamma & \nearrow (1-\theta) \frac{d}{p} +   \theta \frac{d}{p_0'} = \frac{d}{q} & \text{as $\eta \nearrow1$},
\\ \gamma & \searrow (1-\theta) d \left(\frac{1}{p_0} +  \frac{1}{p} - 1\right) = d\left(\frac{1}{p_0} + \frac{1}{q} - 1\right)  & \text{as $\eta \searrow0$}.
\end{align*}
This completes the proof of the result for $\gamma$ near the endpoints as we can take $\eta$ close to one and zero respectively.
\end{proof}

Keeping notation from Theorem \ref{thm:Pitt4}, we now construct a counterexample showing that \eqref{aux13} is far from being true if $p_0 < p < q < p_0'$, and $\gamma \leq d\left(\frac{1}{p_0} + \frac{1}{q} - 1 \right)$.

\begin{remark}\label{rem:SharpnessPitt2}
	In general, Theorem \ref{thm:Pitt4} is not true if $\gamma \leq d\left(\frac{1}{p_0} + \frac{1}{q} - 1 \right)$. Let $p \in (p_0, p_0')$ and $p \leq q < p_0'$. Suppose furthermore that $\gamma \leq d\left(\frac{1}{p_0} + \frac{1}{q} - 1 \right)$ and $ \beta - \gamma = d \Big(1- \frac{1}{p} - \frac{1}{q}\Big)$. Then, we claim that
		\begin{equation}\label{rem:SharpnessPitt2.1}
		 \F : L^p(\R^d, |\cdot|^{\beta p};X) \to L^q(\R^d, |\cdot|^{-\gamma q};X) \quad \text{does not hold true}
		\end{equation}
		for all $X$ having Fourier type $p_0$. Indeed, put $X = \ell{^{p_0'}}(\Z^d)$ and define $f:\T^d\to X$ by
	\begin{equation}\label{aux14}
		f(t) = (a_n e^{2 \pi i n\cdot t}), \quad t \in \T^d,
	\end{equation}
	where $(a_n)_{n \in \Z^d}$ is a scalar sequence to be specified below. Firstly, if $\gamma < d\left(\frac{1}{p_0} + \frac{1}{q} - 1 \right)$ then we set
	\begin{equation*}
	 a_n = |n|^{-\varepsilon} \one_{n\neq 0}
	 \end{equation*}
	 where $d \left(1 - \frac{1}{p_0}\right) < \varepsilon < -\gamma + \frac{d}{q}$.
	  We have
	\begin{equation*}
		\|f\|_{L^p(\T^d, |\cdot|^{\beta p}; \ell{^{p_0'}}(\Z^d))} \asymp \left(\sum_{n\neq 0} |n|^{-\varepsilon p_0'}\right)^{1/p_0'} < \infty.
	\end{equation*}
	On the other hand,
	\begin{equation*}
		\|(\widehat{f}(n))\|_{\ell^q(\Z^d, (|n|+1)^{-\gamma q}; \ell{^{p_0'}}(\Z^d))} \eqsim \left(\sum_{n\neq 0} |n|^{-(\gamma +  \varepsilon) q} \right)^{1/q} = \infty.
	\end{equation*}
	Suppose now that $\gamma = d\left(\frac{1}{p_0} + \frac{1}{q} - 1 \right)$. A counterexample is given by the function $f$ (see \eqref{aux14}) with
	\begin{equation*}
	 a_n = |n|^{-d (1 - 1/p_0)} (1 + \log |n|)^{-\eta} \one_{n\neq 0}, \quad 1/p_0' < \eta < 1/q,
	 \end{equation*}
	 because
	 \begin{equation*}
		\|f\|_{L^p(\T^d, |\cdot|^{\beta p}; \ell{^{p_0'}}(\Z^d))} \asymp \left(\sum_{n\neq 0}(1 + \log |n|)^{-\eta p_0'} \frac{1}{|n|^d}\right)^{1/p_0'} < \infty
	\end{equation*}
	and
	\begin{equation*}
		\|(\widehat{f}(n))\|_{\ell^q(\Z^d, (|n|+1)^{-\gamma q}; \ell{^{p_0'}}(\Z^d))} = \left(\sum_{n\neq 0} (1 + \log |n|)^{-\eta q} \frac{1}{|n|^d} \right)^{1/q} = \infty.
	\end{equation*}
	According to Proposition \ref{prop:WeightedTrans}, we have shown the desired assertion \eqref{rem:SharpnessPitt2.1}.
\end{remark}

Using Proposition \ref{prop:WeightedTrans}, we get the periodic counterpart of Pitt's inequality obtained in Theorem \ref{thm:Pitt2intro}.

\begin{corollary}\label{cor:Pitt2Periodic}
Let $X$ be of Fourier type $p_0 \in (1, 2]$. Let $1 < p \leq q < \infty$ and $\beta, \gamma \geq 0$. Assume that
\begin{equation*}
\max\left\{0, d \left(\tfrac{1}{\min\{p, p_0\}} + \tfrac{1}{q} - 1\right) \right\} < \gamma < \tfrac{d}{q} \quad \text{and} \quad \beta - \gamma = d \Big(1- \tfrac{1}{p} - \tfrac{1}{q}\Big).
\end{equation*}
	 Then, we have
		\begin{equation}\label{cor:Pitt2Periodic*}
	\Big\|t\mapsto \sum_{n \in \Z^d} e^{2\pi i t \cdot n} x_n \Big\|_{L^q(\T^d, |\cdot|^{-\gamma q};X)} \leq C \|(x_n)\|_{\ell^p(\Z^d, (|n| + 1)^{\beta p};X)}
\end{equation}
and
	\begin{equation}\label{cor:Pitt2Periodic**}
	\|(\widehat{f}(n))\|_{\ell^q(\Z^d, (|n| + 1)^{-\gamma q};X)} \leq C \|f\|_{L^p(\T^d, |\cdot|^{\beta p};X)}.
\end{equation}
Furthermore, if $\gamma = \max\left\{0, d \left(\tfrac{1}{\min\{p, p_0\}} + \tfrac{1}{q} - 1\right) \right\}$ then \eqref{cor:Pitt2Periodic*} and \eqref{cor:Pitt2Periodic**} hold true if one of the following conditions is satisfied:
\begin{enumerate}[\upshape(i)]
\item \label{thm:Pitt2Periodic3} $p = q \in (1, p_0) \cup (p_0',\infty), \quad p_0 \neq 2$;
\item \label{thm:Pitt2Periodic3*} $p=q \in (1, \infty), \quad p_0=2$;
\item \label{thm:Pitt2Periodic4} $p < q, \quad p \in (1, p_0] \cup [p_0',\infty)$;
\item \label{thm:Pitt2Periodic5}$p < q, \quad p \in (p_0, p_0'), \quad p_0' \leq q$.
\end{enumerate}
If none of the conditions \eqref{thm:Pitt2Periodic3}-\eqref{thm:Pitt2Periodic5} holds, then \eqref{cor:Pitt2Periodic*} and \eqref{cor:Pitt2Periodic**} fail to be true.
\end{corollary}

\begin{remark}
In \cite{Stein56} classical Pitt's inequality has been extended to other orthonormal systems. It would be interesting to prove a version of Corollary \ref{cor:Pitt2Periodic} in this generality as well. Note however that in the vector-valued situation it seems unclear how this is connected to the Fourier type of the space $X$ (see \cite[Section 8.3]{GaKaKoTo98} and \cite[Section 6.5]{Pietsch-Wenzel98}).
\end{remark}

\section{Connections to Rademacher type and cotype}\label{sec:Radtype}

Let $(\varepsilon_{n})_{n\in \N}$ be a sequence of independent random variables on a probability space $(\Omega,\mathbb{P})$ such that each $\varepsilon_n$ is uniformly distributed on $\{z\in\C\mid \abs{z}=1\}$. Such random variables are called \emph{(complex) Rademacher or Steinhaus random variables}. Equivalently, one can use the usual real Rademacher random variables below (see \cite{HyNeVeWe2} for details).

Let $(\varepsilon_{n})_{n\in\N}$ be a Rademacher sequence on a probability space $(\Omega,\mathbb{P})$. For $p\in[1,2]$, $X$ is said to have \emph{type} $p$ if there exists a constant $C\geq 0$ such that for all finite choices $x_{1},\ldots, x_{m}\in X$,
\begin{align}\label{eq:type}
\Big(\mathbb{E}\Big\|\sum_{n=1}^{m} \varepsilon_{n}x_{n}\Big\|_X^{2}\Big)^{1/2}\leq C\Big(\sum_{n=1}^{m}\|x_{n}\|_X^{p}\Big)^{1/p}.
\end{align}
For $q\in[2,\infty]$, $X$ is said to have \emph{cotype} $q$ if there exists a constant $C\geq 0$ such that for all finite choices $x_{1},\ldots, x_{m}\in X$,
\begin{align}\label{eq:cotype}
\Big(\sum_{n=1}^{m}\|x_{n}\|_X^{q}\Big)^{1/q}\leq C\Big(\mathbb{E}\Big\|\sum_{n=1}^{m}
\varepsilon_{n}x_{n}\Big\|_X^{2}\Big)^{1/2},
\end{align}
with the obvious modification for $q=\infty$. By the Kahane-Khintchine inequalities (see, e.g., \cite[Theorem 11.1]{DiJaTo95} and \cite[Theorem 6.2.4]{HyNeVeWe2}) the $L^2(\Omega;X)$ norm (with the exception of $q=\infty$) in \eqref{eq:type} and \eqref{eq:cotype} can be replaced by any $L^r(\Omega;X)$-norm with $r\in [1, \infty)$.

We further say that $X$ has \emph{nontrivial type} if $X$ has type $p$ for some $p\in(1,2]$, and \emph{finite cotype} if $X$ has cotype $q$ for some $q\in[2,\infty)$.

Each Banach space $X$ has type $p=1$ and cotype $q=\infty$. If $X$ has type $p$ and cotype $q$ then $X$ has type $r$ for all $r\in[1,p]$ and cotype $s$ for all $s\in[q,\infty]$. A Banach space $X$ is isomorphic to a Hilbert space if and only if $X$ has type $p=2$ and cotype $q=2$, by Kwapie\'n's theorem (see \cite[Theorem 7.4.1]{Albiac-Kalton06}).

Let $X$ be a Banach space, $r\in[1,\infty)$ and let $(S, \Sigma, \mu)$ be a measure space. If $X$ has type $p\in(1,2]$ and cotype $q\in[2,\infty)$ then $\Ellr(S;X)$ has type $\min\{p,r\}$ and cotype $\max\{q,r\}$ (see \cite[Theorem 11.12]{DiJaTo95}).

By a result of Bourgain \cite{Bourg88},
\begin{equation}\label{B}
 X \ \text{has nontrivial type if and only if} \ X \ \text{has nontrivial Fourier type}
 \end{equation}
(see also \cite[Section 5.6.30]{Pietsch-Wenzel98}).

Further details on type and cotype can be found in \cite{Albiac-Kalton06}, \cite{DiJaTo95}, \cite{HyNeVeWe2}, \cite[Section 9.2]{Lindenstrauss-Tzafriri79} and \cite{Pietsch-Wenzel98}.

\subsection{Connections with Hardy--Littlewood type}

A Banach space with Fourier type $p\in[1,2]$ has type $p$ and cotype $p'$ (see \cite[Proposition 7.3.6]{HyNeVeWe2}). We show in Theorem \ref{thm:HLimpliestypecotype} below that for HL type and cotype a variant of this result holds. At first sight it might be surprising that unlike in the Fourier (co)type case one has an improvement on the type or cotype.
\begin{theorem}\label{thm:HLimpliestypecotype}
Let $p_0\in (1, 2)$ and $q_0\in (2, \infty)$. Then the following assertions hold:
\begin{enumerate}[\upshape(i)]
\item\label{it:HLimpliestypecotype1} If $X$ has HL cotype $q_0$, then $X$ has cotype $q_0$ and type $p$ for some $p\in (q_0', 2]$.
\item\label{it:HLimpliestypecotype2} If $X$ has HL type $p_0$, then $X$ has type $p_0$ and cotype $q$ for some $q\in [2, p_0')$.
\end{enumerate}
\end{theorem}

\begin{remark}
	From the fact that $X$ has HL cotype $q_0$ it does not follow that $X$ has cotype less than $q_0$. This is illustrated by the following example: Let $X = \ell^{p}(\Z^d), \,  2 < p < \infty$. We have that $X$ has HL cotype $p$ (see Example \ref{ex:PaleyCotype2}) and cotype exactly $p$ (see \cite[Proposition 5.10]{GaKaKoTo98}). The same observation applies to type in the assertion \eqref{it:HLimpliestypecotype2} (as an example, we can take $X= \ell^p(\Z^d),\, 1 < p < 2$. See Example \ref{ex:Paley2}).
\end{remark}

For the proof of Theorem \ref{thm:HLimpliestypecotype} we use a result of Pisier. For $p\in [1, \infty]$ and $\lambda\geq 1$, we say that $X$ {\em contains $\ell^p_N$'s $\lambda$-uniformly} if for every integer $N\geq 1$ there exists a bounded linear mapping $T:\ell^p_N\to X$ such that
\[\lambda^{-1} \|a\|_{\ell^p_N}\leq \|Ta\|_X\leq \lambda\|a\|_{\ell^p_N}, \ \ a\in \ell^p_N.\]
The next result follows from \cite[Proposition 4.4, Theorem 4.5 and Corollary 4.7]{Pis1206}:
\begin{lemma}\label{lem:Pisiertype}
Let $p_0\in [1, 2)$ and $\lambda\geq 1$. Then the following assertions are equivalent:
\begin{enumerate}
\item $X$ does not contain $\ell^{p_0}_N$'s $\lambda$-uniformly.
\item There exists $p\in (p_0,2]$ such that $X$ has type $p$.
\end{enumerate}
\end{lemma}

\begin{proof}[Proof of Theorem \ref{thm:HLimpliestypecotype}]
\eqref{it:HLimpliestypecotype1}: Let us first prove that $X$ has cotype $q_0$. By Corollary \ref{cor:HLCoTrans} applied with $f(t) = \sum_{n=1}^N \varepsilon_n x_n e^{2\pi i t \cdot n}$  we obtain

\begin{align*}
\Big(\sum_{n=1}^N \|x_n\|_X^{q_0}\Big)^{1/q_0} & \leq C \|f\|_{L^{q_0}(\T^d, |\cdot|^{d(q_0-2)};X)}
\\ & \leq C \|f\|_{L^{q_0}(\T^d;X)} =  C \Big\|t\mapsto \sum_{n=1}^N \varepsilon_n x_n e^{2\pi i t \cdot n}\Big\|_{L^{q_0}(\T^d;X)}.
\end{align*}
Taking $L^{q_0}(\Omega)$-norms on both sides and using that for fixed $t\in \T^d$, $(\varepsilon_n e^{2\pi i t \cdot n})_{n=1}^N$ is identically distributed with $(\varepsilon_n)_{n=1}^N$ we find
\[\Big(\sum_{n=1}^N \|x_n\|_X^{q_0}\Big)^{1/q_0} \leq C \Big\|\sum_{n=1}^N \varepsilon_n x_n\Big\|_{L^{q_0}(\Omega;X)}. \]

Next we show that $X$ has type $p$ for some $p\in (q_0',2]$. Assume this is not the case. Then by Lemma \ref{lem:Pisiertype} we know that there exists a $\lambda>1$ such that $X$ contains $\ell^{q_0'}_{(2 N)^d}$'s $\lambda$-uniformly. Fix $N\geq 1$ and let $T:\ell^{q_0'}_{(2 N)^d}\to X$ be such that,
\[\|a\|_{\ell^{q_0'}_{(2 N)^d}}\leq \|Ta\|_X\leq \lambda\|a\|_{\ell^{q_0'}_{(2 N)^d}}, \ \ a\in \ell^{q_0'}_{(2 N)^d}.\]
Let $\alpha\in (\tfrac{1}{q_0-1}, 1)$  and define $f:\T^d\to \ell^{q_0'}_{(2 N)^d}$ by
\[f(t) = ((|n|+1)^{-d/q_0'} (\log(|n|+1))^{-\alpha/q_0'} e^{2\pi i t \cdot n})_{1\leq |n|\leq N},\]
where $|n| = \max\{|n_j|: j\in \{1, \ldots, d\}\}$. Since $\wh{T f}(n) = T (\wh{f}(n))$, one can apply Corollary \ref{cor:HLCoTrans} to get
\begin{align*}
\|f\|_{L^{q_0}(\T^d;\ell^{q_0'}_{(2 N)^d})}^{q_{0}} & \leq \|T f\|_{L^{q_0}(\T^d;X)}^{q_0}
\leq C^{q_0} \sum_{1\leq |n|\leq N} \|\wh{T f}(n)\|_X^{q_0} (|n|+1)^{d(q_0-2)}
\\ & \leq C^{q_0} \lambda^{q_0} \sum_{1\leq |n|\leq N} \|\wh{f}(n)\|^{q_0}_{\ell^{q_0'}_{(2 N)^d}} (|n|+1)^{d(q_0-2)}.
\end{align*}
Elementary calculus shows that
\begin{align*}
\|f\|_{L^{q_0}(\T^d;\ell^{q_0'}_{(2 N)^d})}^{q_0} &=\Big(\sum_{1\leq |n|\leq N} \frac{1}{(|n|+1)^d (\log(|n|+1))^{\alpha}}\Big)^{q_0/q_0'}
\\ & \geq \Big(\int_{B_N\setminus B_1} \frac{1}{(|x|+1)^d(\log(|x|+1))^{\alpha}} \, \ud x\Big)^{q_0/q_0'}
\\ & \eqsim_{\alpha, q_0, d}  \log(N+1)^{\frac{(1-\alpha)q_0}{q_0'}},
\end{align*}
where we used polar coordinates to calculate the integral.

On the other hand, since $\|\wh{f}(n)\|^{q_0}_{\ell^{q_0'}_{(2 N)^d}} (|n|+1)^{d(q_0-2)} = \frac{1}{(|n|+1)^{d} (\log(|n|+1))^{\alpha(q_0-1)}}$, we have
\begin{align*}
\sum_{1\leq|n|\leq N} \|\wh{f}(n)\|^{q_0}_{\ell^{q_0'}_{(2 N)^d}} (|n|+1)^{d(q_0-2)} & =  \sum_{1\leq|n|\leq N} \frac{1}{(|n|+1)^{d} (\log(|n|+1))^{\alpha (q_0-1)}}
\\ & \leq \int_{\R^d} \frac{1}{(|x|+1)^{d} (\log(|x|+1))^{\alpha (q_0-1)}} \, \ud x
<\infty.
\end{align*}
Therefore, $\log(N+1)^{\frac{(1-\alpha)q_0}{q_0'}} \lesssim_{\alpha, q_0, d}  C^{q_0} \lambda^{q_0}$. Choosing $N$ large enough, we obtain a contradiction and this finishes the proof.

\eqref{it:HLimpliestypecotype2}: By Proposition \ref{prop:duality} $X^*$ has HL cotype $p_0'$. Therefore, by \eqref{it:HLimpliestypecotype1} $X^*$ has type $p$ for some $p\in (p_0, 2]$. Now a further duality argument (see \cite[Proposition 7.1.13]{HyNeVeWe2}) yields that $X^{**}$ and hence $X$ has cotype $p'$. To prove that $X$ has type $p_0$ one can use a duality argument which is based on Pisier's K-convexity theorem (see \cite[Proposition 7.4.10 and Theorem 7.4.23]{HyNeVeWe2}). However, it is simpler to use the same argument as in the proof of cotype $q_0$ of \eqref{it:HLimpliestypecotype1}. Indeed, Corollary \ref{cor:HLTrans} yields
\begin{align*}
\Big\|t\mapsto \sum_{n=1}^N \varepsilon_n x_n e^{2\pi i t \cdot n}\Big\|_{L^{p_0}(\T^d;X)}& \leq \Big\|t\mapsto \sum_{n=1}^N \varepsilon_n x_n e^{2\pi i t \cdot n}\Big\|_{L^{p_0}(\T^d,|\cdot|^{-d(2-p_0)};X)}
\\ & \leq C
\Big(\sum_{n=1}^N \|x_n\|_X^{p_0}\Big)^{1/p_0}
\end{align*}
and taking $L^{p_0}(\Omega)$-norms the result follows.
\end{proof}

From Bourgain's result \eqref{B} we obtain that type $p$ implies Fourier type $r$ for some $r\in (1,2]$, we obtain the following consequence of Theorem \ref{thm:HLimpliestypecotype}:
\begin{corollary}\label{cor:FouriertypeHL}
Let $p\in (1, 2)$ and $q\in (2, \infty)$.
\begin{enumerate}[\upshape(i)]
\item\label{it:FouriertypeHL1} If $X$ has HL type $p$, then $X$ has Fourier type $r$ for some $r\in (1,2]$.
If moreover $p\in (\tfrac43,2)$, then $X$ has Fourier type $r$ for some $r>\frac{2p}{4-p}$.
\item\label{it:FouriertypeHL2} If $X$ has HL cotype $q$, then $X$ has Fourier type $r$ for some $r\in (1,2]$.
If moreover $q\in [2,4)$, then $X$ has Fourier type $r$ for some $r>\frac{2q}{3q-4}$.
\end{enumerate}
\end{corollary}

\begin{proof}
\eqref{it:FouriertypeHL1}: By Theorem \ref{thm:HLimpliestypecotype} $X$ has type $p$ and cotype $q$ for some $q\in [2,p')$. Therefore, by
\cite[6.1.8.6]{PietschHis} $X$ has Fourier type $r$ whenever $r\in (1,2)$ satisfies $\frac{1}{r} - \frac12 > \frac1p-\frac1{q}$. In particular, this holds for some $r$ with $\frac{1}{r}-\frac{1}{2}<\frac1p-\frac1{p'} = \frac2p - 1$. This does not lead to anything useful except if $p\in (4/3,2]$. Indeed, then we find that $X$ has Fourier type for some $r>\frac{2p}{4-p}$.

\eqref{it:FouriertypeHL2}: This can be proved in the same way or by using duality.
\end{proof}

\subsection{Connections with Pitt's inequality}

As in Theorem \ref{thm:HLimpliestypecotype} one can show that Pitt's inequality implies type and cotype properties:
\begin{theorem}\label{ThmPittImpliesType}
Assume $X$ obeys the Pitt's inequality \eqref{thm:Pitt2intro2} for some $p, q, \beta, \gamma$ satisfying \eqref{thm:Pitt2intro1}. Then
\begin{enumerate}[\upshape(i)]
\item\label{it:FouriertypePitt1} $X$ has type $r$ for any $r\in (1, 2]\cap (1,p)$ such that $\beta<\frac{d}{r} - \frac{d}{p}$.
\item\label{it:FouriertypePitt2} $X$ has cotype $s$ for any $s\in [2, \infty)\cap (q, \infty)$ such that $\gamma<\frac{d}{q} -\frac{d}{s}$.
\end{enumerate}
\end{theorem}
\begin{proof}
\eqref{it:FouriertypePitt1}: Assume $X$ does not have type $r$. Then by Lemma \ref{lem:Pisiertype} we know that there exists a $\lambda>1$ such that $X$ contains $\ell^{r}_{(2 N)^d}$'s $\lambda$-uniformly. Fix $N\geq 1$ and let $T:\ell^{r}_{(2 N)^d}\to X$ be such that,
\[\|a\|_{\ell^{r}_{(2 N)^d}}\leq \|Ta\|_X\leq \lambda\|a\|_{\ell^{r}_{(2 N)^d}}, \ \ a\in \ell^{r}_{(2 N)^d}.\]
Let $f:\T^d\to \ell^{r}_{(2 N)^d}$ be given by
\[f(t) = ((|n|+1)^{-d/r} e^{2\pi i n\cdot t})_{1\leq |n|\leq N},\]
where $|n| = \max\{|n_j|: j\in \{1, \ldots, d\}\}$. Then by Proposition \ref{prop:WeightedTrans}
\begin{align*}
\|f\|_{L^{q}(\T^d,|\cdot|^{-\gamma q};\ell^{r}_{(2 N)^d})} &\leq \|T f\|_{L^{q}(\T^d,|\cdot|^{-\gamma q};X)}
\leq C \Big(\sum_{1 \leq |n| \leq N} \|\wh{T f}(n)\|_X^{p} (|n|+1)^{\beta p}\Big)^{1/p}
\\ & \leq C \lambda \Big(\sum_{1 \leq |n| \leq N} \|\wh{f}(n)\|^{p}_{\ell^{r}_{(2 N)^d}} (|n|+1)^{\beta p}\Big)^{1/p}.
\end{align*}
Note that
\begin{align*}
\|f\|_{L^{q}(\T^d,|\cdot|^{-\gamma q};\ell^{r}_{(2 N)^d})} &\geq \|f\|_{L^{q}(\T^d;\ell^{r}_{(2 N)^d})}
= \Big(\sum_{1\leq |n|\leq N} \frac{1}{(|n|+1)^d }\Big)^{1/r}=K_N
\end{align*}
and $K_N$ is unbounded for $N\to \infty$.
On the other hand,
\begin{align*}
\Big(\sum_{1\leq |n|\leq N} \|\wh{f}(n)\|^{p}_{\ell^{r}_{(2 N)^d}} (|n|+1)^{\beta p}\Big)^{1/p} = \Big(\sum_{1\leq |n|\leq N} (|n|+1)^{\beta p-\frac{dp}{r}}\Big)^{1/p}<\infty
\end{align*}
because $\beta < d/r - d/p$. Therefore, $K_N \lesssim_{\beta,p,r,d}  C\lambda$. Choosing $N$ large enough, we obtain the required contradiction.

\eqref{it:FouriertypePitt2}: By Proposition \ref{prop:WeightedTrans} $\F:L^p(\T^d, |\cdot|^{\beta p};X)\to \ell^q(\Z^d,(|n| + 1)^{-\gamma q};X)$ is bounded. By a duality argument we find that $\F:\ell^{q'}(\Z^d,(|n| + 1)^{\gamma q'};X^*)\to L^{p'}(\T^d, |\cdot|^{-\beta p'};X^*)$ is bounded. Therefore, by \eqref{it:FouriertypePitt1} we find that $X^*$ has type $r\in (1, 2]\cap (1,q')$ if $\gamma<\frac{d}{r} - \frac{d}{q'} = \frac{d}{q}-\frac{d}{r'}$.
By duality we obtain $X$ has cotype $s = r'$.
\end{proof}

As a consequence of Theorems \ref{thm:Pitt2intro}, \ref{ThmPittImpliesType} and \eqref{B}, we see that Banach spaces with a nontrivial Fourier type can be characterized in terms of Pitt's inequalities:
\begin{corollary}
	For every Banach space $X$ the following are equivalent
	\begin{enumerate}[\upshape(i)]
		\item \label{it:FType} $X$ has nontrivial Fourier type,
		\item \label{it:PittType} $X$ obeys the Pitt's inequality \eqref{thm:Pitt2intro2} for some $p, q, \beta, \gamma$ satisfying \eqref{thm:Pitt2intro1}.
	\end{enumerate}
\end{corollary}

\subsection{Banach lattices}

In this section we shall deal with Banach lattices of functions. We start by recalling some basic facts about  the notions of $p$-convexity and $q$-concavity. For further details we refer the reader to \cite{Lindenstrauss-Tzafriri79}.

Let $1 \leq p, q \leq \infty$. We say that $X$ is \emph{$p$-convex} if there exists a constant $C \geq 0$ such that for all finite choices $f_1, \ldots, f_m \in X$,
\begin{equation*}
	\left\| \left(\sum_{n=1}^m |f_n|^p \right)^{1/p}\right\|_X \leq C \left(\sum_{n=1}^m \|f_n\|_X^p \right)^{1/p}, \quad \text{if} \quad p < \infty,
\end{equation*}
or
\begin{equation*}
	\left\|\sup_{1 \leq n \leq m} |f_n| \right\|_X \leq C \sup_{1 \leq n \leq m} \|f_n\|_X, \quad \text{if} \quad p =\infty.
\end{equation*}
We say that $X$ is \emph{$q$-convex} if there exists a constant $C \geq 0$ such that for all finite choices $f_1, \ldots, f_m \in X$,
\begin{equation*}
	 \left(\sum_{n=1}^m \|f_n\|_X^q \right)^{1/q} \leq C \left\| \left(\sum_{n=1}^m |f_n|^q \right)^{1/q}\right\|_X, \quad \text{if} \quad q < \infty,
\end{equation*}
or
\begin{equation*}
	 \sup_{1 \leq n \leq m} \|f_n\|_X  \leq C \left\|\sup_{1 \leq n \leq m} |f_n| \right\|_X, \quad \text{if} \quad q =\infty.
\end{equation*}

If $X$ is of type $p \in (1, 2]$ (respectively, cotype $q \in [2, \infty)$) then $X$ is $r$-convex for every $1 < r < p$ (respectively, $r$-concave for every $q < r < \infty$). See \cite[Corollary 1.f.9]{Lindenstrauss-Tzafriri79}.

Let $(S, \Sigma, \mu)$ be a measure space and $1 \leq p < \infty$. Then, $L^p(S)$ is $p$-convex and $p$-concave.

The following result was obtained in \cite[Proposition 2.2]{GTK96}.

\begin{proposition}\label{GKT}
	Assume that $X$ is $p$-convex and $p'$-concave for some $p \in (1,2]$. Then, $X$ is of Fourier type $p$.
\end{proposition}

In the special case of Banach lattices we obtain an interesting characterization of the notions introduced in Definition \ref{def:type}.

\begin{corollary}\label{cor:BLFouriertype}
Assume $X$ is a Banach lattice and let $p_0\in (1,2)$. Then the following are equivalent:
\begin{enumerate}[\upshape(i)]
\item\label{cor:BLFouriertype1} $X$ has Fourier type $p$ for some $p\in (p_0, 2]$.
\item\label{cor:BLFouriertype2} $X$ has Paley type $p_0$ and Paley cotype $p_0'$.
\item\label{cor:BLFouriertype3} $X$ has HL type $p_0$ and HL cotype $p_0'$.
\end{enumerate}
\end{corollary}
\begin{proof}

\eqref{cor:BLFouriertype1} $\Rightarrow$\eqref{cor:BLFouriertype2} and \eqref{cor:BLFouriertype2}$\Rightarrow$\eqref{cor:BLFouriertype3} follow from Propositions \ref{prop:WeakFourier} and \ref{prop:relationsPFHL}, respectively, and are true for general Banach spaces.

\eqref{cor:BLFouriertype3}$\Rightarrow$\eqref{cor:BLFouriertype1}: By Theorem \ref{thm:HLimpliestypecotype} $X$ has type $p_0+\varepsilon$ and cotype $(p_0+\varepsilon)'$ for some $\varepsilon>0$. Then, $X$ is $p_0+\tfrac12\varepsilon$-convex and $(p_0+\tfrac12\varepsilon)'$-concave. Now, by Proposition \ref{GKT}, $X$ has Fourier type $p_0+\tfrac12\varepsilon$.
\end{proof}

In the same way as in Proposition \ref{GKT} one can obtain the following improvement of Theorem \ref{thm:Pitt2intro} for Banach lattices.
\begin{proposition}\label{PropNew}
Let $X$ be a Banach lattice which is $p$-convex and $q$-concave with $1<p\leq q<\infty$. Assume
\begin{equation*}
\max\left\{0, d \left(\tfrac{1}{p} + \tfrac{1}{q} - 1\right) \right\} \leq \gamma < \tfrac{d}{q} \quad \text{and} \quad \beta - \gamma = d \Big(1- \tfrac{1}{p} - \tfrac{1}{q}\Big).
\end{equation*}
Then
\begin{equation}\label{PropNew2}
\|\widehat{f}\|_{L^q(\R^d, |\cdot|^{-\gamma q};X)} \leq C \|f\|_{L^p(\R^d, |\cdot|^{\beta p};X)}.
\end{equation}
\end{proposition}

The proof of this result can be obtained, mutatis mutandis, the method of proof of Proposition \ref{GKT} given in \cite[Proposition 2.2]{GTK96}, which allows to reduce the $X$-valued inequality \eqref{PropNew2} to the well-known scalar case (cf. \eqref{PittScalarintroseries} and \eqref{PittAssumpScalarintro}) making use of the $p$-convexity and $q$-concavity of $X$.

\begin{remark}
	Proposition \ref{PropNew} does not contradict Remark \ref{rem:SharpnessPitt2} because $X = \ell^{p_0'}(\Z^d), \, 1 < p_0 < 2$, is $p$-convex for $p \leq p_0'$, but not $q$-concave for $p \leq q < p_0'$.
\end{remark}

\section{Extreme cases of Hausdorff-Young type inequalities}\label{sec:Paley}

\subsection{Zygmund type inequalities}
The classical Zygmund inequality \cite[p. 158]{Zygmund} asserts that
\begin{equation}\label{ZYG}
 \|(\wh{f}(n))\|_{\ell^{\infty, 1}(\log \ell)^b(\Z)} \leq C \|f\|_{L^1 (\log L)^{b+1}(\T)}, \quad b + 1>0.
\end{equation}
The goal of this section is to study this inequality in the context of vector-valued functions.

Our next result shows that the vector-valued Zygmund inequality (as well as its generalizations to Lorentz-Zygmund spaces and discrete counterparts) characterizes the nontrivial Fourier type of the underlying Banach space. Before we state it, we recall that the space $L^{1,q}(\log L)^{-\frac1q +1} (\log \log L)^{1}(\T^d;X)$ is formed by all $f:\T^d\to X$ such that
	\begin{align*}
		\|f\|_{L^{1,q}(\log L)^{-\frac1q +1} (\log \log L)^{1}(\T^d;X)} \\
		& \hspace{-4cm}= \Big(\int_0^1 (t (1 + |\log t|)^{-\frac1q +1} (1 + \log (1 + |\log t|)) f^\ast(t))^q \,\frac{dt}{t} \Big)^{1/q} < \infty.
	\end{align*}

\begin{theorem}\label{ThmZygmund}
	The following statements are equivalent:
	\begin{enumerate}[\upshape(i)]
	\item\label{ThmZygmund*} $X$ has nontrivial Fourier type,
	\item\label{ThmZygmund1} Let $1 \leq q \leq \infty$ and $b > -1/q$. The Fourier transform is bounded from $L^{1,q}(\log L)^{b+1}(\T^d;X)$ into $\ell^{\infty, q}(\log \ell)^{b}(\Z^d;X)$ is bounded,
	\item\label{ThmZygmund2}   Let $1 \leq q < \infty$. The Fourier transform is bounded from $L^{1,q}(\log L)^{-\frac1q +1} (\log \log L)^{1}(\T^d;X)$ into $\ell^{\infty, q}(\log \ell)^{-1/q}(\Z^d;X)$,
	\item\label{ThmZygmund3} Let $1 \leq q \leq \infty$ and $b < -1/q$. The Fourier transform is bounded from $\ell^{1,q}(\log \ell)^{b + 1}(\Z^d;X)$ into $L^{\infty, q}(\log L)^b(\T^d;X)$.
	\end{enumerate}
\end{theorem}

\begin{remark}
The counterpart of \eqref{ThmZygmund2} for sequences, that is, the limiting case $b=-1/q$ in \eqref{ThmZygmund3}, is obviously not true because $L^{\infty, q}(\log L)^{-1/q}(\T^d;X) = \{0\}$ if $q < \infty$.
\end{remark}

Let us discuss some applications of Theorem \ref{ThmZygmund}.
The quantitative version of the vector-valued Zygmund inequality \eqref{ZYG} was shown by Parcet, Soria and Xu \cite[Theorem 2]{ParcetSoriaXu}. Namely, if $X$ has nontrivial Fourier type and
\begin{equation*}
 	\int_\T \|f(t)\|_X (\log^+ \|f(t)\|_X)^b \, dt \leq \rho
 \end{equation*}
for some $b > 0$, then there exist positive constants $a(\rho, b)$ and $A(\rho, b)$ such that
\begin{equation}\label{ParcetSoriaXu}
	\sum_{n=-\infty}^\infty \text{exp} \Big(-a(\rho,b) \|\wh{f}(n)\|_X^{-\frac{1}{b}} \Big) \leq A(\rho,b).
\end{equation}
Now we establish the extension of \eqref{ParcetSoriaXu} to the more general scale of Lorentz-Zygmund spaces $L^{1,q}(\log L)^b(\T^d;X)$ as well as its discrete counterparts.

\begin{theorem}\label{ThmZygmundExp}
	Assume that $X$ has nontrivial Fourier type.
	\begin{enumerate}[\upshape(i)]
	\item \label{ThmZygmundExp1} If
\begin{equation*}
 	\|f\|_{L^{1,q}(\log L)^{b+1}(\T^d;X)} \leq \rho
 \end{equation*}
for some $1 \leq q \leq \infty$ and $b  + 1/q > 0$, then
\begin{equation}\label{CorZyg1}
	\sum_{n \in \Z^d} \emph{exp} \Big(-a(\rho,b, q) \|\wh{f}(n)\|_X^{-\frac{1}{b + \frac{1}{q}}} \Big) \leq A(\rho,b, q)
\end{equation}
for some positive constants $a(\rho, b, q)$ and $A(\rho, b, q)$.
\item \label{ThmZygmundExp2} If
\begin{equation*}
	\|(c_n) \|_{\ell^{1,q}(\log \ell)^{b+1}(\Z^d;X)} \leq \rho
\end{equation*}
for some $1 \leq q \leq \infty$ and $b + 1/q < 0$, then $(c_n)$ is the sequence of Fourier coefficients of a function $f$ such that
\begin{equation*}
	\int_{\T^d} \emph{exp} \Big(a(\rho,b, q) \|f(t)\|_X^{-\frac{1}{b + \frac{1}{q}}} \Big) \, \ud t  \leq A(\rho,b, q)
\end{equation*}
for some positive constants $a(\rho, b, q)$ and $A(\rho, b, q)$.
	\end{enumerate}
\end{theorem}

Our previous result recovers \eqref{ParcetSoriaXu}. Indeed, take into account that
\begin{equation*}
	\|f\|_{L^{1}(\log L)^{b+1}(\T^d;X)} \asymp \int_{\T^d} \|f(t)\|_X (\log^+ \|f(t)\|_X)^{b+1} \, \ud t
\end{equation*}
and
\begin{equation*}
	\|(c_n) \|_{\ell^{1}(\log \ell)^{b+1}(\Z^d;X)} \asymp \sum_{n \in \Z^d} \|c_n\|_X \left(\log^{+} \left(\frac{1}{\|c_n\|_X}\right) \right)^{b+1}
\end{equation*}
(see \cite[Theorems D and D']{BennettRudnick}), writing down Theorem \ref{ThmZygmundExp} with $q=1$ we obtain the following

\begin{corollary}\label{CorZyg}
	Assume that $X$ has non trivial Fourier type.
	\begin{enumerate}[\upshape(i)]
	\item If
\begin{equation*}
 	\int_{\T^d} \|f(t)\|_X (\log^+ \|f(t)\|_X)^b \, \ud t \leq \rho
 \end{equation*}
for some $b > 0$, then
\begin{equation*}
	\sum_{n \in \Z^d} \emph{exp} \Big(-a(\rho,b) \|\wh{f}(n)\|_X^{-\frac{1}{b}} \Big) \leq A(\rho,b)
\end{equation*}
for some positive constants $a(\rho, b)$ and $A(\rho, b)$.
\item If
\begin{equation*}
	\sum_{n \in \Z^d} \|c_n\|_X \left(\log^{+} \left(\frac{1}{\|c_n\|_X}\right) \right)^{-b} \leq \rho
\end{equation*}
for some $b > 0$, then $(c_n)$ is the sequence of Fourier coefficients of a function $f$ such that
\begin{equation*}
	\int_{\T^d} \emph{exp} \Big(a(\rho,b) \|f(t)\|_X^{\frac{1}{b}} \Big) \, \ud t  \leq A(\rho,b)
\end{equation*}
for some positive constants $a(\rho, b)$ and $A(\rho, b)$.
\end{enumerate}
\end{corollary}

The proofs of Theorems \ref{ThmZygmund} and \ref{ThmZygmundExp} will be given in Sections \ref{ProofThmZygmund} and \ref{ProofThmZygmundExp}, respectively.

\subsection{A Bochkarev type inequality}

In the scalar setting it is well known that the Fourier transform does not map Lorentz spaces $L^{2,q}(\T^d)$ to $\ell^{2,q}(\Z^d)$ with $q \neq 2$. This obstacle has been circumvented by Bochkarev \cite{Bochkarev97, Bochkarev} involving logarithmic decays of the sequence of Fourier coefficients. More precisely, if $q > 2$ then
\begin{equation}\label{Bochkarev}
	\wh{f}^\ast (n) \leq C n^{-1/2} (1 + \log n)^{1/2 - 1/q} \|f\|_{L^{2,q}(\T)}, \quad n \in \N.
\end{equation}

The aim of this section is to obtain the vector-valued counterparts of Bochkarev's inequality \eqref{Bochkarev}. Let us state our problem in more detail. Let $X$ be a Banach space of Fourier type $p_0 \in (1, 2]$. Let $1 < p < p_0$ and $1 \leq q \leq \infty$. By \eqref{Konig},
\begin{equation}\label{Boch1}
\F : L^{p, q}(\T^d;X) \to \ell^{p', q}(\Z^d;X).
\end{equation}
In analogy to the scalar case ($p_0 = 2$), one can show that \eqref{Boch1} does not hold true in general if $p=p_0$ and $q > p_0$. Indeed, take $X = \ell^{p_0'}(\Z^d)$ and follow a similar argument to that given in Example \ref{ex:Paley}. However, we will show in Corollary \ref{cor:Bochkarev} below that the vector-valued counterpart of the Bochkarev's inequality \eqref{Bochkarev} holds true. In fact, we will obtain a sharper estimate in terms of Lorentz-Zygmund norms which reads as follows.

\begin{theorem}\label{thm:PaleyExtreme}
	Let $X$ be a Banach space of Fourier type $p_0 \in (1,2]$. Let $1 \leq q \leq \infty$ and $b < -1/q$. Then,
	\begin{equation}\label{thm:PaleyExtreme1new}
		\F : L^{p_0, q}(\log L)^{b + 1/\min\{p_0, q\}} (\T^d;X) \to \ell^{p_0',q}(\log \ell)^{b + 1/\max\{p_0', q\}}(\Z^d;X).
	\end{equation}
\end{theorem}

We postpone the proof of this theorem to Section \ref{Section7.2}, and meanwhile point out some comments and consequences.

\begin{remark}
Theorem \ref{thm:PaleyExtreme} is no longer true under weaker geometric assumptions on $X$, that is to say, given any $r < p_0$ there exist a Banach space $X$ having Fourier type $r$ and a function $f \in L^{p_0,q}(\log L)^{b + 1/\min\{p_0,q\}} (\T^d;X)$ such that $\F f \not \in \ell^{p_0',q}(\log \ell)^{b + 1/\max\{p_0',q\}}(\Z^d;X).$
	Let $r < p_0$ and $X =\ell^{r'}(\Z^d)$. Define
	\begin{equation*}
		f(t) = (|n|^{-\varepsilon} e^{2 \pi i n \cdot t} \one_{n\neq 0}), \quad t \in \T^d,
	\end{equation*}
	where $d/r' < \varepsilon < d/p_0'$. We have
	\begin{equation*}
		\|f\|_{L^{p_0, q}(\log L)^{b + 1/\min\{p_0,q\}} (\T^d;X)} \asymp \left(\sum_{n\neq 0} |n|^{-\varepsilon r'}\right)^{1/r'} < \infty.
	\end{equation*}
	On the other hand, since $\|\widehat{f}(n)\|_X = |n|^{-\varepsilon}, \, n \neq 0,$ we have $(\|\wh{f}(n)\|_X^\ast)_{n \in \N} = (n^{-\varepsilon/d} |B_1|^{\varepsilon/d})_{n \in \N}$, and thus 	\begin{align*}
		\|(\widehat{f}(n))\|_{ \ell^{p_0',q}(\log \ell)^{b + 1/\max\{p_0',q\}}(\Z^d;X)} & \\
		& \hspace{-4cm}\asymp \left(\sum_{n=1}^\infty n^{(-\varepsilon/d + 1/p_0') q} (1 + \log n)^{(b + 1/\max\{p_0',q\}) q} \frac{1}{n}\right)^{1/q} = \infty.
	\end{align*}
	
\end{remark}

\begin{remark}
Theorem \ref{thm:PaleyExtreme} is sharp in the following sense: In general, \eqref{thm:PaleyExtreme1new} does not hold if $b \geq -1/q$. Let $X = \ell^{p_0'}(\Z^d), \, 1 < p_0 \leq 2,$ and $q < p_0'$.  Define
		\begin{equation*}
		f(t) = (a_n e^{2 \pi i n \cdot t}), \quad t \in \T^d,
	\end{equation*}
	for some $(a_n)_{n \in \Z^d}$ to be chosen. If $b=-1/q$ we put
		\begin{equation*}
	 a_n = |n|^{-d/p_0'} (1 + \log |n|)^{-1/p_0'} (1 + \log (1 + \log |n|))^{-\delta} \one_{n \neq 0}
	 \end{equation*}
	 where $1/p_0'< \delta < 1/q$. Then,
	 \begin{align*}
	 	\|f\|_{L^{p_0,q}(\log L)^{-1/q + 1/\min\{p_0,q\}}(\T^d;X)} & \\
		& \hspace{-4cm}\asymp \left(\sum_{n \neq 0} (1 + \log (1 + \log |n|))^{-\delta p_0'} \frac{1}{|n|^d (1 + \log |n|)}\right)^{1/p_0'} < \infty
	 \end{align*}
	 since $\delta > 1/p_0'$.
	 Further, elementary computations yield that
	 \begin{equation*}
	 	\|(\widehat{f}(n))\|_{\ell^{p_0',q} (\log \ell)^{-1/q + 1/p_0'}(\Z^d; X)}^{q} \asymp \sum_{n=1}^\infty  (1 + \log (1 + \log n))^{-\delta q} \frac{1}{n (1 + \log n)} = \infty
	 \end{equation*}
	 because $\delta < 1/q$. Whence, this example illustrates that \eqref{thm:PaleyExtreme1new} does not hold true necessarily if $b = -1/q$.
	
	 The case $b > -1/q$ can be treated analogously by taking the sequence
	 	\begin{equation*}
	 a_n =  |n|^{-d/p_0'} (1 + \log n)^{-\varepsilon} \one_{n \neq 0}
	 \end{equation*}
	 where $1/p_0' < \varepsilon < b + 1/q + 1/p_0'$. Further details are left to the reader.

\end{remark}

\begin{remark}\label{RemarkBochSeq}
	The discrete counterpart of Theorem \ref{thm:PaleyExtreme} reads as follows. 	Let $X$ be a Banach space of Fourier type $p_0 \in (1,2]$. Let $1 \leq q \leq \infty$ and $b < -1/q$. Then,
	\begin{equation*}
		\F : \ell^{p_0, q}(\log \ell)^{b + 1/\min\{p_0, q\}} (\Z^d;X) \to L^{p_0',q}(\log L)^{b + 1/\max\{p_0', q\}}(\T^d;X).
	\end{equation*}
\end{remark}

Writing down Theorem \ref{thm:PaleyExtreme} with $q > p_0$ and $b = -1/p_0$, we obtain the following

\begin{corollary}[\bf{Vector-valued Bochkarev's inequality}]\label{cor:Bochkarev}
	Let $X$ be a Banach space of Fourier type $p_0 \in (1, 2]$. Let $q > p_0$. Then,
	\begin{equation*}
		\F : L^{p_0,q} (\T^d;X) \to \ell^{p_0',q}(\log \ell)^{-1/p_0 + 1/\max\{p_0', q\}}(\Z^d;X).
	\end{equation*}
	In particular,
	\begin{equation}\label{VectorBochkarev}
	\wh{f}^\ast (n) \leq C n^{-1/p_0'} (1 + \log n)^{1/p_0 - 1/\max\{p_0',q\}} \|f\|_{L^{p_0,q}(\T^d;X)}, \quad n \in \N.
\end{equation}
\end{corollary}

\begin{remark}\
\begin{enumerate}[\upshape(i)]
\item In the special case $X=\C$ (or more generally, $X$ is a Hilbert space), the inequality \eqref{VectorBochkarev} coincides with \eqref{Bochkarev}.
\item We will see in Sections \ref{Section7.2} that \eqref{VectorBochkarev} can be strengthened under additional geometric assumptions on $X$.
\end{enumerate}
\end{remark}

\subsection{Limiting interpolation methods and Hardy type inequalities}\label{SectionLimInt}
In this section, for the convenience of the reader, we recall the logarithmic interpolation methods and collect some interpolation formulas and Hardy type inequalities which will play a key role in the proofs of Theorems \ref{ThmZygmund} and \ref{thm:PaleyExtreme}.

Let $(X_0,X_1)$ be an ordered couple of Banach spaces, that is, $X_1 \hookrightarrow X_0$. The real interpolation method can be defined through the \emph{Peetre's $K$-functional}. For $f \in X_0$ and $t > 0$,
\begin{equation*}
	K(t, f) = K(t, f; X_0, X_1) = \inf_{g \in X_1} (\|f - g\|_{X_0} + t \|g\|_{X_1}).
\end{equation*}
It is plain to see that
\begin{equation}\label{KFunctSym}
K(t, f; X_0, X_1) = t K(t^{-1}, f ; X_1, X_0).
\end{equation}

Let $0 < \theta < 1, \, -\infty < b < \infty$, and $1 \leq q \leq \infty$. The \emph{logarithmic interpolation space} $(X_0, X_1)_{\theta, q; b}$ is the set of all $f \in X_0$ for which
\begin{equation}\label{LogIntSpace}
	\|f\|_{(X_0, X_1)_{\theta, q; b}} = \left(\int_0^\infty (t^{-\theta} (1 + |\log t|)^b K(t,f))^q \frac{dt}{t} \right)^{1/q} < \infty
\end{equation}
(with appropriate modifications if $q=\infty$). This method satisfies the interpolation property for bounded linear operators. See \cite{EvansOpic, EvansOpicPick} and the references given there. In particular, if $b=0$ then we recover the classical interpolation space $(X_0, X_1)_{\theta, q}$ (cf. Section \ref{SectionIB}).

It turns out that the Lorentz-Zygmund spaces $L^{p,q}(\log L)^b(S;X)$ (see Section \ref{Sec: FunctionSpaces}) can be generated from Lebesgue spaces applying logarithmically perturbed interpolation methods. More precisely, the following extension of Lemma \ref{lem:interpolationLebesgue} holds.

\begin{lemma}[{\cite[Corollary 8.4]{EvansOpic}}]\label{lem:LogInterpolationLebesgue}
Let $1 \leq p_0, p_1 \leq \infty, p_0 \neq p_1, 1 \leq q \leq \infty, 0 < \theta < 1, -\infty < b < \infty$, and $1/p = (1-\theta)/p_0 + \theta/p_1$. Then,
	\begin{equation*}
		(L^{p_0}(S;X), L^{p_1}(S;X))_{\theta, q;b} = L^{p,q}(\log L)^b(S;X).
	\end{equation*}
\end{lemma}

The finer tuning given by logarithmic weights in \eqref{LogIntSpace} allows us to consider the extreme cases $\theta = 0, 1$ in $(X_0, X_1)_{\theta, q; b}$. We first note that
\begin{equation*}
	\|f\|_{(X_0, X_1)_{\theta, q; b}} \eqsim \left(\int_0^1 (t^{-\theta} (1 + |\log t|)^b K(t,f))^q \frac{dt}{t} \right)^{1/q}, \quad 0 < \theta < 1.
\end{equation*}
This follows easily from the fact that $K(t, f) \eqsim \|f\|_{X_0}, \, t > 1$, together with monotonicity properties of $K(t,f)$.
Now we can define the limiting interpolation space $(X_0, X_1)_{\theta, q; b}, \, \theta = 0, 1$, as the collection of all $f \in X_0$ such that
\begin{equation}\label{LogIntSpace*}
	\|f\|_{(X_0, X_1)_{\theta, q; b}} = \left(\int_0^1 (t^{-\theta} (1 + |\log t|)^b K(t,f))^q \frac{dt}{t} \right)^{1/q} < \infty.
\end{equation}
We shall assume that $b \geq -1/q \, (b > 0 \text{ if } q = \infty)$ if $\theta = 0$ and $b < -1/q \, (b \leq 0 \text{ if } q= \infty)$ if $\theta = 1$. Otherwise, $(X_0, X_1)_{0,q; b} = X_0, \, b < -1/q \, (b \leq 0 \text{ if } q=\infty)$ and $(X_0,X_1)_{1,q;b} = \{0\}, \, b \geq -1/q \, (b > 0 \text{ if } q=\infty)$. For further details on limiting interpolation spaces we refer the reader to \cite{EvansOpic, CFKU, EvansOpicPick} and the references therein.

Next we recall some reiteration formulas involving the interpolation methods considered above.

\begin{lemma}[\bf{Reiteration formulas for limiting interpolation methods}, \cite{EvansOpic, EvansOpicPick}]\label{lem:reiteration}
	Let $(X_0,X_1)$ be an ordered Banach couple, that is, $X_1 \hookrightarrow X_0$. Let $0 < \theta < 1, 1 \leq p, q \leq \infty$.
	\begin{enumerate}[\upshape(i)]
	\item \label{lem:reiteration1} If $b < -1/q$ then
	\begin{equation*}
		(X_0,X_1)_{\theta,q;b + 1/\min\{p,q\}} \hookrightarrow (X_0, (X_0, X_1)_{\theta,p})_{1,q;b} \hookrightarrow (X_0,X_1)_{\theta,q;b + 1/\max\{p,q\}}.
	\end{equation*}
	\item \label{lem:reiteration2} If $b \geq -1/q \, (b > 0  \text{ if }  q=\infty)$ then
	\begin{equation*}
		(X_0, (X_0, X_1)_{\theta, p})_{0,q;b} = (X_0, X_1)_{0,q;b}.
	\end{equation*}
	\item\label{lem:reiteration3} If $b < -1/q \, (b \leq 0 \text{ if } q=\infty)$ then
	\begin{equation*}
		((X_0, X_1)_{\theta,p}, X_1)_{1,q;b} = (X_0, X_1)_{1,q;b}.
	\end{equation*}
	\end{enumerate}
\end{lemma}

Our next result collects some Hardy type inequalities involving logarithmic weights that are useful in later considerations. For more information, we refer to the monograph \cite{OpicKufner}.

\begin{lemma}[\textbf{Hardy type inequalities for functions}]\label{LemmaHardy}
		Let $1 \leq q \leq \infty$. Let $\psi$ be a non-negative measurable function on $(0,1)$.
		\begin{enumerate}[\upshape(i)]
		\item \label{HardyInequal1} If $b + 1/q > 0$ then
	\begin{equation*}
		\left(\int_0^1 \left((1 - \log t)^b \int_0^t \psi(s) ds \right)^q \frac{dt}{t}\right)^{1/q} \lesssim \left(\int_0^1 (t (1-\log t)^{b + 1} \psi(t))^q \frac{dt}{t}\right)^{1/q}.
	\end{equation*}
	\item \label{HardyInequal3} If $b + 1/q>0$ then
	\begin{align*}
		\left(\int_0^1 \left((1+\log(1 - \log t))^b \int_0^t \psi(s) ds \right)^q \frac{dt}{t(1-\log t)}\right)^{1/q} \\
		& \hspace{-7cm}\lesssim \left(\int_0^1 (t (1-\log t) (1+\log(1-\log t))^{b + 1} \psi(t))^q \frac{dt}{t (1-\log t)}\right)^{1/q}.
	\end{align*}
	\end{enumerate}
\end{lemma}

\begin{lemma}[\textbf{Hardy type inequality for sequences}]\label{LemmaHardySeq}
	Let $1 \leq q \leq \infty, \, b + 1/q < 0$, and $c_n \geq 0, \, n \in \N$. Then
	\begin{equation*}
		\left(\sum_{n=1}^\infty \left((1 + \log n)^b \sum_{k=1}^n c_k \right)^q \frac{1}{n} \right)^{1/q} \lesssim \left(\sum_{n=1}^\infty (n (1 + \log n)^{b+1} c_n)^q \frac{1}{n} \right)^{1/q}.
	\end{equation*}
\end{lemma}

\subsection{Proof of Theorem \ref{ThmZygmund}}\label{ProofThmZygmund}
\eqref{ThmZygmund*} $\Rightarrow$ \eqref{ThmZygmund1}, \eqref{ThmZygmund2}: By assumption, there exists $p \in (1, 2]$ such that $\F : L^p(\T^d;X) \to \ell^{p'}(\Z^d;X)$. On the other hand, $\F : L^1(\T^d;X) \to \ell^\infty(\Z^d;X)$. Let $1 \leq q \leq \infty$ and $b \geq -1/q$ if $q < \infty \, (b > 0 \text{ if } q=\infty)$. Applying the limiting interpolation method \eqref{LogIntSpace*} with $\theta = 0$, we derive
\begin{equation}\label{ProofThmZygmund1}
	\F : ( L^1(\T^d;X) , L^p(\T^d;X))_{0,q;b}  \to (\ell^\infty(\Z^d;X), \ell^{p'}(\Z^d;X))_{0,q;b}.
\end{equation}
In view of Lemmas \ref{lem:LogInterpolationLebesgue} and \ref{lem:reiteration}\eqref{lem:reiteration2}, we have
\begin{align}
	 ( L^1(\T^d;X) , L^p(\T^d;X))_{0,q;b} & = (L^1(\T^d;X), (L^1(\T^d;X), L^\infty(\T^d;X))_{1/p',p})_{0,q;b} \nonumber\\
	 & = (L^1(\T^d;X), L^\infty(\T^d;X))_{0,q;b}. \label{ProofThmZygmund2}
\end{align}
This, together with the well-known formula (cf. \cite[Chapter 5, Theorem 1.6, page 298]{Bennett-Sharpley88} and \cite[Section 1.18.6, (9), page 133]{Triebel95})
\begin{equation}\label{ProofThmZygmund3}
	K(t, f ; L^1(\T^d;X), L^\infty(\T^d;X)) = \int_0^t f^\ast(s) \, ds,
\end{equation}
yields that
\begin{align*}
	\|f\|_{( L^1(\T^d;X) , L^p(\T^d;X))_{0,q;b} } & \asymp \|f\|_{(L^1(\T^d;X), L^\infty(\T^d;X))_{0,q;b}} \\
	& \hspace{-2cm}= \left(\int_0^1  \left((1-\log t)^{b}\int_0^t f^\ast(s) \, ds \right)^q \frac{dt}{t} \right)^{1/q}.
\end{align*}
Now we distinguish two cases. If $b > -1/q$ then we can apply \ref{LemmaHardy}\eqref{HardyInequal1} to get
\begin{equation}\label{ProofThmZygmund4}
\|f\|_{( L^1(\T^d;X) , L^p(\T^d;X))_{0,q;b} } \lesssim \|f\|_{L^{1,q}(\log L)^{b+1}(\T^d;X)}.
\end{equation}
If $b=-1/q$ then we invoke Lemma \ref{LemmaHardy}\eqref{HardyInequal3} to obtain
\begin{equation}\label{ProofThmZygmund5}
\|f\|_{( L^1(\T^d;X) , L^p(\T^d;X))_{0,q;b} } \lesssim \|f\|_{L^{1,q}(\log L)^{-1/q+1} (\log \log L)(\T^d;X)}.
\end{equation}

On the other hand, similarly as \eqref{ProofThmZygmund2}, we can prove
\begin{equation*}
	(\ell^\infty(\Z^d;X), \ell^{p'}(\Z^d;X))_{0,q;b} = (\ell^\infty(\Z^d;X), \ell^{1}(\Z^d;X))_{0,q;b}.
\end{equation*}
This formula, \eqref{ProofThmZygmund3}, which also holds when we replace $\T^d$ by $\Z^d$, \eqref{KFunctSym} and monotonicity properties of $K$-functionals imply that
\begin{align}
	\|(c_n)\|_{(\ell^\infty(\Z^d;X), \ell^{p'}(\Z^d;X))_{0,q;b}} & \asymp \|(c_n)\|_{(\ell^\infty(\Z^d;X), \ell^{1}(\Z^d;X))_{0,q;b}} \nonumber \\
	& \hspace{-2cm}\asymp \left(\sum_{i=1}^\infty (i^{-1} (1+ \log i)^{b}  K(i, (c_n);  \ell^{1}(\Z^d;X),  \ell^{\infty}(\Z^d;X))^q \frac{1}{i} \right)^{1/q} \nonumber \\
	& \hspace{-2cm} \asymp \left(\sum_{i=1}^\infty \left(i^{-1} (1+ \log i)^{b}  \sum_{j=1}^{i} c_j^\ast\right)^q \frac{1}{i} \right)^{1/q} \nonumber \\
	& \hspace{-2cm} \gtrsim \left( \sum_{i=1}^\infty ((1 + \log i)^b c_i^\ast)^q \frac{1}{i}\right)^{1/q} = \|(c_n)\|_{\ell^{\infty,q} (\log \ell)^b (\Z^d;X)}. \label{ProofThmZygmund6}
\end{align}

Finally, the desired assertions follow from \eqref{ProofThmZygmund1}, \eqref{ProofThmZygmund4}, \eqref{ProofThmZygmund5} and \eqref{ProofThmZygmund6}.

\eqref{ThmZygmund*} $\Rightarrow$ \eqref{ThmZygmund3}: By assumption, $\F : \ell^p(\Z^d;X) \to L^{p'}(\T^d;X)$ for some $p \in (1,2]$. Applying the limiting interpolation method \eqref{LogIntSpace*} with $\theta = 1$, we have
\begin{equation}\label{ProofThmZygmund7}
	\F : ( \ell^p(\Z^d;X),  \ell^1(\Z^d;X))_{1, q; b} \to (L^{p'}(\T^d;X), L^\infty(\T^d;X))_{1,q;b}.
\end{equation}

In virtue of Lemmas \ref{lem:LogInterpolationLebesgue} and \ref{lem:reiteration}\eqref{lem:reiteration3},
\begin{align*}
	 ( \ell^p(\Z^d;X),  \ell^1(\Z^d;X))_{1, q; b}& = ((\ell^\infty(\Z^d;X), \ell^1(\Z^d;X))_{1/p,p}, \ell^1(\Z^d;X))_{1, q; b} \\
	 &=  ( \ell^\infty(\Z^d;X),  \ell^1(\Z^d;X))_{1, q; b},
\end{align*}
and thus, by \eqref{KFunctSym}, \eqref{ProofThmZygmund3} and Lemma \ref{LemmaHardySeq},
\begin{align}
	\|(c_n)\|_{ ( \ell^p(\Z^d;X),  \ell^1(\Z^d;X))_{1, q; b}} & \asymp \|(c_n)\|_{ ( \ell^\infty(\Z^d;X),  \ell^1(\Z^d;X))_{1, q; b}} \nonumber \\
	& \hspace{-3cm}\asymp \left(\sum_{i=1}^\infty \left( (1 + \log i)^b \sum_{j=1}^{i} c_j^\ast  \right)^q \frac{1}{i} \right)^{1/q} \nonumber\\
	& \hspace{-3cm} \lesssim \left(\sum_{i=1}^\infty (i (1 + \log i)^{b + 1} c_i^\ast)^q \frac{1}{i} \right)^{1/q} = \|(c_n)\|_{\ell^{1,q}(\log \ell)^{b+1}(\Z^d;X)}. \label{ProofThmZygmund8}
\end{align}

On the other hand, applying again Lemmas \ref{lem:LogInterpolationLebesgue} and \ref{lem:reiteration}\eqref{lem:reiteration3},
\begin{align*}
	(L^{p'}(\T^d;X), L^\infty(\T^d;X))_{1,q;b} &= ((L^1(\T^d;X), L^\infty(\T^d;X))_{1/p, p'}, L^\infty(\T^d;X))_{1,q;b} \\
	& = (L^1(\T^d;X), L^\infty(\T^d;X))_{1,q;b}
\end{align*}
and so, by \eqref{ProofThmZygmund3},
\begin{align}
	\|f\|_{(L^{p'}(\T^d;X), L^\infty(\T^d;X))_{1,q;b} } & \asymp \|f\|_{(L^1(\T^d;X), L^\infty(\T^d;X))_{1,q;b} } \nonumber \\
	& \hspace{-3cm}= \left(\int_0^1 \left( t^{-1} (1- \log t)^b \int_0^t f^\ast(s) \, ds\right)^q \frac{dt}{t} \right)^{1/q} \nonumber \\
	& \hspace{-3cm}\gtrsim \left(\int_0^1 ((1-\log t)^b f^\ast(t))^q \frac{dt}{t} \right)^{1/q} = \|f\|_{L^{\infty,q}(\log L)^b(\T^d;X)}. \label{ProofThmZygmund9}
\end{align}

According to \eqref{ProofThmZygmund7}-\eqref{ProofThmZygmund9},
\begin{equation*}
	\F : \ell^{1,q}(\log \ell)^{b+1}(\Z^d;X) \to L^{\infty,q}(\log L)^b(\T^d;X).
\end{equation*}

\eqref{ThmZygmund3} $\Rightarrow$ \eqref{ThmZygmund*}: Assume $\F:\ell^{1,q}(\log \ell)^{b + 1}(\Z^d;X)\to L^{\infty, q}(\log L)^b(\T^d;X)$ is bounded for some $b<-1/q$. We show that $X$ has nontrivial Fourier type. By  \eqref{B} it suffices to show that $X$ has nontrivial type. As in the proof of Theorem \ref{thm:HLimpliestypecotype} it is enough to construct a suitable example in the space $\ell^1_{(2N)^d}$ with $N\geq 1$.
Define $f:\T^d\to \ell^{1}_{(2N)^d}$ by
\[f(t) = ((|n|+1)^{-d} e^{2\pi i t \cdot n})_{1\leq |n|\leq N},\]
where $|n| = \max\{|n_j|: j\in \{1, \ldots, d\}\}$. Then
\[\|f(t)\|_{\ell^1_{(2N)^d}} \eqsim_d \log(N+1),\]
and therefore,
\begin{equation}\label{eq:growthLHS}
\|f\|_{L^{\infty, q}(\log L)^b(\T^d;X)} \eqsim_{b,d,q} \log(N+1)
\end{equation}
because $b < -1/q$. On the other hand, $\|\wh{f}(n)\|_{\ell^1_{(2N)^d}} = (|n|+1)^{-d}$ if $1\leq |n|\leq N$ and zero otherwise. Therefore
\begin{align*}
\|\wh{f}\|_{\ell^{1,q}(\log \ell)^{b + 1}(\Z^d;X)}  & \eqsim_{b,d,q} \Big(\sum_{n=1}^{ (2 N)^d} (1 + \log n)^{(b+1)q} \frac{1}{n}\Big)^{1/q} \\
 &\lesssim_{b,d, q} \left\{\begin{array}{lcl}
                                (\log(N+1))^{b+1+\frac1q} & \text{ if } & b + 1/q > -1, \\
                                             &              &              \\
                                (\log (\log (N + 1)))^{1/q} & \text{ if } & b + 1/q = -1, \\
                                & & \\
                                C & \text{ if } & b + 1/q < -1.
                                \end{array}
                        \right.
\end{align*}
Since $b<-1/q$ the asymptotic behaviour of this sequence is smaller than \eqref{eq:growthLHS}. This yields the required result.

\eqref{ThmZygmund1} $\Rightarrow$ \eqref{ThmZygmund*}: Assume that
\begin{equation}\label{DualFourierlL}
	\F : L^{1,q}(\log L)^{b+1}(\T^d;X) \to \ell^{\infty, q}(\log \ell)^{b}(\Z^d;X)
\end{equation}
for some $b > -1/q$. Then, a duality argument allows us to obtain
\begin{equation}\label{DualFourierlL*}
	\F : \ell^{1, q'}(\log \ell)^{-b}(\Z^d;X^\ast)\to L^{\infty,q'}(\log L)^{-b-1}(\T^d;X^\ast).
\end{equation}
Indeed, let $y = (y_m)$ be a sequence of elements of $X^\ast$ and let $g (t)= \sum_{m \in \Z^d} \wh{g}(m) e^{2 \pi i m \cdot t} \in L^{1,q}(\log L)^{b+1}(\T^d;X)$. It follows from the discrete version of the Hardy-Littlewood inequality \eqref{HL} (see \cite[p. 43]{Bennett-Sharpley88}), H\"older's inequality and \eqref{DualFourierlL} that
\begin{align*}
	\int_{\T^d} \Big|\Big\lb g(t), \sum_{m \in \Z^d} y_m e^{2 \pi i m\cdot t} \Big\rb \Big| \, \ud t  & = \sum_{m \in \Z^d} |\lb \wh{g}(m), y_m \rb|  \leq \sum_{m \in \Z^d} \|\wh{g}(m)\|_X \|y_m\|_{X^\ast}  \\
	& \hspace{-3cm} \leq \sum_{n=1}^\infty \|\wh{g}(n)\|_X^\ast \|y_n\|_{X^\ast}^\ast \leq \|\F g\|_{\ell^{\infty, q}(\log \ell)^b(\Z^d; X)} \|y\|_{\ell^{1,q'}(\log \ell)^{-b}(\Z^d;X^\ast)} \\
	& \hspace{-3cm} \lesssim \|g\|_{ L^{1,q}(\log L)^{b+1}(\T^d;X)} \|y\|_{\ell^{1,q'}(\log \ell)^{-b}(\Z^d;X^\ast)}.
\end{align*}
Taking now the supremum over all $g \in L^{1,q}(\log L)^{b+1}(\T^d;X)$ with $\|g\|_{L^{1,q}(\log L)^{b+1}(\T^d;X)} \leq 1$, we arrive at
\begin{equation*}
	\sup_{\|g\|_{L^{1,q}(\log L)^{b+1}(\T^d;X)} \leq 1} \int_{\T^d} \Big|\Big\lb g(t), \sum_{m \in \Z^d} y_m e^{2 \pi i m\cdot t} \Big\rb \Big| \, \ud t \lesssim  \|y\|_{\ell^{1,q'}(\log \ell)^{-b}(\Z^d;X^\ast)}.
\end{equation*}
Note that the left-hand side term is equivalent to $\left\|  \sum_{m \in \Z^d} y_m e^{2 \pi i m\cdot t}  \right\|_{L^{\infty,q'}(\log L)^{-b-1}(\T^d;X^\ast)}$ because $L^{1,q}(\log L)^{b+1}(\T^d;X)$ is norming for $L^{\infty,q'}(\log L)^{-b-1}(\T^d;X^\ast)$. This latter fact is well known in the scalar-valued case (see \cite[Theorem 6.2]{OpicPick}) and can be easily generalized to the vector-valued setting by following a similar argument as that given in the proof of Lemma \ref{lem:dualitynew} which takes into account the density of simple functions in Lorentz-Zygmund spaces (see \cite[Section 2.3.2]{CarroRaposoSoria}). Hence, we have shown \eqref{DualFourierlL*}.

Finally, since \eqref{ThmZygmund3} $\Rightarrow$ \eqref{ThmZygmund*} holds, we derive that $X^\ast$ has nontrivial Fourier type, or equivalently, $X$ has nontrivial Fourier type (see Proposition \ref{prop:duality}\eqref{DualF}).

\eqref{ThmZygmund2} $\Rightarrow$ \eqref{ThmZygmund*}: Suppose that
\begin{equation*}
	\F : L^{1,q}(\log L)^{-1/q+1} (\log \log L)^1(\T^d;X) \to \ell^{\infty, q}(\log \ell)^{-1/q}(\Z^d;X).
\end{equation*}
Then, applying duality in a similar fashion as in the proof of the implication \eqref{ThmZygmund1} $\Rightarrow$ \eqref{ThmZygmund*} given above, one can show that
\begin{equation}\label{FlLInfty}
	\F : \ell^{1, q'}(\log \ell)^{1/q}(\Z^d;X^\ast)\to L^{\infty,q'}(\log L)^{1/q-1} (\log \log L)^{-1}(\T^d;X^\ast).
\end{equation}
Next we prove that \eqref{FlLInfty} implies that $X^\ast$ has nontrivial type and hence, $X$ has nontrivial Fourier type. Indeed, define $f:\T^d\to \ell^{1}_{(2N)^d}$ by
\[f(t) = ((1 + |n|)^{-d}  (1 + \log |n|)^{-1} e^{2\pi i t \cdot n})_{1\leq |n|\leq N},\]
where $|n| = \max\{|n_j|: j\in \{1, \ldots, d\}\}$. Elementary computations yield that
\[\|f(t)\|_{\ell^1_{(2N)^d}} \eqsim_d \log (1 + \log N),\]
and so
\begin{equation}\label{eq:growthLHSnew}
\|f\|_{L^{\infty,q'}(\log L)^{1/q-1} (\log \log L)^{-1}(\T^d;\ell^{1}_{(2N)^d})} \eqsim_{d,q} \log (1 + \log N)
\end{equation}
because $\int_0^1 (1 + \log (1-\log t))^{-q'} \frac{dt}{t (1 - \log t)} < \infty$. On the other hand, since $\|\wh{f}(n)\|_{\ell^1_{(2N)^d}} = (1+ |n|)^{-d} (1 + \log |n|)^{-1} \one_{1 \leq |n| \leq N}$, we have
\begin{equation*}
\|\wh{f}\|_{\ell^{1,q'}(\log \ell)^{1/q}(\Z^d;\ell^{1}_{(2N)^d})}   \eqsim_{d,q} \Big(\sum_{n=1}^{ (2 N)^d}  \frac{1}{n (1 + \log n)}\Big)^{1/q'} \eqsim_{d,q} (\log (1 +\log N))^{1/q'}.
\end{equation*}
Then, according to \eqref{FlLInfty} and \eqref{eq:growthLHSnew}, we arrive at
\begin{equation*}
	\log (1 + \log N) \lesssim (\log (1 +\log N))^{1/q'}
\end{equation*}
which is not true. This allows us to conclude that $X^\ast$ does not contain $\ell^1_{(2N)^d}$'s uniformly and, by Lemma \ref{lem:Pisiertype}, $X^\ast$ has nontrivial type.

	\qed

\subsection{Proof of Theorem \ref{ThmZygmundExp}}\label{ProofThmZygmundExp}

 \eqref{ThmZygmundExp1}: By Theorem \ref{ThmZygmund}\eqref{ThmZygmund1}, $\F :  L^{1,q}(\log L)^{b+1}(\T^d;X) \to \ell^{\infty, q}(\log \ell)^{b}(\Z^d;X)$. Hence, since $\ell^{\infty, q}(\log \ell)^{b}(\Z^d;X) \hookrightarrow \ell^\infty (\log \ell)^{b+1/q}(\Z^d;X)$ (see \cite[Theorem 9.5]{BennettRudnick}), we derive
 \begin{equation*}
 	\F :   L^{1,q}(\log L)^{b+1}(\T^d;X) \to \ell^\infty (\log \ell)^{b+1/q}(\Z^d;X)
 \end{equation*}
 which yields \eqref{CorZyg1} because $\ell^\infty (\log \ell)^{b+1/q}(\Z^d;X)$ coincides with space of all vector-valued sequences $(c_n)$ such that
 \begin{equation*}
  \sum_{n \in \Z^d} \text{exp} \Big(- \lambda \|c_n\|_X^{-\frac{1}{b + \frac{1}{q}}} \Big) < \infty
 \end{equation*}
 for some $\lambda > 0$ which depends on $b, q$ and $\|(c_n)\|_{\ell^\infty (\log \ell)^{b+1/q}(\Z^d;X)}$ (see \cite[Theorem D']{BennettRudnick}).

 The proof of \eqref{ThmZygmundExp2} can be done in a similar way. We omit further details. 	\qed

\subsection{Proof of Theorem \ref{thm:PaleyExtreme}}\label{Section7.2}
	Since $X$ has Fourier type $p_0$, it holds that (see Lemma \ref{lem:ClassicalTrans})
	\begin{equation}\label{thm:PaleyExtreme1}
		\F : L^{p_0}(\T^d; X) \to \ell^{p_0'}(\Z^d;X).
	\end{equation}
	On the other hand, we have the trivial estimate
	\begin{equation}\label{thm:PaleyExtreme2}
		\F : L^{1}(\T^d; X) \to \ell^{\infty}(\Z^d;X).
	\end{equation}
	If we interpolate \eqref{thm:PaleyExtreme1} and \eqref{thm:PaleyExtreme2} by applying the interpolation method \eqref{LogIntSpace*} with $\theta = 1$, we get
	\begin{equation}\label{thm:PaleyExtreme3}
		\F : (L^{1}(\T^d; X), L^{p_0}(\T^d; X))_{1, q;b} \to ( \ell^{\infty}(\Z^d;X), \ell^{p_0'}(\Z^d;X))_{1,q;b}.
	\end{equation}
	Concerning the source space in \eqref{thm:PaleyExtreme3}, we invoke Lemma \ref{lem:LogInterpolationLebesgue} and the left-hand side embedding of Lemma \ref{lem:reiteration}\eqref{lem:reiteration1} to establish
	\begin{align}
		 (L^{1}(\T^d; X), L^{p_0}(\T^d; X))_{1, q;b} \nonumber \\
		 & \hspace{-2.5cm}=  (L^{1}(\T^d; X), (L^{1}(\T^d; X), L^{\infty}(\T^d; X))_{1/p_0',p_0})_{1,q;b} \nonumber\\
		 & \hspace{-2.5cm}\hookleftarrow (L^{1}(\T^d; X), L^{\infty}(\T^d; X))_{1/p_0',q;b + 1/\min\{p_0,q\}} \nonumber \\
		 & \hspace{-2.5cm} = L^{p_0,q}(\log L)^{b + 1/\min\{p_0,q\}}(\T^d;X). \label{thm:PaleyExtreme4}
	\end{align}
	As far as the target space, we use the right-hand side embedding of Lemma \ref{lem:reiteration}\eqref{lem:reiteration1} to get
	\begin{align}
		 ( \ell^{\infty}(\Z^d;X), \ell^{p_0'}(\Z^d;X))_{1, q;b} & \nonumber \\
		 & \hspace{-3.5cm}= (\ell^{\infty}(\Z^d;X), (\ell^{\infty}(\Z^d;X), \ell^1(\Z^d;X))_{1/p_0', p_0'})_{1,q;b} \nonumber \\
		 & \hspace{-3.5cm} \hookrightarrow (\ell^{\infty}(\Z^d;X), \ell^1(\Z^d;X))_{1/p_0', q; b+1/\max\{p_0',q\}} \nonumber \\
		 &  \hspace{-3.5cm} = \ell^{p_0',q}(\log \ell)^{b + 1/\max\{p_0',q\}}(\Z^d;X) . \label{thm:PaleyExtreme5}
	\end{align}
	Finally, it follows from \eqref{thm:PaleyExtreme3}, \eqref{thm:PaleyExtreme4} and \eqref{thm:PaleyExtreme5} that
	\begin{equation*}
		\F : L^{p_0,q}(\log L)^{b + 1/\min\{p_0,q\}}(\T^d;X) \to \ell^{p_0',q}(\log \ell)^{b + 1/\max\{p_0',q\}}(\Z^d;X) .
	\end{equation*}
\qed

Theorem \ref{thm:PaleyExtreme} can be improved under additional geometric assumptions on $X$. Indeed, the same method of proof of Theorem \ref{thm:PaleyExtreme} allows us to show the following
 \begin{theorem}\label{ThmFinalRem1}
	Let $X$ be a Banach space of Paley type $p_0 \in (1,2]$. Let $1 \leq q \leq \infty$ and $b < -1/q$. Then,
	\begin{equation}\label{thm:PaleyExtreme1*new}
		\F : L^{p_0, q}(\log L)^{b + 1/\min\{p_0, q\}} (\T^d;X) \to \ell^{p_0',q}(\log \ell)^{b + 1/\max\{p_0, q\}}(\Z^d;X).
	\end{equation}
	In particular, if $q > p_0$ then
		\begin{equation*}
	\wh{f}^\ast (n) \leq C n^{-1/p_0'} (1 + \log n)^{1/p_0 - 1/q} \|f\|_{L^{p_0,q}(\T^d;X)}, \quad n \in \N.
\end{equation*}
\end{theorem}

Note that $\ell^{p_0',q}(\log \ell)^{b + 1/\max\{p_0, q\}}(\Z^d;X) \hookrightarrow \ell^{p_0',q}(\log \ell)^{b + 1/\max\{p_0', q\}}(\Z^d;X)$ and $\ell^{p_0',q}(\log \ell)^{b + 1/\max\{p_0, q\}}(\Z^d;X) \neq \ell^{p_0',q}(\log \ell)^{b + 1/\max\{p_0', q\}}(\Z^d;X)$ if $q < p_0'$ and $p_0 \neq 2$. In particular, \eqref{thm:PaleyExtreme1*new} strengthens \eqref{thm:PaleyExtreme1new} for $X = \ell^{p_0}(\Z^d), \, 1 <p_0 < 2$ (see Example \ref{ex:Paley2}).

The same comment also applies to Remark \ref{RemarkBochSeq}. Namely, we have the following
 \begin{theorem}\label{ThmFinalRem2}
	Let $X$ be a Banach space of Paley type $p_0 \in (1,2]$. Let $1 \leq q \leq \infty$ and $b < -1/q$. Then,
		\begin{equation*}
		\F : \ell^{p_0, q}(\log \ell)^{b + 1/\min\{p_0, q\}} (\Z^d;X) \to L^{p_0',q}(\log L)^{b + 1/\max\{p_0, q\}}(\T^d;X).
	\end{equation*}
\end{theorem}

 \subsection{Other orthonormal systems}\label{SectionAdditional}

Both Fourier type (see Definition \ref{def:type}) and (Rademacher) type (see \eqref{eq:type}) are special cases of a more general notion of type introduced by Garc\'ia-Cuerva et al. \cite[Section 7]{GaKaKoTo98}. We start by recalling this notion.

Let $(S,  \Sigma, \mu)$ be a measure space, and let $\mu (S) = 1$. Let $\Phi = (\varphi_n)_{n \in \N}$ be an orthonormal system of scalar functions on $S$. Suppose further that $\Phi$ is \emph{uniformly bounded}, that is,
\begin{equation*}
	|\varphi_n (s)| \leq C \quad \text{for all} \quad n \in \N \quad \text{and} \quad s \in S.
\end{equation*}

For $f \in L^1(S;X)$ and $n \in \N$, the \emph{$n$-th Fourier coefficient of $f$} is given by
\begin{equation*}
	c_n(f) = \int_S f(s) \overline{\varphi_n(s)} \, d\mu(s).
\end{equation*}

 Let $p \in [1,2]$. We say that $X$ has \emph{$\Phi$-type $p$} if there exists a positive constant $C$ such that
 \begin{equation*}
 	\left( \int_S \left\|\sum_{k=1}^n \varphi_k(s) x_k \right\|_X^{p'} \, d \mu(s)\right)^{1/p'} \leq C \left(\sum_{k=1}^n \|x_k\|_X^p \right)^{1/p}
 \end{equation*}
 for all $x_1, \ldots, x_n \in X$, and $X$ has \emph{strong $\Phi-$cotype $p'$} if
 \begin{equation*}
 	\left(\sum_{n=1}^\infty \|c_n(f)\|_X^{p'} \right)^{1/p'} \leq C \left(\int_S \|f(s)\|_X^p \, d \mu(s) \right)^{1/p}.
 \end{equation*}
In general, there is no relation between these notions for different $\Phi$. Clearly, we have that $X$ has $\Phi$-type 1 and strong $\Phi$-cotype $\infty$. The space $X$ has \emph{non trivial $\Phi$-type} (respectively, \emph{non trivial strong $\Phi$-cotype}) if there exists $p \in (1,2]$ such that $X$ has $\Phi$-type $p$ (respectively, strong $\Phi$-cotype $p'$).

It turns out that the method of proof of all results given in this section can be extended, with obvious modifications, to arbitrary orthonormal systems. In particular, the counterparts of Theorems \ref{ThmZygmund} and \ref{ThmZygmundExp} for the notion of $\Phi$-type read as

\begin{theorem}
	Assume that $X$ has non trivial $\Phi$-type. Let $1 \leq q \leq \infty$ and $b < -1/q$. If $(c_n) \in \ell^{1,q}(\log \ell)^{b + 1}(\N;X)$ then $(c_n)$ is the sequence of Fourier coefficients of a function $f \in L^{\infty, q}(\log L)^b(S;X)$.
\end{theorem}

\begin{theorem}
	Assume that $X$ has non trivial $\Phi$-type.
 If
\begin{equation*}
	\|(c_n) \|_{\ell^{1,q}(\log \ell)^{b+1}(\N;X)} \leq \rho
\end{equation*}
for some $1 \leq q \leq \infty$ and $b + 1/q < 0$, then $(c_n)$ is the sequence of Fourier coefficients of a function $f$ such that
\begin{equation*}
	\int_{S} \emph{exp} \Big(a(\rho,b, q) \|f(s)\|_X^{-\frac{1}{b + \frac{1}{q}}} \Big) \, d \mu(s)  \leq A(\rho,b, q)
\end{equation*}
for some positive constants $a(\rho, b, q)$ and $A(\rho, b, q)$.
\end{theorem}

\begin{corollary}\label{Cor7.19}
	Assume that $X$ has non trivial $\Phi$-type. If
\begin{equation*}
	\sum_{n =1}^\infty \|c_n\|_X \left(\log^{+} \left(\frac{1}{\|c_n\|_X}\right) \right)^{-b} \leq \rho
\end{equation*}
for some $b > 0$, then $(c_n)$ is the sequence of Fourier coefficients of a function $f$ such that
\begin{equation*}
	\int_{S} \emph{exp} \Big(a(\rho,b) \|f(t)\|_X^{\frac{1}{b}} \Big) \, d \mu(s)  \leq A(\rho,b)
\end{equation*}
for some positive constants $a(\rho, b)$ and $A(\rho, b)$.
\end{corollary}

The corresponding assertions for strong $\Phi$-cotype also hold.

\begin{theorem}
	Assume that $X$ has non trivial strong $\Phi$-cotype. Then
		\begin{enumerate}[\upshape(i)]
	\item Let $1 \leq q \leq \infty$ and $b > -1/q$. If $f \in L^{1,q}(\log L)^{b+1}(S;X)$ then $(c_n(f)) \in \ell^{\infty, q}(\log \ell)^{b}(\N;X)$.
	\item  Let $1 \leq q < \infty$. If $f \in L^{1,q}(\log L)^{-1/q +1} (\log \log L)^{1}(S;X)$ then $(c_n(f)) \in \ell^{\infty, q}(\log \ell)^{-1/q}(\N;X)$.
	\end{enumerate}
\end{theorem}

\begin{theorem}
		Assume that $X$ has non trivial strong $\Phi$-cotype. If
\begin{equation*}
 	\|f\|_{L^{1,q}(\log L)^{b+1}(S;X)} \leq \rho
 \end{equation*}
for some $1 \leq q \leq \infty$ and $b  + 1/q > 0$, then
\begin{equation*}
	\sum_{n =1}^\infty \emph{exp} \Big(-a(\rho,b, q) \|c_n(f)\|_X^{-\frac{1}{b + \frac{1}{q}}} \Big) \leq A(\rho,b, q)
\end{equation*}
for some positive constants $a(\rho, b, q)$ and $A(\rho, b, q)$.
\end{theorem}

\begin{corollary}
	Assume that $X$ has non trivial strong $\Phi$-cotype. If
	\begin{equation*}
 	\int_{S} \|f(t)\|_X (\log^+ \|f(s)\|_X)^b \, d\mu(s) \leq \rho
 \end{equation*}
for some $b > 0$, then
\begin{equation*}
	\sum_{n =1}^\infty \emph{exp} \Big(-a(\rho,b) \|c_n(f)\|_X^{-\frac{1}{b}} \Big) \leq A(\rho,b)
\end{equation*}
for some positive constants $a(\rho, b)$ and $A(\rho, b)$.
\end{corollary}

Similarly, the vector-valued Bochkarev inequalities given in Theorem \ref{thm:PaleyExtreme} and Remark \ref{RemarkBochSeq} can be extended to arbitrary orthonormal systems.

\begin{theorem}\label{ThmFinalRem3}
	Let $X$ be a Banach space of strong $\Phi$-cotype $p_0' \in [2, \infty)$. Let $1 \leq q \leq \infty$ and $b < -1/q$. Then,
	\begin{equation*}
		\F : L^{p_0, q}(\log L)^{b + 1/\min\{p_0, q\}} (S;X) \to \ell^{p_0',q}(\log \ell)^{b + 1/\max\{p_0', q\}}(\N;X).
	\end{equation*}
		In particular, if $q > p_0$ then
	\begin{equation*}
 c_n(f)^\ast \leq C n^{-1/p_0'} (1 + \log n)^{1/p_0 - 1/\max\{p_0',q\}} \|f\|_{L^{p_0,q}(S;X)}, \quad n \in \N.
\end{equation*}
\end{theorem}

\begin{theorem}\label{ThmFinalRem4}
	Let $X$ be a Banach space of $\Phi$-type $p_0 \in (1,2]$. Let $1 \leq q \leq \infty$ and $b < -1/q$. Then,
	\begin{equation*}
		\F : \ell^{p_0, q}(\log \ell)^{b + 1/\min\{p_0, q\}} (\N;X) \to L^{p_0',q}(\log L)^{b + 1/\max\{p_0', q\}}(S;X).
	\end{equation*}
\end{theorem}

In analogy to Theorems \ref{ThmFinalRem1} and \ref{ThmFinalRem2}, one can improve the target spaces of the Fourier inequalities given in Theorems \ref{ThmFinalRem3} and \ref{ThmFinalRem4} under additional geometric assumptions on $X$. Namely, following \cite{GaKaKo01}, we say that $X$ has \emph{Paley $\Phi$-type $p \in (1,2]$} if there exists a positive constant $C$ such that
 \begin{equation*}
 	\|f\|_{L^{p',p}(S;X)} \leq C \left(\sum_{k=1}^n \|x_k\|_X^p \right)^{1/p}
 \end{equation*}
 for all $x_1, \ldots, x_n \in X$, and where $f (s) = \sum_{k=1}^n \varphi_k(s) x_k$. The space $X$ is of \emph{strong Paley $\Phi-$cotype $p'$} if
 \begin{equation*}
 	\|(c_n(f))\|_{\ell^{p',p}(\N;X)} \leq C \left(\int_S \|f(s)\|_X^p \, d \mu(s) \right)^{1/p}.
 \end{equation*}

 \begin{theorem}
	Let $X$ be a Banach space of strong Paley $\Phi$-cotype $p_0' \in [2, \infty)$. Let $1 \leq q \leq \infty$ and $b < -1/q$. Then,
	\begin{equation*}
		\F : L^{p_0, q}(\log L)^{b + 1/\min\{p_0, q\}} (S;X) \to \ell^{p_0',q}(\log \ell)^{b + 1/\max\{p_0, q\}}(\N;X).
	\end{equation*}
		In particular, if $q > p_0$ then
	\begin{equation*}
 c_n(f)^\ast \leq C n^{-1/p_0'} (1 + \log n)^{1/p_0 - 1/q} \|f\|_{L^{p_0,q}(S;X)}, \quad n \in \N.
\end{equation*}
\end{theorem}

\begin{theorem}
	Let $X$ be a Banach space of Paley $\Phi$-type $p_0 \in (1,2]$. Let $1 \leq q \leq \infty$ and $b < -1/q$. Then,
	\begin{equation*}
		\F : \ell^{p_0, q}(\log \ell)^{b + 1/\min\{p_0, q\}} (\N;X) \to L^{p_0',q}(\log L)^{b + 1/\max\{p_0, q\}}(S;X).
	\end{equation*}
\end{theorem}

\subsection{Limiting interpolation for the Fourier transform on $\R^d$\label{SectionAdditional2}}
Many of the results of this section can also be formulated for the Fourier transform on $\R^d$ with the help of the so-called generalized Lorentz-Zygmund spaces. For $\mathbb{A} = (\alpha_0, \alpha_\infty) \in \R^2$, we put
\begin{equation*}
	  \ell^{\mathbb{A}}(t) = \left\{\begin{array}{lcl}
                                (1 + |\log t|)^{\alpha_0} & \text{ for } & t \in (0,1], \\
                                             &              &              \\
                                (1 + |\log t|)^{\alpha_\infty} & \text{ for } & t \in
                                (1,\infty).
                                \end{array}
                        \right.
\end{equation*}
Let $(S, \Sigma, \mu)$ be a measure space. For $1 \leq p \leq \infty, 1 \leq q \leq \infty$ and $\mathbb{A} \in \R^2$, the \emph{generalized Lorentz-Zygmund space} $L^{p,q;\mathbb{A}}(S;X)$ is formed by all $\mu$-strongly measurable functions $f : S \to X$ for which
\begin{equation*}
	\|f\|_{L^{p,q;\mathbb{A}}(S;X)} = \Big(\int_0^{\mu(S)} (t^{1/p} \ell^{\mathbb{A}}(t)   f^\ast(t))^q \,\frac{dt}{t} \Big)^{1/q} < \infty
\end{equation*}
(with the usual modification if $q=\infty$). Obviously, if $\alpha_0 = \alpha_\infty = \alpha$ then $L^{p,q;\mathbb{A}}(S;X) = L^{p,q}(\log L)^{\alpha}(S;X)$ (see \eqref{DefLZ}). Note that the use of different logarithmic powers near 0 and near $\infty$ is only useful if $\mu(S) = \infty$. Otherwise, $L^{p,q;\mathbb{A}}(S;X) = L^{p,q}(\log L)^{\alpha_0}(S;X)$. For detail study of the spaces $L^{p,q;\mathbb{A}}(S;X)$ we refer to \cite{OpicPick}.

Using logarithmic weights $\ell^{\mathbb{A}}(t)$, one can introduce limiting interpolation spaces for general Banach couples, not necessarily ordered (see Section \ref{SectionLimInt}). Namely, if $0 \leq \theta \leq 1, \mathbb{A} = (\alpha_0, \alpha_\infty) \in \R^2$, and $1 \leq q \leq \infty$. The interpolation space $(X_0, X_1)_{\theta, q; \mathbb{A}}$ is the set of all $f \in X_0 + X_1$ for which
\begin{equation*}
	\|f\|_{(X_0, X_1)_{\theta, q; \mathbb{A}}} = \left(\int_0^\infty (t^{-\theta} \ell^{\mathbb{A}}(t) K(t,f))^q \frac{dt}{t} \right)^{1/q} < \infty.
\end{equation*}
See \cite{EvansOpic, EvansOpicPick}. In particular, working with ordered couples $X_1 \hookrightarrow X_0$ and under suitable choices of the logarithmic parameters, the spaces $(X_0, X_1)_{\theta, q; \mathbb{A}}$ with $\theta = 0, 1$ coincide with those given in \eqref{LogIntSpace*} (see \cite[Proposition 1]{EdmundsOpic}).

We shall use the following notation. If $\mathbb{A} = (\alpha_0, \alpha_\infty) \in \R^2$ and $\alpha \in \R$ then $\mathbb{A} + \alpha = \mathbb{A} + (\alpha, \alpha) = (\alpha_0 + \alpha, \alpha_\infty + \alpha)$.

Using the counterparts of the reiteration formulas given in Lemma \ref{lem:reiteration} for the interpolation method $(\theta, q; \mathbb{A})$ (see \cite{EvansOpic, EvansOpicPick}), one can easily adapt our methods applied in Sections \ref{ProofThmZygmund} and \ref{Section7.2} to show the corresponding results to Theorems \ref{ThmZygmund} and \ref{thm:PaleyExtreme} for the Fourier transform on $\R^d$. They read as follows.

 \begin{theorem}
	Let $X$ be a Banach space of Fourier type $p_0 \in (1,2]$. Let $1 \leq q < \infty$ and $\mathbb{A} = (\alpha_0, \alpha_\infty) \in \R^2$ be such that $\alpha_0 + 1/q < 0 < \alpha_\infty + 1/q$. Then,
	\begin{equation*}
		\F : L^{p_0, q;\mathbb{A} + 1/\min\{p_0, q\}} (\R^d;X) \to L^{p_0',q;\mathbb{A} + 1/\max\{p_0', q\}}(\R^d;X).
	\end{equation*}
\end{theorem}

\begin{theorem}
	Assume that $X$ has non trivial Fourier type. Let $1 \leq q < \infty$ and $\mathbb{A} = (\alpha_0, \alpha_\infty) \in \R^2$ be such that $\alpha_0 + 1/q < 0 <\alpha_\infty + 1/q$. Then,
	\begin{equation*}
		\F : L^{1,q;\mathbb{A} + 1}(\R^d;X) \to L^{\infty,q;\mathbb{A}}(\R^d;X).
	\end{equation*}
\end{theorem}

\bibliographystyle{alpha-sort}
\bibliography{BibliografieFourier}

\end{document}